\documentclass[11pt,letterpaper]{amsart}
\usepackage[top=1.15in,bottom=1.15in,left=1.15in,right=1.15in,marginpar=1in]{geometry}
\usepackage{subcaption}
\usepackage{graphicx}
\usepackage{amssymb}
\usepackage{epstopdf}
\usepackage[mathcal]{eucal}
\usepackage{mathtools}  
\usepackage{pinlabel}
\usepackage{hyperref}
\usepackage{enumitem}
\usepackage[all,cmtip]{xy}
\usepackage[textsize=tiny]{todonotes}

\usepackage{scalerel}

\makeatletter
\newcommand*{\bigcdot}{}
\DeclareRobustCommand*{\bigcdot}{%
  \mathbin{\mathpalette\bigcdot@{}}%
}
\newcommand*{\bigcdot@scalefactor}{.75}
\newcommand*{\bigcdot@widthfactor}{1.15}
\newcommand*{\bigcdot@}[2]{%
  \sbox0{$#1\vcenter{}$}
  \sbox2{$#1\cdot\m@th$}%
  \hbox to \bigcdot@widthfactor\wd2{%
    \hfil
    \raise\ht0\hbox{%
      \scalebox{\bigcdot@scalefactor}{%
        \lower\ht0\hbox{$#1\bullet\m@th$}%
      }%
    }%
    \hfil
  }%
}
\makeatother
\newcommand{\nc}{\newcommand}
\nc{\dmo}{\DeclareMathOperator}
\nc{\nt}{\newtheorem}

\nt{theorem}{Theorem}[section]
\nt{proposition}[theorem]{Proposition}
\nt{corollary}[theorem]{Corollary}
\nt{lemma}[theorem]{Lemma}
\nt*{tpdp}{Topological polynomial decision problem}

\let\ORIincludegraphics\includegraphics
\renewcommand{\includegraphics}[2][]{\ORIincludegraphics[scale=0.91666,#1]{#2}}

\nc{\C}{\mathcal{C}}

\dmo{\Mod}{Mod}
\dmo{\Teich}{Teich}
\dmo{\PMod}{PMod}
\dmo{\LMod}{LMod}
\dmo{\Homeo}{Homeo}
\dmo{\Aut}{Aut}
\dmo{\Sq}{Sq}
\dmo{\Cub}{Cub}
\dmo{\Rot}{Rot}

\nc{\Z}{\mathbb Z}
\nc{\Q}{\mathbb Q}
\nc{\N}{\mathcal N}
\nc{\M}{\mathcal M}
\nc{\R}{\mathbb R}
\nc{\A}{\mathcal A}
\nc{\Y}{\mathcal Y}
\nc{\T}{\mathcal T}

\nc{\p}[1]{\bigskip\noindent\emph{#1.}}

\title[Recognizing topological polynomials by lifting trees]{Recognizing topological polynomials \\ by lifting trees}
\author{James Belk}
\author{Justin Lanier}
\author{Dan Margalit}
\author{Rebecca R. Winarski}
\address{James Belk \\ Department of Mathematics \\ 
Cornell University \\
Ithaca, NY 14853}
\email{jmb226@cornell.edu}
\address{Justin Lanier \\ Department of Mathematics \\ University of Chicago \\ 5734 S University Ave \\ Chicago, IL 60637}
\email{jlanier@math.uchicago.edu}
\address{Dan Margalit \\ School of Mathematics\\ Georgia Institute of Technology \\ 686 Cherry St. \\ Atlanta, GA 30332}
\email{margalit@math.gatech.edu}
\address{Rebecca R. Winarski \\  Department of Mathematics and Computer Science\\ College of the Holy Cross\\
1 College Street
Worcester, MA 01610}
\email{rebecca.winarski@gmail.com}

\begin{document}
\maketitle

\vspace*{-2.5ex}

\begin{abstract}
We give a simple algorithm that determines whether a given post-critically finite topological polynomial is Thurston equivalent to a polynomial. If it is, the algorithm produces the Hubbard tree; otherwise, the algorithm produces the canonical obstruction.  Our approach is rooted in geometric group theory, using iteration on a simplicial complex of trees, and building on work of Nekrashevych.  As one application of our methods, we resolve the polynomial case of Pilgrim's finite global attractor conjecture.  We also give a new solution to Hubbard's twisted rabbit problem, and we state and solve several generalizations of Hubbard's problem where the number of post-critical points is arbitrarily large.
\end{abstract}


\section{Introduction}

Thurston's theorem in complex dynamics states that a post-critically finite map $S^2 \to S^2$ is equivalent to a rational map if and only if it does not have what is called a Thurston obstruction~\cite{DH}.  Thurston's work does not provide an effective algorithm for deciding whether such a map is equivalent to a rational map.  In this paper we give an efficient, geometric, algorithmic solution to this basic decision problem in the special case of post-critically finite topological polynomials.  Specifically, we address the following.

\medskip \noindent \emph{Topological polynomial decision problem.}
Given a post-critically finite topological polynomial, determine whether or not it is Thurston equivalent to a polynomial.  If it is, determine the polynomial.  If it is not, determine the canonical Thurston obstruction.

\medskip

In the case where a given post-critically finite topological polynomial is Thurston equivalent to a polynomial, the sense in which the algorithm determines the polynomial is as follows: the output of our algorithm is the isotopy class of the corresponding Hubbard tree, relative to the post-critical set.  The Hubbard tree is a combinatorial invariant that completely determines the corresponding polynomial. If one further wanted to know the coefficients of the polynomial, a numerical approximation algorithm such as the Hubbard--Schleicher spider algorithm could be applied \cite{spider}.

Our algorithm is based on our main theoretical result, Theorem~\ref{thm:main}.  For each $n$ we define a simplicial complex $\T_n$, whose vertices are homotopy classes of trees in the plane with $n$ marked points.  For a post-critically finite topological polynomial $f$  with $n$ post-critical points, we define a simplicial map $\lambda_f : \T_n \to \T_n$, called the lifting map.  Theorem~\ref{thm:main} states that $\lambda_f$ has a nucleus that is contained in the 2-neighborhood of the Hubbard tree when $f$ is unobstructed and that is contained in the 1-neighborhood of the set of trees that are compatible with the canonical obstruction when $f$ is obstructed.  The algorithm proceeds then by iteration of $\lambda_f$ on an arbitrary vertex of $\T_n$.  Our approach is inspired by similar constructions in geometric group theory, in particular related work of Nekrashevych \cite{nek_cactus}; see the discussion after the statement of Theorem~\ref{thm:main}.

In addition to solving the topological polynomial decision problem, we apply our methods to address two of the guiding problems in the field.  As a corollary of Theorem~\ref{thm:main}, we resolve in the affirmative Pilgrim's finite global attractor conjecture for the case of polynomials.  As a second application of our methods, we give a new, self-contained solution to Hubbard's twisted rabbit problem, which was originally solved by Bartholdi--Nekrashevych using iterated monodromy groups.  Then we state and solve a generalization of the twisted rabbit problem where the number of post-critical points is arbitrarily large. Finally, we state and solve another family of twisted polynomial problems where the number of post-critical points is arbitrarily large and where obstructed maps arise in the answer.

The remainder of the introduction is structured as follows.  In Section~\ref{sec:statement} we give the relevant background, state Theorem~\ref{thm:main}, give an overview of the proof, and give several examples of nuclei.  In Section~\ref{sec:algorithms} we explain in detail how Theorem~\ref{thm:main} gives the iterative algorithm for the topological polynomial decision problem.  In Section~\ref{sec:apps} we discuss our applications to the finite global attractor conjecture and to twisted polynomial problems.  In Section~\ref{sec:comp} we compare our work with prior results related to the topological polynomial decision problem.  We conclude by giving an outline of the body of the paper.  


\subsection{Statement of the main result}
\label{sec:statement}

Before stating our main result, Theorem~\ref{thm:main}, we introduce some background about topological polynomials.  In particular we discuss two objects that can be associated to a topological polynomial, one in the case when it is equivalent to a polynomial and one in the case when it is not.  These objects are called Hubbard trees and canonical obstructions, respectively.

\p{Topological polynomials} A topological polynomial is an orientation-preserving branched cover $f: \R^2 \to \R^2$ with degree greater than 1 and finitely many critical points.  First examples of topological polynomials are polynomials in one variable defined over $\mathbb{C}$ with degree greater than 1. The post-critical set $P_f$ of a topological polynomial $f$ is the set of forward orbits of the set of its critical points (critical points are not necessarily post-critical).  We say that $f$ is post-critically finite if $P_f$ is finite.  

Two post-critically finite topological polynomials $f$ and $g$ are said to be Thurston equivalent if there are orientation-preserving homeomorphisms $\phi_0,\phi_1 : (\R^2,P_f) \to (\R^2,P_g)$ that are isotopic (relative to $P_f$) and make the following diagram commute:
\[
\xymatrix
   { 
   (\R^2,P_f)  \ar[d]_f \ar[r]^{\phi_1} & (\R^2,P_g) \ar[d]^g  & \\
   (\R^2,P_f) \ar[r]^{\phi_0} & (\R^2,P_g) & \\
   }
\]
Thurston rigidity states that two Thurston equivalent  polynomials are conjugate by an affine map \cite{DH}, and so if a topological polynomial is Thurston equivalent to a polynomial, then this polynomial is unique up to affine equivalence.  

\p{Hubbard trees} Every post-critically finite polynomial $f$ has an associated Hubbard tree.  Several non-equivalent definitions of Hubbard trees appear in the literature.  In this paper, the Hubbard tree for a post-critically finite polynomial $f$ is the subset of $\mathbb{C}$ given by the union of all regulated arcs in the filled Julia set for~$f$ between pairs of points in the post-critical set~$P_f$.   This tree is invariant in the sense that $f(H_f)\subseteq H_f$.  Such trees were first described by Douady--Hubbard \cite{DH1,DH2}. 

If a topological polynomial $f$ is equivalent to a polynomial, then it has an associated topological Hubbard tree that is invariant up to isotopy: given a Thurston equivalence from $f$ to a polynomial $p$, we may pull back the Hubbard tree for $p$ to obtain the topological Hubbard tree for $f$.  A (topological) Hubbard tree, together with its preimage and the associated mapping of trees, is a complete invariant for the Thurston equivalence class of an unobstructed topological polynomial (see the Alexander method in Section~\ref{sec:alex}).

\p{Obstructions} A Levy cycle is a nonempty collection of disjoint essential simple closed curves $\{c_0,\dots,c_{k-1}\}$ in $\R^2 \setminus P_f$ so that for each $i$, at least one component $\tilde c_{i-1}$ of the preimage $f^{-1}(c_{i})$ is homotopic to $c_{i-1}$ and so that the restriction $\tilde c_{i-1} \to c_i$ of $f$ is a degree 1 map, with indices taken modulo~$k$.

A post-critically finite topological polynomial is equivalent to a polynomial if and only if it does not have a Levy cycle; see Hubbard's book  \cite[Theorem 10.3.8]{hubbard} and his paper with Bielefeld and Fisher \cite[Proposition 5.5]{BFH}.  This criterion is a specialization of Thurston's theorem~\cite{DH},  which treats the more general case of a post-critically finite branched cover of the sphere.

If a post-critically finite topological polynomial has a Levy cycle, then we say that the topological polynomial is obstructed.  Similarly, a topological polynomial is unobstructed if it is equivalent to a polynomial.

An obstructed post-critically finite topological polynomial may have infinitely many different Levy cycles (see Section~\ref{sec:canonical form} for an example).  Pilgrim proved that an obstructed post-critically finite topological polynomial $f$ has associated to it a collection of curves called the canonical obstruction~$\Gamma_f$, which is a specific union of Levy cycles together with all (essential) iterated preimages of these cycles; see Section~\ref{sec:obs} for more details. 

\p{Statement of the main theorem} In Section~\ref{sec:complex} we define for each $n$ a locally finite simplicial complex  $\T_n$.  The vertices of $\T_n$ are isotopy classes of trees, and there is a natural metric on the set of vertices given by the path metric in the 1-skeleton.  For each topological polynomial $f$ with $|P_f|=n$ we define a simplicial map
\[
\lambda_f : \T_n \to \T_n,
\]
that we call the lifting map.  The map $\lambda_f : \T_n \to \T_n$ is a combinatorial analogue of Thurston's pullback map on Teichm\"uller space; see Section~\ref{sec:lifting} for a comparison.  

For an unobstructed post-critically finite topological polynomial, the topological Hubbard tree is a fixed vertex $H_f$ for $\lambda_f$ in $\T_n$; we refer to $H_f$ as the Hubbard vertex for $f$.  In the obstructed case, there is a subset of the set of vertices of $\T_n$ that we call the Levy set $L_f$, which encodes the canonical obstruction $\Gamma_f$; specifically, $L_f$ is the set of vertices of $\T_n$ with the property that $\Gamma_f$ is the boundary of a neighborhood of a subforest of the corresponding tree.  

We say that a subcomplex $N$ of $\T_n$ is a nucleus for $\lambda_f$ if for every vertex $T$ of $\T_n$, the sequence of vertices $\lambda_f^k(T)$ lies in $N$ for all $k$ sufficiently large. The following is our main theorem.  It is the theoretical underpinning of our algorithm to solve the decision problem stated above.

\begin{theorem}
\label{thm:main}
Let $f$ be a post-critically finite topological polynomial with $|P_f|=n$, and let $\lambda_f : \T_n \to \T_n$ be the lifting map.
\begin{enumerate}
    \item If $f$ is unobstructed, then the 2-neighborhood of $H_f$ is a nucleus for $\lambda_f$.\smallskip
    \item If $f$ is obstructed, then the 1-neighborhood of $L_f$ is a nucleus for $\lambda_f$.
\end{enumerate}
\end{theorem}

In the case where $f$ is unobstructed, Theorem~\ref{thm:main}(1) implies that $\lambda_f$ has a finite nucleus.  Even more, there is a unique minimal nucleus, contained in the 2-neighborhood of $H_f$, consisting of all vertices that are periodic under $\lambda_f$.  

In 2014, Nekrashevych \cite[Section 7.6]{nek_cactus} defined a polysimplicial complex $\widetilde D_n$ that is closely related to our simplicial complex $\T_n$.  He also defined for any topological polynomial with $n$ post-critical points an associated map $\Phi : \widetilde D_n \to \widetilde D_n$, which is analogous to our lifting map $\lambda_f$.  There is a straightforward argument to show that $\Phi$ is contracting in the case that $f$ is hyperbolic, hence giving a version of Theorem~\ref{thm:main}(1) for the special case when $f$ is a hyperbolic polynomial.   Here, a topological polynomial is hyperbolic if every critical point is attracted to a cycle containing a critical point; there are many topological polynomials that are not hyperbolic, including all obstructed topological polynomials and, for example, the polynomial $z^2+i$ discussed below.  See the paragraph ``Comparisons to prior works'' below for further discussion of Nekrashevych's work.

\p{Examples of nuclei: the rabbit, the co-rabbit, and the airplane polynomials} Up to Thurston equivalence, there are exactly three quadratic polynomials where the critical point is periodic with period 3.  These polynomials are called the rabbit, co-rabbit, and airplane polynomials (the names come from the shapes of their Julia sets); we denote them $R(z)$, $C(z)$, and~$A(z)$.  They are all of the form $z^2+c$, where $c$ is a nonzero root of the quartic polymomial $(c^2+c)^2+c$ (this is exactly the condition that 0 is 3-periodic).  The values of $c$ for $R$, $C$, and $A$ are approximately $-.12+.74i$, $-.12-.74i$, and $-1.75$, respectively.

\begin{figure}
\centering
\raisebox{-0.47\height}{\includegraphics{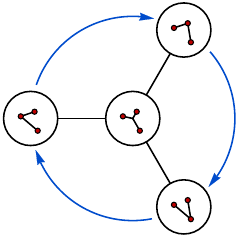}}\qquad\qquad\quad
\raisebox{-0.47\height}{\includegraphics{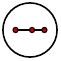}}
\quad\qquad\qquad
\raisebox{-0.47\height}{\includegraphics{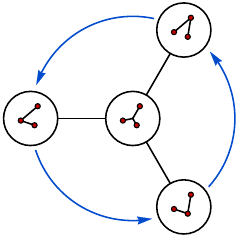}}
\caption{The minimal nuclei in $\T_3$ for $R$, $A$, and $C$}
\label{fig:nuclei}
\end{figure}

The minimal nuclei for these three polynomials are shown in Figure~\ref{fig:nuclei}; in the cases of the rabbit and co-rabbit polynomials, the central vertex is invariant and the other three vertices are cyclically permuted in the direction of the arrows under lifting.  

By Theorem~\ref{thm:main}, we can verify the purported nuclei by inspecting the action of the lifting map on the 2-neighborhood of the Hubbard vertex in each case.  Figure~\ref{fig:nuclei3} illustrates this action in the case of the rabbit polynomial.  

For the rabbit and co-rabbit polynomials, the minimal nucleus is equal to the 1-neighborhood of the Hubbard vertex and for the airplane polynomial the nucleus is the Hubbard vertex itself; in particular, in these cases the minimal nucleus is strictly smaller than the 2-neighborhood of the Hubbard vertex.  On the other hand, in Section~\ref{sec:alex} we give an example of a polynomial $f$ whose minimal nucleus is not contained in the 1-neighborhood of the Hubbard vertex.  Our proof of Theorem~\ref{thm:main} can be refined to give more precise information about the size of the minimal nucleus; see Section~\ref{sec:pfobs} for a discussion.

\begin{figure}[b]
\centering
\includegraphics{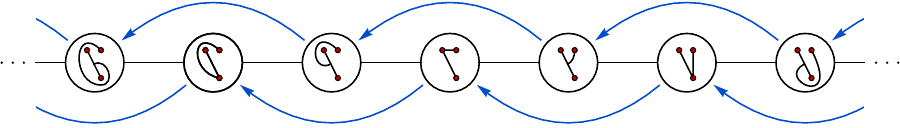}
\caption{A segment of a nucleus in $\T_3$ for $D_a^{-1}I$}
\label{fig:nuclei2}
\end{figure}

\p{An example of a nucleus for an obstructed map: twisted $z^2+i$} We now give an example of a nucleus for a certain obstructed topological polynomial that we will study in detail in Section~\ref{sec:z2+i}.  Consider the polynomial $I(z) = z^2+i$.  Its three post-critical points $i$, $-1+i$, and $-i$ are shown inside each of the circles in Figure~\ref{fig:nuclei2}.    Let $a$ be the curve in $(\R^2,P_I)$ given by the boundary of a neighborhood of the straight arc connecting $-i$ to $i$ (see also Figure~\ref{fig:I3curves} for a picture of $a$).  Then let $D_a$ be the left-handed Dehn twist about $a$.   The composition $D_a^{-1}I$ is a topological polynomial with the same post-critical set as $I$.  

The map $D_a^{-1}I$ is obstructed; the curve $b$ corresponding to the straight arc connecting $-i$ to $-1+i$ is a Levy cycle.  In $(\R^2,P_I)$ the only multicurves are single curves, and so we further deduce that $b$ is the canonical obstruction.  The Levy set $L_{D_a^{-1}I}$ thus consists of all trees compatible with $b$; these are exactly the trees shown in Figure~\ref{fig:nuclei2}.  The subgraph of $\T_3$ spanned by this Levy set is homeomorphic to $\R$; it corresponds to a ``horocycle'' in Figure~\ref{fig:TreeComplex3}.

One nucleus for $D_a^{-1}I$ is exactly this Levy set $L_{D_a^{-1}I}$, and the action of $\lambda_{D_a^{-1}I}$ on this nucleus is a translation by two ``clicks'' to the left.  As in the case of the rabbit, co-rabbit, and airplane polynomials, we can verify that this is a nucleus by considering the action of the lifting map on the 1-neighborhood of the Levy set (since $D_b$ commutes with $\lambda_{D_a^{-1}I}$, this reduces to a finite check); see Section~\ref{sec:pfobs} for another argument.

To check that the action on the lifting map on the nucleus is the translation indicated in Figure~\ref{fig:nuclei2}, we can use the description of the canonical form for $D_a^{-1}I$, described at the end of Section~\ref{sec:z2+i}.

\p{Overview of the proof} We give here a summary of the main ideas in the proof of Theorem~\ref{thm:main}.  Suppose first that $f$ is unobstructed.  Since $\T_n$ is locally finite, since $\lambda_f$ is simplicial, and since $\lambda_f$ fixes $H_f$, it follows that every vertex of $\T_n$ is pre-periodic under $\lambda_f$.  In other words, up to passing to a power of $f$ (which does not change $H_f$), every vertex of $\T_n$ is mapped to an invariant tree for $f$ by some iterate of $\lambda_f$.  Poirier gives combinatorial conditions that determine whether or not a tree is the Hubbard tree for a given polynomial \cite{poirier}; we use a version of his conditions to show that an invariant tree for $f$ has distance at most 2 from $H_f$.

In order to prove our version of Poirier's conditions, we describe a basic tool for specifying branched covers that we call the Alexander method (Proposition~\ref{prop:alexander}).  Roughly the Alexander method states that a topological polynomial $f$ is completely determined (up to homotopy $\mathrm{rel}\;P_f$) by its action on any single tree in $(\R^2,P_f)$.  It is a natural analogue in the context of branched covers of the Alexander method from the theory of mapping class groups \cite[Proposition 2.8]{primer}, which loosely states that a mapping class is determined by its action on a finite set of curves.  Versions of the Alexander method for topological polynomials have appeared in the literature, e.g. in the work of Bielefeld--Fisher--Hubbard \cite[Theorem 7.8]{BFH}

In the case that $f$ is obstructed, the proof of Theorem~\ref{thm:main} follows a similar outline.  One new ingredient is that we consider an augmented complex of trees $\hat \T_n$ to which $\lambda_f$ partially extends, so that $\lambda_f$ has a fixed vertex $H_f$, called the Hubbard vertex, as in the unobstructed case.  The vertices of $\hat \T_n$ are generalizations of trees called bubble trees, and $\hat \T_n$ contains $\T_n$ as a subcomplex.  In order to define the Hubbard vertex and to show that it is fixed by $f$, we apply Pilgrim's theory of canonical obstructions \cite{pilgrim03}, and in particular Selinger's topological characterization of canonical obstructions \cite{selinger13}.

One interpretation of Theorem~\ref{thm:main} is that in both the obstructed and unobstructed cases, the 2\mbox{-}neighborhood of $H_f$ contains a nucleus for the (analogous) lifting map $\lambda_f : \hat \T_n \dasharrow \hat \T_n$.  While this version of our theorem unifies the two cases, it does not immediately give an algorithm, since $\hat \T_n$ is not locally finite at the vertices of $\hat \T_n$ that do not lie in~$\T_n$.

We may think of the fact that $\lambda_f$ is simplicial as a sort of contraction property.  More accurately, it is a non-expansion property: the distance between any two vertices may not increase under a simplicial map.  Our proof of Theorem~\ref{thm:main} does not show directly that $\lambda_f$ is globally contracting; this only comes out as a consequence. It would be interesting to understand the global rate of contraction of~$\lambda_f$.  Nekrashevych studies a cell complex that is closely related to our $\T_n$ and shows that a hyperbolic polynomial acts on it with exponential contraction \cite[Section~7.6]{nek_cactus}.  We suspect that a similar argument would show that a hyperbolic polynomial gives exponential contraction of $\T_n$; however, the non-hyperbolic case is more mysterious.

\begin{figure}
\centering
\includegraphics[scale=.87]{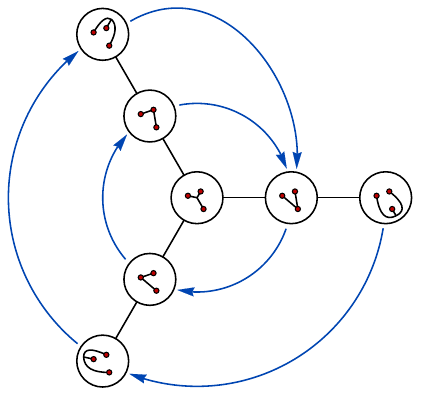}
\caption{The 2-neighborhood of the Hubbard vertex for the rabbit polynomial}
\label{fig:nuclei3}
\end{figure}

Figure~\ref{fig:nuclei3} illustrates the contraction property for the rabbit polynomial.  Here the 2-neighborhood of the Hubbard vertex maps into the 1-neighborhood after 3 iterates of the lifting map.  It follows that any vertex of distance $d \geq 2$ from the Hubbard vertex gets mapped to another vertex of distance at most $d-1$ from the Hubbard vertex after 3 iterates.  Therefore, upon iteration, every vertex of distance $d \geq 1$ eventually reaches the 1-neighborhood, which is a nucleus in this case.  

\p{The monoid of topological polynomials} Mapping class groups and braid groups are central objects of study in geometric group theory.  We describe here a monoid $\mathcal{TP}_n$ that naturally generalizes the braid group and that comes with an action on $\T_n$.  Our work suggests a more in-depth study of the monoid $\mathcal{TP}_n$ using its action on $\T_n$, in a manner analogous to the ways in which braid groups and mapping class groups have been studied using their actions on complexes of curves.  

We now define the monoid $\mathcal{TP}_n$.  Given a finite set $P\subset\mathbb{R}^2$ of marked points, let $\mathrm{BrMod}(\mathbb{R}^2,P)$ be the semigroup of all isotopy classes of topological polynomials whose post-critical set is contained in $P$. By an isotopy of topological polynomials we mean a homotopy that preserves the number and local degrees of any unmarked critical points throughout, or equivalently a homotopy obtained by pre- and post-composing by homeomorphisms that are isotopic to the identity relative to $P$. If we additionally include isotopy classes orientation-preserving homeomorphisms of $(\mathbb{R}^2,P)$ as ``topological polynomials’’ of degree one, the resulting algebraic structure is a monoid $\mathcal{TP}(\mathbb{R}^2,P)$, where the invertible elements are precisely the elements of the mapping class group $\Mod(\mathbb{R}^2,P)$. The isomorphism type of $\mathcal{TP}(\R^2,P)$ depends only on the cardinality of $P$, so we denote this monoid by $\mathcal{TP}_n$ for $|P|=n$.

The lifting maps $\lambda_f$ define a simplicial action of the monoid $\mathcal{TP}_n$ on the tree complex~$\T_n$.  This action restricts to the natural action of $\mathrm{Mod}_n$ on~$\T_n$. Note that two elements of $\mathcal{TP}_n$ are Thurston equivalent if and only if they are conjugate by an element of $\Mod_n$, so our algorithm can be viewed as a solution to the conjugacy problem in the monoid~$\mathcal{TP}_n$.

We learned about this viewpoint from Kevin Pilgrim, who has drawn many connections between the theory of $\mathcal{TP}_n$ and the theory of mapping class groups \cite{pilgrim_notices}.  As he suggests, the monoid $\mathcal{TP}_n$ has been underemphasized in the field, in favor of the set of Thurston equivalence classes.


\subsection{The tree lifting algorithm}
\label{sec:algorithms}

In this section we explain how to apply Theorem~\ref{thm:main} to give an algorithm for the topological polynomial decision problem.  We call our algorithm the tree lifting algorithm.  We begin by explaining a simplified version of the algorithm, which determines whether or not a topological polynomial is equivalent to a polynomial and outputs either a topological Hubbard tree or a Levy cycle.  Afterward, we explain how to modify the algorithm so that it outputs either a topological Hubbard tree or the canonical obstruction.

Suppose we are given a post-critically finite topological polynomial $f$ with $|P_f|=n$.  The steps of the simplified algorithm are as follows.
\begin{enumerate}
    \item Choose some vertex $T$ of $\T_n$.\smallskip
    \item Check if any element of the 2-neighborhood of $T$ is the topological Hubbard tree for~$f$ by checking if it is invariant and if it satisfies Poirier's conditions.  If it is the topological Hubbard tree, the algorithm outputs this tree and terminates.\smallskip
    \item Check if any tree in the 1-neighborhood of $T$ has a sub-tree whose boundary is a curve of a Levy cycle.  If so, the algorithm outputs that Levy cycle and the algorithm terminates.\smallskip
    \item Replace $T$ with $\lambda_f(T)$ and return to Step 2.
\end{enumerate}
Because of the local finiteness of $\T_n$, Steps 2 and 3 are finite checks.  By Theorem~\ref{thm:main} the algorithm terminates.  

We now explain how to modify the algorithm so that it outputs the canonical obstruction in the obstructed case.  Any topological polynomial $f$ has a Hubbard vertex $H_f$ in the augmented complex $\hat \T_n$ which is invariant under lifting, and we show in Proposition~\ref{prop:PoirierConditionsBubbleTree} how to generalize Poirier's conditions to give a recognition algorithm for~$H_f$.  When $f$ is obstructed, the Hubbard vertex $H_f$ includes the canonical obstruction as part of its definition, and it lies in the 1-neighborhood in $\hat\T_n$ of any vertex in the Levy set~$L_f$.  Thus we can use the following algorithm.
\begin{enumerate}
    \item Choose some vertex $T$ of $\T_n$.\smallskip
    \item Check if any element of the restricted 2-neighborhood of $T$ is the Hubbard vertex for~$f$ by checking if it is invariant and if it satisfies the conditions of Proposition~\ref{prop:PoirierConditionsBubbleTree}.  If it is the Hubbard vertex, the algorithm outputs the Hubbard tree or canonical obstruction and terminates.\smallskip
    \item Replace $T$ with $\lambda_f(T)$ and return to Step 2.
\end{enumerate}
Here the restricted 2-neighborhood is the 1-neighborhood in $\hat\T_n$ of the 1-neighborhood of $T$ in $\T_n$.  This restricted 2-neighborhood is always finite, so step~2 is a finite check.

We emphasize that the tree lifting algorithm as stated is only one way to convert our Theorem~\ref{thm:main} into an algorithm; there are many improvements one could make by examining the inner workings of the proof of Theorem~\ref{thm:main}.  For example, in Section~\ref{sec:pfobs} we give two successive refinements of Theorem~\ref{thm:main} that result in improvements in the tree lifting algorithm.


\subsection{Applications}
\label{sec:apps}

We explain here two applications of our methods, namely, to Pilgrim's finite global attractor conjecture and to Hubbard's twisted rabbit problem.

\p{Application to the finite global attractor conjecture} As above, a topological polynomial is obstructed if and only if it has a Levy cycle.  More generally, Thurston proved that a post-critically finite branched cover of the sphere is obstructed if and only if it has (what is now called) a Thurston obstruction, a multicurve that is invariant under pullback (and also satisfies an additional combinatorial property).  It is therefore a fundamental problem to understand the behavior of the set of multicurves under pullback.  Pilgrim's global attractor conjecture addresses one of the most basic aspects of this problem.

Let $f$ be a branched cover of the sphere with finite post-critical set $P$.  Let $\M_f$ denote the set of isotopy classes of multicurves in $S^2 \setminus P$; here a multicurve is a set of pairwise disjoint, pairwise non-homotopic essential simple closed curves.  We emphasize that $\M_f$ includes the empty multicurve.  Similar to the lifting map $\lambda_f : \T_n \to \T_n$, there is a lifting map $\lambda_f : \M_f \to \M_f$ (in the lift we discard inessential components and all but one component from each parallel family).  Pilgrim's conjecture is that if $f$ is (Thurston equivalent to) a rational map and is not a flexible Latt\`es example, then $\lambda_f : \M_f \to \M_f$ has a finite nucleus, that is, a finite subset to which all  elements of $\M_f$ are attracted (Pilgrim refers to a nucleus as a global attractor);  this conjecture appeared in a lecture by Pilgrim \cite{talk}, and later in a paper by Lodge \cite[Section 6]{Lodge}.  

Let $f$ be a polynomial whose post-critical set $P$ consists of $n$ points.  For a vertex $T$ of $\T_n$, let $\M(T)$ be the set of $M \in \M_f$ so that $M$ is the boundary of a neighborhood of a subforest of $T$.  Let $\{T_1,\dots,T_N\}$ be the vertices of any finite nucleus for~$f$ in $\T_n$ (it follows from Theorem~\ref{thm:main}(1) that this finite set exists).  For any vertex $T$ we have $\lambda_f(\M(T)) \subseteq \M(\lambda_f(T))$.  Thus, for $k$ large, the set $\M(\lambda_f^k(T))$ is contained in the union of the $\M(T_i)$.  Each such $\M(T_i)$ is finite, and moreover is algorithmically computable, via the tree lifting algorithm.  We thus have the following corollary of the first statement of Theorem~\ref{thm:main}, which resolves Pilgrim's conjecture in the case of polynomials.

\begin{corollary}
Let $f$ be a post-critically finite polynomial.  Then the lifting relation $\lambda_f$ on $\M_f$ has a finite nucleus.  More specifically, the nucleus is contained in the union of the $\M(T_i)$, which can be computed via the tree lifting algorithm.
\end{corollary}

Pilgrim has previously verified the finite global attractor conjecture in two cases: (1) where $f$ is a post-critically finite branched cover over the sphere whose associated virtual endomorphism on mapping class groups is contracting \cite[Theorem 1.4]{pilgrim12}, and (2) where $f$ is a quadratic polynomial with periodic critical point \cite[Corollary 7.2]{pilgrim12}.  He and Lodge also verified the conjecture for three specific quadratic polynomials \cite[Theorems 1.6, 1.7, 1.8]{pilgrim12}.  Additionally, Kelsey and Lodge verified the conjecture for all quadratic non-Latt\`es maps with four post-critical points.  The work of Nekrashevych discussed after the statement of Theorem~\ref{thm:main} implies the conjecture for the case of hyperbolic polynomials.  Finally, Hlushchanka proved the conjecture for critically fixed rational maps \cite{mh}.  

\p{Applications to twisted polynomial problems} In the early 1980s, Hubbard posed the so-called twisted rabbit problem, in part to emphasize how little was understood about the topological polynomial decision problem. If we post-compose, say, the rabbit polynomial $R$ with a homeomorphism $h$ of $\R^2$ fixing $P_R$ pointwise, we obtain a new topological polynomial $h R$ with the same dynamics on $P_R$ (throughout, we suppress the symbol for composition of a topological polynomial with a homeomorphism).  Such a topological polynomial cannot have a Levy cycle (by the Berstein--Levy theorem \cite{hubbard}), and so $h R$ is Thurston equivalent to $R$, $C$, or $A$.  Let $D_x$ be the (left-handed) Dehn twist about the curve $x$ in Figure~\ref{fig:generators} and let $m \in \Z$.  Hubbard's problem is: \emph{determine the Thurston equivalence class of $D_x^m R$ as a function of $m$.  In other words, determine the corresponding function $\Z \to \{R,C,A\}$.}

In 2006, Bartholdi--Nekrashevych solved the twisted rabbit problem \cite{BaNe}.  Their approach is to associate an algebraic object, called an iterated monodromy group, to a topological polynomial, and to show that this iterated monodromy group has a nucleus (similarly to how our lifting maps have nuclei).  The nucleus is a finite state automaton (in particular, it is a finite amount of data) that completely describes the Thurston equivalence class of the topological polynomial.  The nucleus is computable, and so this method solves the recognition problem for topological polynomials.  In particular, it solves the twisted rabbit problem.  They also give an explicit formula in terms of $m$ for whether $D_x^m R$ is equivalent to $R$, $C$, or $A$ (see Section~\ref{sec:rabbit}).  In their paper, Bartholdi--Nekrashevych also apply their methods to several variations of the twisted rabbit problem, by changing the original polynomial and/or the twisting homeomorphism.  In all of their examples, the size of the post-critical set is 3.

In Section~\ref{sec:onear} we apply our tree lifting algorithm to give a new solution to Hubbard's twisted rabbit problem.  In place of iterated monodromy groups we use the Alexander method, mentioned above.

In Section~\ref{sec:manyear} we give a generalization of the twisted rabbit problem to the case where there are $n$ post-critical points.  One feature of our method is that the generalization to $n$ post-critical points is readily apparent from the picture for the case $n=4$.  

Finally, in Sections~\ref{sec:z2+i} and~\ref{sec:genz2+i} we consider twistings of the polynomial $I(z) = z^2+i$ and of certain generalizations $I_n(z)$ that have $n$ post-critical points.  The critical point of each $I_n(z)$ has pre-period $n-2$ and period 2.  As such, the Berstein--Levy theorem does not apply, and in fact there are obstructed twistings of $I_n(z)$ for each $n$.  Bartholdi--Nekrashevych already gave an algorithm for determining the Thurston equivalence class of every twisting of $I=I_3$ by a pure mapping class \cite[Section 6]{BaNe}.  As they show, the result can be $I$, the polynomial $\bar I(z) =z^2-i$, or one of infinitely many distinct obstructed maps (which they completely catalog).  

In order to describe the answers to the twisted $z^2+i$ problem and its generalizations, we introduce in Section~\ref{sec:canonical form} a normal form for obstructed topological polynomials that we call the canonical form.  The canonical form is analogous to the Nielsen--Thurston normal form in the theory of mapping class groups.  The canonical form is a complete topological description of a given map into canonical pieces.  In particular, it carries more information than the collection of first-return maps.  In the case of mapping class groups, the analogous ``canonical form'' is not canonical (see the discussion in Section~\ref{sec:z2+i}); and so the fact that the pieces here are canonical is a novel feature.  


\subsection{Comparisons to prior works}
\label{sec:comp}

There have been many works on the decision problem for post-critically finite topological polynomials, and more generally, for post-critically finite branched covers of the sphere.

As discussed after the statement of Theorem~\ref{thm:main}, Nekrashevych defined for a topological polynomial $f$ with $n$ post-critical points a cell complex $\widetilde D_n$ and a map $\Phi : \widetilde D_n \to \widetilde D_n$ that is closely related to our tree lifting map.  The points of $\widetilde D_n$ are called metric cactus diagrams.  Nekrashevych mentions that $\Phi$ is contracting when $f$ is hyperbolic \cite[Proof of Theorem 7.2]{nek_csa}; the details are omitted, but it is a straightforward argument \cite{nekper}.  As such, there is a finite nucleus for $f$ in $\widetilde D_n$, and so this gives an algorithm for the recognition of topological polynomials in the special case of hyperbolic polynomials.  This approach gives finer information than Theorem~\ref{thm:main} in that the contraction on $\widetilde D_n$ is exponential.  However, it is unclear if this method can be extended to all topological polynomials.  
 
Our work is also closely related to the work of Bartholdi--Nekrashevych described above \cite{BaNe}.  One might hope that there is a way to translate between the minimal nuclei for our tree lifting map and the nuclei for their iterated monodromy groups.  A na\"ive guess would be that the elements of the fundamental group appearing in their nuclei are the ones represented by loops intersecting the topological Hubbard tree in at most one point.  For the rabbit, co-rabbit, and airplane polynomials, this is indeed the case, once we pass from $\pi_1(\R^2 \setminus P)$ to $\pi_1(S^2 \setminus P)$ (by adding a point at infinity).  It would be interesting to know if this correspondence holds in general.  In addition, the 

Hubbard--Schleicher \cite{spider} describe the spider algorithm, which uses an iterated lifting procedure on tuples of points in $\mathbb{C}$, and isotopy classes of arcs from $\infty$ to the these points, in order to find the coefficients of a post-critically finite unicritical polynomial from a given combinatorial description. This is complementary to our tree-lifting algorithm, which starts with a topological description of a map and obtains a combinatorial description.  Given an unobstructed post-critically finite unicritical topological polynomial, one can use our algorithm to determine the combinatorics of the Hubbard tree for the polynomial in its Thurston class and then use the spider algorithm to find the coefficients of this polynomial.

D.\ Thurston studies the case of post-critically finite branched covers of the sphere where each cycle of post-critical points contains a critical point \cite{thurston_positive}.  He gives a positive characterization for such a map to be equivalent to a rational map.  His criterion involves the existence of an elastic graph that stretches under iteration of the corresponding lifting map.  This result should be viewed as complementary to W.\ Thurston's theorem discussed earlier.

Bartholdi--Dudko \cite{BD} have written a series of papers that prove the decidability of the Thurston equivalence of pairs of post-critically finite branched covers of the sphere.  They also give an algorithm to determine whether an unobstructed branched cover $S^2\rightarrow S^2$ is rational.  Bartholdi--Dudko describe branched coverings
of the sphere in terms of group-theoretical objects called bisets.  The resulting algorithms have a symbolic nature.  Also, they involve floating point calculations as well as manipulations of triangulations on the sphere. Their algorithms have been implemented in the software package Img within the computer algebra system GAP.

Nekrashevych uses the theory of bisets to give a ``combinatorial spider algorithm" that classifies post-critically finite topological polynomials by their bisets \cite{nek_csa}.  This algorithm is, however, not known to terminate.

Utilizing the work of Bartholdi--Dudko and Bartholdi--Nekrashevych, Kelsey--Lodge enumerate the Thurston equivalence classes of branched covers of the sphere of degree 2 with at most 4 post-critical points \cite{KL}.  

Shepelevtseva--Timorin \cite{ST} define invariant spanning trees for quadratic rational maps as a tool for classifying post-critically finite branched covers of the sphere of degree 2.  As in our paper, they have a scheme where they iteratively lift trees in order to search for an invariant tree.  Their process is similar to ours, but they do not prove that their process converges.  One of their results appears in Section~\ref{sec:poirier} below.  As in the work of Bartholdi--Dudko, their proofs are phrased in terms bisets.

Bonnot--Braverman--Yampolsky \cite{BBY} prove that it is decidable whether or not a post-critically finite branched cover of the sphere is equivalent to a rational map \cite{BBY}.  Like Bartholdi--Dudko, they work directly with triangulations of the sphere.  Their algorithm involves two parallel exhaustive searches, one searching for a Thurston obstruction, and one searching for an equivalent rational map.

Building on the work of Bonnot--Braverman--Yampolsky, Selinger--Yampolsky give an algorithm that finds the canonical obstruction for a post-critically finite branched cover of the sphere \cite{SY}. Recent work of Rafi--Selinger--Yampolsky \cite{RSY} pairs the improved algorithm of Selinger--Yampolsky for detecting obstructions with an improved algorithm for detecting Thurston equivalence of rational maps.  The improvements are obtained by applying known algorithms for the conjugacy problem in the mapping class group.   

Cannon--Floyd--Parry--Pilgrim \cite{CFPP} focus attention on a special subset of post-critically finite branched covers of the sphere they call  nearly-Euclidean Thurston maps (NET maps).  A NET map is a post-critically finite branched cover of the sphere with exactly 4 post-critical points and the property that each critical point has local degree 2.   Floyd--Parry--Pilgrim proved that rationality is decidable for NET maps \cite {FPP}.  They leverage the near-Euclidean behavior of the maps to find an upper bound on the slope of an obstruction.  In a separate paper, Floyd--Parry--Pilgrim provide an algorithm for constructing dynamic portraits for NET maps and they classify dynamic portraits of degree up to 30; see \cite{FPP17}.

Our tree lifting algorithm has some important qualitative differences from the above works:
\begin{enumerate}
    \item It applies to all post-critically finite topological polynomials. \smallskip
    \item It does not require an exhaustive search.\smallskip
    \item It gives recognition of topological polynomials, not just comparison.\smallskip
    \item It gives the conjugating map between two equivalent topological polynomials, not just the fact that they are equivalent.\smallskip
    \item It effectively computes a primary invariant, namely the Hubbard tree, rather than a secondary invariant, such as a  biset or an iterated monodromy group.\smallskip
    \item It is coordinate free, unlike algebraic methods which require choosing a basis for a biset or a generating set for a group.
\end{enumerate}
We suspect that our algorithm runs in polynomial time (possibly even quadratic time), and that it can be implemented effectively.   Our algorithm is no doubt more efficient than an algorithm that simply lists and checks all possible isotopy classes of trees and all possible obstructions.

One shortcoming of our tree lifting algorithm is that it does not have an immediate extension to the case of rational maps because it relies on the existence (and theory of) Hubbard trees.  However, there has been work on invariant trees for special classes of post-critically finite branched covers of the sphere: by Shepelevtseva--Timorin for quadratic rational maps~\cite{ST} and by Hlushchanka for expanding rational maps \cite{Hlushchanka2017}. Their work may provide a framework for generalizing our tree lifting algorithm.  

Finally, we have recently learned of work in preparation by Ishii--Smillie wherein they give an algorithm for computing the homotopy class of the Hubbard tree for the class of post-critically finite expanding polynomials in terms of iterated pullbacks of loops in the sphere.

\p{Outline of the paper} We begin in Section~\ref{sec:complex} by introducing the tree complex $\T_n$, the augmented tree complex $\hat \T_n$, and the associated lifting maps~$\lambda_f$.  We also prove there that $\T_n$ and $\hat \T_n$ are contractible. In Sections~\ref{sec:poirier} and~\ref{sec:obs} we prove the first and second statements of Theorem~\ref{thm:main}, respectively.  Finally, in Section~\ref{sec:rabbit} we explain how to use our tree lifting algorithm to solve Hubbard's original twisted rabbit problem, our generalization to the case of $n$ post-critical points (Theorem~\ref{thm:many}), and the generalized twisted $z^2+i$ problem (Theorem~\ref{thm:I}).  

\p{Acknowledgments} We thank Nicol\`as Alvarado, Mladen Bestvina, Benson Farb, Bill Floyd, Mikhail Hlushchanka, Chris Hruska, Sarah Koch, Chris Leininger, Volodymyr Nekrashevych, Kevin Pilgrim, Roberta Shapiro, Dylan Thurston, and several anonymous referees for helpful comments and conversations.  We are also grateful to Eko Hironaka and Sarah Koch for organizing a conference called ``Braids and Rational Maps: An Informal Gathering'' at Harvard University in August 2017, where our collaboration was begun.  We are also grateful to ICERM for hosting a conference called ``Algorithms in Complex Dynamics and Mapping Class Groups'' in November 2019, which helped the development of the paper.  The second author thanks the School of Mathematics and Statistics at St Andrews for their hospitality during a research visit.  The first author was supported by EPSRC grant EP/R032866/1 and the National Science Foundation under Grant No.\ DMS - 1854367. The second author was supported by the National Science Foundation under Grant No.\ DGE - 1650044.  The third author was supported by the National Science Foundation under Grant No.\ DMS - 1745583.  The fourth author was supported by the National Science Foundation under Grant No.\ DMS - 2002951.  


\section{The complex of trees and the lifting map}
\label{sec:complex}

The goal of this section is to introduce some of the main objects of study in this paper.  Specifically, we define
\begin{enumerate}
    \item the complex of trees $\T_n$,
    \item the space of metric trees $\Y_n$,
    \item the augmented complex of trees $\hat \T_n$, and
    \item the lifting maps $\lambda_f : \T_n \to \T_n$ and $\lambda_f : \hat \T_n \dasharrow \hat \T_n$. 
\end{enumerate}
We accomplish these goals in four corresponding subsections below.  The space $\Y_n$ is introduced mainly as an auxiliary object; it is used to prove that $\T_n$ is contractible (Proposition~\ref{prop:tnc}).  Specifically, we first show that $\Y_n$ is homeomorphic to Teichm\"uller space (Proposition~\ref{prop:teich}), which is a contractible space.   We then show that  $\T_n$ is a spine for $\Y_n$, that is, a subspace to which $\Y_n$ deformation retracts (Proposition~\ref{prop:tnc}).  

As mentioned in the introduction, the complex $\T_n$ is closely related to a poly-simplicial complex $\widetilde D_n$ defined by Nekrashevych \cite[Section 7.6]{nek_csa}.  One point of distinction is that the complex $\widetilde D_n$ does not seem to be directly related to a subdivision of Teichm\"uller space.

In geometric group theory, there are many analogues of the above objects, such as the curve complex, the arc complex, and Teichm\"uller space.  For instance, in their work on quadratic differentials, Hubbard and Masur \cite{HM}  constructed a simplicial complex of trees that is related to $\T_n$.  Our work is in particular inspired by the theory of outer space, a simplicial complex defined by Culler--Vogtmann in their study of the automorphism group of a free group \cite{cv}.


\subsection{The complex of trees}\label{sec:trees}  Before defining the complex $\T_n$ we first specify precisely what we mean by an isotopy class of trees.  Throughout this section, $P \subseteq \R^2$ is a set with $n$ elements; we refer to $P$ as a set of marked points. 

\p{Trees and isotopy} By a \emph{tree} in $(\R^2,P)$ we mean an embedding $\varphi$ of an abstract tree $T_0$ into $\R^2$ with the following three properties:
\begin{enumerate}
    \item the set $P$ is contained in the set $\varphi(T_0)$,
    \item the set $\varphi^{-1}(P)$ is contained in the set of vertices of $T_0$, and
    \item the set of vertices of $T_0$ with valence at most 2 is contained in $\varphi^{-1}(P)$.
\end{enumerate}
Let $T = \varphi(T_0)$ be a tree in $(\R^2,P)$.  We refer to the images of the vertices and edges of $T_0$ as the vertices and edges of $T$.  Some examples of trees in $(\R^2,P)$ are given in Figure~\ref{fig:ManyAllowedTrees}. In our diagrams, marked points are colored red.  We say that two trees in $(\R^2,P)$ are \emph{isotopic} if they are isotopic (as maps) through trees in $(\R^2,P)$.  

Let $T$ be a tree in $(\R^2,P)$ and let $F$ be a subforest of $T$ with the property that each component of $F$ contains at most one point of $P$; such a forest is said to be \emph{collapsible}.  We may form a new tree $T'=T/F$ in $(\R^2,P)$ by collapsing each component of $F$ to a single point.  We say that $T'$ is obtained from $T$ by a \emph{forest collapse}.  We may equivalently say that $T$ is obtained from $T'$ by a \emph{forest expansion}.  Forest collapses and expansions are well-defined operations for isotopy classes of trees.

\begin{figure}[t]
\centering
$\begin{array}{c@{\qquad}c@{\qquad}c}
\includegraphics{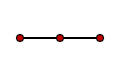} & \includegraphics{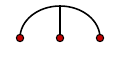} & \includegraphics{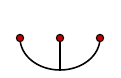} \\[16pt]
\includegraphics{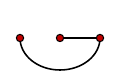} & \includegraphics{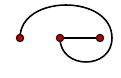} & \includegraphics{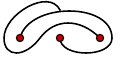}
\end{array}$
\caption{Several trees in $(\R^2,P)$ for $P=\{-1,0,1\}$}
\label{fig:ManyAllowedTrees}
\end{figure}

\p{Trees as arc systems} There is an alternative way to describe a tree in $(\R^2,P)$, in terms of arc systems.  

First, an \emph{arc based at infinity} is the image of a proper, simple embedding of $(0,1)$ into $\R^2$ that avoids $P$ (in particular, all arcs connect infinity to itself).  Such an arc is \emph{essential} if it is not isotopic to infinity through arcs based at infinity. An \emph{arc system} in $(\R^2,P)$ is a collection of essential arcs based at infinity that are pairwise disjoint and pairwise non-isotopic.  An arc system in $(\R^2,P)$ is \emph{filling} if each complementary region is a disk with at most one marked point.  

There is a natural bijection between the set of isotopy classes of trees in $(\R^2,P)$ and the set of isotopy classes of filling arc systems in $(\R^2,P)$.  Given a tree $T$ in $(\R^2,P)$, a corresponding arc system has one arc $\alpha$ for each edge $e$ of $T$; specifically, $\alpha$ is an arc that crosses $e$ in one point and is disjoint from $T$ otherwise. 

Collapsing a forest in a tree $T$ corresponds to deleting arcs in the corresponding arc system.  See Figure~\ref{fig:TreeDual} for an example; the tree on the right is obtained by contracting the middle edge in the tree on the left.

\begin{figure}[t]
\centering
\includegraphics{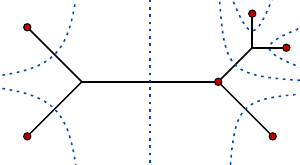}\qquad\qquad\qquad
\includegraphics{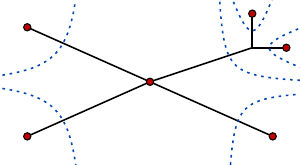}
\caption{Two trees in $(\R^2,P)$ and their dual arcs}
\label{fig:TreeDual}
\end{figure}

\p{Definition of the complex of trees} The \emph{complex of trees} $\T_n$ is the simplicial complex defined as follows.  The vertices of $\T_n$ are the isotopy classes of trees in $(\R^2,P)$.  A set of vertices $\{T_0,\dots,T_k\}$ spans a $k$-simplex if (up to relabeling) for each $i > 0$ the vertex $T_i$ is obtained from $T_{i-1}$ by a forest collapse (equivalently, for $i < j$ the vertex $T_j$ is obtained from $T_i$ by a forest collapse).  The complex of trees $\T_3$ is isomorphic to an infinite $3$-regular tree. A portion of $\T_3$ is illustrated in Figure~\ref{fig:TreeComplex3}.

\begin{figure}[h]
\centering
\includegraphics{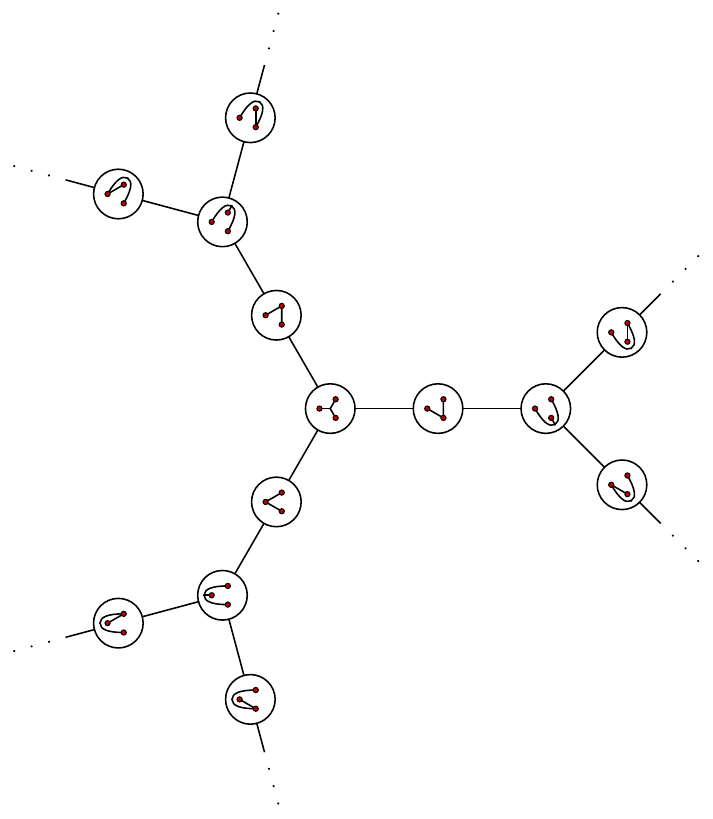}
\caption{A portion of $\T_3$}
\label{fig:TreeComplex3}
\end{figure}

A priori the complex $\T_n$ depends on the choice of $P$.  However, if $P'$ is another subset of $\R^2$ with $|P'|=|P|$ then a homeomorphism $(\R^2,P) \to (\R^2,P')$ induces an isomorphism between the corresponding complexes of trees.

\p{Local finiteness and metric balls} Since each tree representing a vertex of $\T_n$ has finitely many vertices and edges, it follows that $\T_n$ is locally finite, that is, the degree of each vertex in the 1-skeleton is finite. 

There is a natural metric on the set of vertices, given by the path metric in the 1-skeleton of~$\T_n$.  Combined with the local finiteness, this means that balls of finite radius contain finitely many vertices.  This finiteness property will be crucial in our proof of Theorem~\ref{thm:main}.

\p{Mapping class group action} Let $P$ be any set of marked points in $\R^2$.  The mapping class group $\Mod(\R^2,P)$ is the group of isotopy classes of the set of orientation-preserving homeomorphisms of $(\R^2,P)$.  The pure mapping class group $\PMod(\R^2,P)$ is the subgroup consisting of elements that fix each point of $P$.

The group $\Mod(\R^2,P)$ acts on $\T_n$ in a natural way: for $h \in \Mod(\R^2,P)$ and $T \in \T_n$, we have that $h \cdot T$ is the point of $\T_n$ represented by the tree $\phi(t)$, where $t$ is a representative of the isotopy class $T$ and $\phi$ is a representative of $h$.


\subsection{The space of metric trees and contractibility}
\label{sec:space}

Our next goal is to show that $\T_n$ is contractible.  The strategy is to show that $\T_n$ can be realized as the spine of a space $\Y_n$ of metric trees, which is itself contractible.  

\p{Metric trees} As above, fix $P \subseteq \R^2$ with $|P|=n$.  A \emph{metric} on a tree $T$ in $(\R^2,P)$ is a function from the set of edges of $T$ to $\R_{\geq 0}$; we refer to the image of an edge as its length.  It makes sense to define a metric on an isotopy class of trees, since an isotopy between trees induces a bijection on the sets of edges. We say that a metric on a tree (or an isotopy class of trees) is \emph{degenerate} if there is a path of edges of length 0 connecting distinct points of $P$.

\p{The complex of metric trees} We will define $\Y_n$ as a sort of cell complex.  The ``cells'' we define will not be compact, and so the result is not a cell complex in the usual sense, but something more general.

For each vertex $T$ of $\T_n$ we consider the set of all nondegenerate metrics on $T$ where the sum of the lengths of the edges is 1.  This set of metrics is a subset of the standard simplex in~$\R^m$, where $m$ is the number of edges of $T$.  
The resulting subset of the standard simplex will be referred to as a \emph{cell}.  Any nonempty intersection of a face of the simplex with a cell will be called a \emph{face} of the cell.  A face of a cell is also a cell; specifically it is the cell corresponding to the tree $T'$ obtained from $T$ by collapsing some edges of $T$.

We form the cell complex $\Y_n$ by starting with the disjoint union of the cells associated to all vertices of $\T_n$.  We identify a face of the cell for a tree $T$ with the cell for the corresponding collapsed tree $T'$.

Figure~\ref{fig:FareyTessellation} illustrates the cell complex $\Y_3$.  Every cell of $\Y_3$ has dimension~$1$ or~$2$, and $\Y_3$ is homeomorphic to an open disk. The simplicial subdivision of this disk is the same as the usual Farey tessellation of the hyperbolic plane.

It is also possible to define $\Y_n$ as a topological space independently of any cell structure: an isotopy class of metric trees induces a length function on the set $\mathcal{X}$ of isotopy classes of arcs in $(\R^2,P)$ between points of $P$.  This length function can be shown to be injective.  The topology on $\Y_n$ is then the topology induced from the product topology on $\R^{\mathcal X}$.  The resulting topology is homeomorphic to the one that $\Y_n$ inherits from the cell structure given above.

\begin{figure}[h]
\centering
\includegraphics{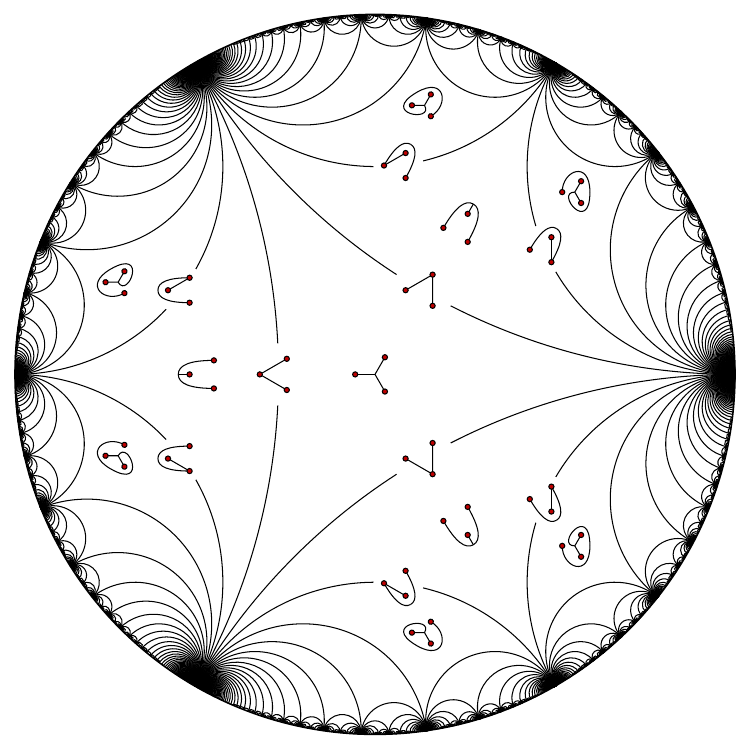}
\caption{The complex $\Y_3$}
\label{fig:FareyTessellation}
\end{figure}

\p{Connection with Teichm\"uller space and contractibility} Our next goal is to show that $\T_n$ is contractible.  To this end, we first show that $\Y_n$ is contractible.  The key is to identify the latter with a subspace of a certain arc complex, as follows.

Let $\A_n$ denote the simplicial complex whose vertices are isotopy classes of essential arcs in $(\R^2,P)$ based at $\infty$, and whose simplices correspond to arc systems, that is, collections of pairwise disjoint isotopy classes of arcs.  Then let $\A_n^\circ$ denote the topological space obtained from the geometric realization of $\A_n$ by deleting the simplices of $\A_n$ corresponding to arc systems that fail to fill $(\R^2,P)$.  Using barycentric coordinates on the simplices, a point of $\A_n^\circ$ can be regarded as a weighted arc system, where we assign a number in $[0,1]$ to each arc. 

The space $\Y_n$ is naturally homeomorphic to $\A_n^\circ$: the weighted arc system corresponding to a tree in $(\R^2,P)$ is given by a collection of arcs transverse to the edges of the tree, with weights inherited from the transverse edges.   More specifically, for each edge of the tree, we take the unique isotopy class of arcs that intersect that edge in one point and are disjoint from the other edges of the tree, and we declare the weight of the arc to be the length of the corresponding edge of the tree.  This process is reversible (this is where we use the fact that the arc systems corresponding to points of $\A_n^\circ$ are filling), whence the homeomorphism. 

Let $\Teich_{0,n+1}$ denote the Teichm\"uller space of a sphere with $n+1$ punctures.  Penner proved that there is a $\Mod(\R^2,P)$-equivariant homeomorphism from $\Teich_{0,n+1}$ to $\A_n^\circ$ \cite[Theorem 1]{penner}.  We thus obtain the following immediate corollary.

\begin{proposition}
\label{prop:teich}
For $n \geq 2$, there is a $\Mod(\R^2,P)$-equivariant homeomorphism from $\Teich_{0,n+1}$ to $\Y_n$.  In particular $\Y_n$ is homeomorphic to $\R^{2n-4}$.
\end{proposition}

In this paper the only consequence of Proposition~\ref{prop:teich} that we use is the fact that $\Y_n$ is connected.

\p{Contractibility of the complex of trees} Our complex $\T_n$ can be regarded as a \emph{$\Mod(\R^2,P)$-equivariant spine} for the space $\Y_n$, by which we mean that $\T_n$ is a subspace of $\Y_n$ which is invariant under the action of $\Mod(\R^2,P)$ and onto which $\Y_n$ deformation retracts in a $\Mod(\R^2,P)$-equivariant fashion.  We can regard $\T_n$ as the subset of $\Y_n$ consisting of isotopy classes of metric trees with the following property: if we scale the metric so that the maximal length of an edge is 1, then the set of edges with length strictly less than 1 forms a subforest where each component contains at most one element of $P$.  

Identifying $\T_n$ as a subset of $\Y_n$ as above, we can realize $\T_n$ as a $\Mod(\R^2,P)$-equivariant deformation retract of $\Y_n$.  In order to retract an arbitrary point $T$ of $\Y_n$ to $\T_n$ we choose the largest $\ell \in [0,1]$ so that there is no path in $T$ that connects two vertices of $P$ and only traverses edges of length less than $\ell$.  We then build a new point $T'$ of $\Y_n$ by changing all lengths in $T$ in $[\ell,1]$ to 1 and then rescaling so that the sum of the lengths of the edges is 1.  We then form the linear interpolation between $T$ and $T'$.  This process describes the desired retraction.  In summary we have the following proposition.

\begin{proposition}
\label{prop:tnc}
For $n \geq 2$, there is a natural embedding of the complex $\T_n$ as a $\Mod(\R^2,P)$-equivariant spine in $\Y_n$.  In particular, $\T_n$ is contractible.  
\end{proposition}


\subsection{The augmented complex of trees}\label{sec:aug} In our analysis of obstructed topological polynomials, it will be advantageous for us to consider an augmentation $\hat \T_n$ of $\T_n$, that is, a simplicial complex that contains $\T_n$ as a subcomplex.  The vertices of $\hat \T_n$ that do not lie in $\T_n$ are called bubble trees; we begin by describing these.

\p{Bubble trees} A simple closed curve is \textit{essential} if it is not homotopic into a neighborhood of a marked point or to a neighhborhood of $\infty$ (homotopies here may not pass a curve through a marked point).  Equivalently, a curve is essential if it has at least two marked points in its interior and at least one marked point in its exterior. 

A \textit{multicurve} $M$ in $(\R^2,P)$ is a collection $\{c_1,\ldots,c_m\}$ of pairwise disjoint, pairwise non-homotopic, essential simple closed curves in $(\mathbb{R}^2, P)$.  A multicurve $M$ is \textit{un-nested} if no two curves of $M$ are nested.  

Given an un-nested multicurve $M$, we may obtain a new surface $(\mathbb{R}^2, \bar P)$ from $(\mathbb{R}^2, P)$ by collapsing the interior of each component of $M$ to a marked point.  The new set of marked points $\bar P$ has one element for each component of $M$ and one element for each element of $P$ not contained in the interior of an element of $M$. 

A \emph{bubble tree} in $(\mathbb{R}^2, P)$ is a graph $B$ in $\R^2$ with the following properties:
\begin{enumerate}
\item $B$ is the union of a (possibly empty) un-nested multicurve $M$ in $(\mathbb{R}^2, P)$ with a forest $B_E$ in $\mathbb{R}^2$,\smallskip
\item the leaves of $B_E$ lie in $M \cup P$,\smallskip
\item the intersection $B_E \cap M$ is contained in the set of leaves of $B_E$, \smallskip
\item the forest $B_E$ is disjoint from the interiors of the disks bounded by $M$, and \smallskip
\item the image of $B_E$ in the surface $(\mathbb{R}^2, \bar P)$ obtained by collapsing the disks bounded by $M$ is a tree in $(\mathbb{R}^2, \bar P)$.  
\end{enumerate}
We refer to $B_E$ as the \emph{exterior forest} for $B$ and we refer to the image $T_E$ in $(\mathbb{R}^2, \bar P)$ as the \emph{exterior tree}.  We refer to $M$ as the multicurve of $B$, and we also refer to the components of $M$ as the \emph{bubbles} of $B$.  See the left-hand side of Figure~\ref{fig:BubbleTrees} for a picture of a bubble tree.  

An alternate way to designate a bubble tree is by a pair $(M,T_E)$, where $M$ is an un-nested multicurve, and $T_E$ is a tree in the surface obtained by crushing the interiors of the components of $M$ to points.  Up to isotopy, there is a unique bubble tree $B$ with $M$ as its multicurve and with $T_E$ as the exterior tree.

\begin{figure}
\centering
\includegraphics{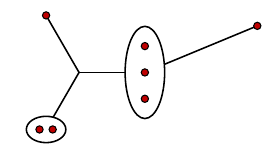}
\qquad\qquad
\includegraphics{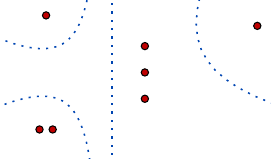}
\caption{A bubble tree and the corresponding non-filling arc system.}
\label{fig:BubbleTrees}
\end{figure}

\p{Arc systems versus bubble trees} There is a natural correspondence between bubble trees and simplices of the arc complex $\A_n$ from Section~\ref{sec:space}.  Given a simplex of $\A_n$, that is, an arc system, the bubbles of the corresponding bubble tree correspond to those complementary regions to the arc system containing more than one marked point (we should choose the bubble curves to each be contained in one of these complementary regions, and homotopic to the boundary). An example of this correspondence is shown in Figure~\ref{fig:BubbleTrees}.  If we collapse the disk bounded by each such bubble to a marked point, then the arc system becomes filling, and we obtain an isotopy class of trees on the collapsed surface; by un-crushing, this gives rise to a bubble tree in the original surface.  This process is reversible, and gives the desired identification.

\p{Simplices} One way to describe the simplices of $\hat \T_n$ is to define $\hat \T_n$ as the barycentric subdivision of $\A_n$.  This agrees with our description of the vertices of $\hat \T_n$ as simplices of $\A_n$.

It is also possible to describe the simplices of $\hat \T_n$ directly in terms of bubble trees.  In particular:
\begin{enumerate}
\item If $T$ is a tree of $\T_n$ and $F$ is a subforest of $T$ that has at least one component with more than one marked point, then collapsing $F$ yields a bubble tree $B$ as shown in Figure~\ref{fig:BubbleForestCollapse}, with one bubble for each component of $F$ that has more than one marked point.  This collapse represents an edge in $\hat\T_n$ joining $T$ and~$B$.\smallskip
\item If $B$ is a bubble tree and $F$ is a subforest of the exterior forest for $B$, then collapsing $F$ yields a new bubble tree~$B'$, representing an edge in $\hat\T_n$ joining $B$ and $B'$.  If any component of $F$ touches a bubble of $B$, then that bubble enlarges and subsumes the collapsed edges in~$B'$, with bubbles joining together if some component of $F$ touches more than one bubble.
\end{enumerate}
As in $\T_n$, higher-dimensional simplices correspond to sequences of such collapses.
\begin{figure}
\centering
$\raisebox{-0.47\height}{\includegraphics{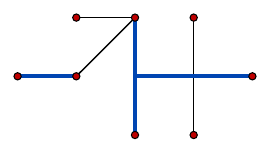}}\qquad\longrightarrow\qquad\raisebox{-0.47\height}{\includegraphics{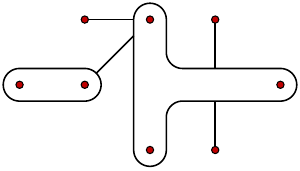}}$
\caption{Collapsing a forest (the thick blue edges) whose components join different marked points together yields a bubble tree.}
\label{fig:BubbleForestCollapse}
\end{figure}

\p{Contractibility} The complex $\A_n$ is contractible (for instance the proof given in \cite[Theorem 5.5]{primer} applies).  It follows that $\hat \T_n$ is contractible, although we will not use this fact in this paper.  We will use the fact that if $\T_n^{\bigcdot}$
is the full subcomplex of $\hat \T_n$ spanned by $\T_n$ and one other vertex in $\hat \T_n$, then $\T_n^{\bigcdot}$ is connected.

\p{The augmented space of metric trees} Another way of describing the complex $\hat \T_n$ is through the corresponding augmentation of $\Y_n$.  There is a natural augmentation $\hat \Y_n$ of $\Y_n$, namely, the space of metric trees with edge lengths in $[0,1]$ instead of $(0,1]$.  Additionally, there is a natural cell structure on $\hat \Y_n$; this is like the cell structure on $\Y_n$, except that instead of simplices with some faces removed, the cells are the entire simplices. The complex $\hat \T_n$ is the poset of cells of~$\hat \Y_n$.


\subsection{Lifting maps}
\label{sec:lifting}

Let $f$ be a post-critically finite topological polynomial with $|P_f|=n$.  The final goal of this section is to describe lifting maps
\begin{align*}
    \lambda_f &: \T_n \to \T_n, \\
    \lambda_f &: \hat \T_n \dasharrow \hat \T_n, \text{ and} \\
    \lambda_f &: \Y_n \to \Y_n,
\end{align*}
where the lifting map on the augmented complex $\hat\T_n$ is defined only on a subcomplex (hence the dashed arrow). After giving the definitions, we give a comparison between the last map and Thurston's pullback map $\sigma_f : \Teich_{0,n+1} \to \Teich_{0,n+1}$.

\p{Lifting on the complex of trees} Let $T$ be a tree in $(\R^2,P)$ where $P=P_f$.  The preimage of $T$ under $f$ is a tree in $(\R^2,f^{-1}(P))$.  Indeed, if $f^{-1}(T)$ had a cycle, then its complement would contain multiple connected components; each of these would necessarily map to the exterior of~$T$, violating the assumption that $f$ is a topological polynomial (since, thinking of $f$ as a map of the sphere, we would have multiple points mapping to $\infty$). Further, if $f^{-1}(T)$ had more than one connected component, then the complement of $f^{-1}(T)$ in $\R^2$ would be a sphere with more than two punctures and the map $f$ would induce an unbranched cover of this complement to the complement of $T$, which is a sphere with exactly two punctures; this is a contradiction.  

We may obtain a tree in $(\R^2,P)$ from $f^{-1}(T)$ by taking the convex hull of $P$ in $f^{-1}(T)$.  By \emph{convex hull} we mean the union of the simple paths in $f^{-1}(T)$ connecting the points of $P$ pairwise.  The result is the desired tree $\lambda_f(T)$ in $(\R^2,P)$. An example of this process is shown in Figure~\ref{fig:LiftExample}.

\begin{figure}
\centering
$\underset{\textstyle \phantom{f^{-1}}T\phantom{f^{-1}}}{\raisebox{-0.47\height}{\includegraphics{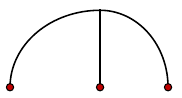}}}$
\qquad$\overset{\textstyle f^{-1}}{\xrightarrow{\qquad}}$\qquad
$\underset{\textstyle f^{-1}(T)}{\raisebox{-0.47\height}{\includegraphics{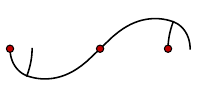}}}$
\qquad$\overset{\textstyle\text{hull}}{\xrightarrow{\qquad}}$\qquad
$\underset{\textstyle \phantom{f^{-1}}\lambda_f(T)\phantom{f^{-1}}}{\raisebox{-0.47\height}{\includegraphics{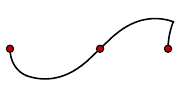}}}$
\caption{The lift of a tree $T$ in $(\R^2,P)$ under the airplane polynomial.  The result is isotopic to the Hubbard tree for the airplane.}
\label{fig:LiftExample}
\end{figure}

The above procedure of lifting and taking the convex hull is well defined on isotopy classes of trees, and so in this way we obtain the desired map $\lambda_f : \T_n \to \T_n$.  This map is simplicial, because any forest collapse in $T$ lifts to a forest collapse in $f^{-1}(T)$.

As usual, we may reinterpret $\lambda_f$ in terms of arc systems.  The preimage of a filling arc system is another filling arc system, and so the correspondence between filling arc systems and vertices of $\T_n$ gives another description of $\lambda_f$.

\p{Hubbard vertices} As in the introduction, we refer to the vertex of $\T_n$ corresponding to the topological Hubbard tree $H_f$ for an unobstructed topological polynomial $f$ as the Hubbard vertex for $f$.  The Hubbard tree for a polynomial $p$ has the property that $p(H_p)\subseteq H_p$, or equivalently $H_p\subseteq p^{-1}(H_p)$.  From this second inclusion it follows that $H_p$ is the convex hull of $P_p$ in $p^{-1}(H_p)$, and therefore $H_p$ is a fixed point of the lifting map $\lambda_p$ on~$\T_n$.  We conclude that the Hubbard vertex $H_f$ is a fixed point of $\lambda_f$ for any unobstructed topological polynomial~$f$.

Because we defined the Hubbard tree for a polynomial $p$ completely in terms of the post-critical set $P_p$ and the Julia set for $p$, and since both of these are invariant under passing to iterates of $p$, it follows that the Hubbard tree for an iterate $p^k$ of~$p$ is the same as the Hubbard tree for~$p$. Thus the Hubbard vertex for a topological polynomial $f$ is the same as the Hubbard vertex for any iterate $f^k$.

\p{Lifting on the augmented complex of trees} The simplest way to define the lifting map $\lambda_f$ on $\hat \T_n$ is to use the correspondence between bubble trees and non-filling arc systems, and to lift arc systems though $f$, as for $\T_n$.  It is possible for the lift of a non-filling arc system to be empty.  This happens when $f(P)$ is a proper subset of~$P$ and the given non-filling arc system does not separate any points of $f(P)$ from one another. Since the empty arc system does not correspond to a vertex of~$\hat\T_n$, the lifting map is only be defined on a subcomplex of~$\hat \T_n$ for such~$f$.

It is also possible to describe $\lambda_f$ in terms of bubble trees.  Given a bubble tree $B$, its pre-image under $f$ is a graph.  If a bubble of this preimage has one or fewer marked points in its interior, then we may collapse its interior to a single point.  Then $\lambda_f(B)$ is the union of the essential bubbles in the resulting graph, along with any edges that lie on some simple path connecting marked points or essential bubbles. Again, if all of the points of $f(P)$ lie in a single bubble of~$B$, it is possible that one bubble of the preimage will have all of the marked points in its interior, in which case the lift of the given bubble tree is not defined.

\p{Lifting on the complex of metric trees} The map $\lambda_f : \Y_n \to \Y_n$ is defined in the same way as the map $\lambda_f : \T_n \to \T_n$.  The only new ingredient is that if $T$ is a metric tree, then $f^{-1}(T)$ inherits a metric.  Indeed, each edge of $f^{-1}(T)$ is a union of preimages of edges of $T$, and so the length of a given edge is the sum of the lengths of the corresponding edges in $T$ (with multiplicity).  The convex hull also inherits a metric from the metric on $f^{-1}(T)$, namely, the restriction.  We may scale the metric on the convex hull so that the sum of the lengths of the edges is 1, thus obtaining a point $\lambda_f(T)$ of $\Y_n$.  

\p{A comparison with Thurston's pullback map} Through the identification of $\Y_n$ with $\Teich_{0,n+1}$ there is a strong analogy between our lifting map $\lambda_f : \Y_n \to \Y_n$ and Thurston's pullback map $\sigma_f : \Teich_{0,n+1} \to \Teich_{0,n+1}$.  One similarity between these maps is that they are both invariant under the action of the liftable mapping class group $\LMod(\mathbb{R}^2,P)$ (see Section~\ref{sec:onear} for the definition). That is, if $h\in
\LMod(\mathbb{R}^2,P)$ is a liftable mapping class, $\tilde{h}\in\PMod(\mathbb{R}^2,P)$ is its lift, and $h_\star$ and $\tilde h_\star$ are the induced homeomorphisms of $\Teich_{0,n+1}$, then $\sigma_f h_\star=\tilde{h}_\star\sigma_f$ and $\lambda_f h_\star=\tilde{h}_\star\lambda_f$.

However, the maps $\sigma_f$ and $\lambda_f$ are in general not equal to each other, or even conjugate to each other.  For instance, when $f$ is an unobstructed topological polynomial, $\sigma_f$ has a unique fixed point in $\Teich_{0,n+1}$ and this fixed point is attracting under iteration of $\sigma_f$.  But there are examples of polynomials (such as the third iterate of the rabbit polynomial $R$) where $\lambda_f$ has a positive-dimensional simplex where each point is fixed.  On the other hand, there exist examples where $\sigma_f$ and $\lambda_f$ are equal, for instance when they are both the constant map.


\section{Finding Hubbard vertices for unobstructed maps}
\label{sec:poirier}

In this section we prove the first statement of Theorem~\ref{thm:main}, which says that the 2-neighborhood of the Hubbard vertex for an unobstructed topological polynomial $f$ is a nucleus for the action of $\lambda_f$ on the corresponding tree complex $\T_n$.   

Our proof is based on a theorem of Poirier that gives two combinatorial conditions for an invariant tree to be a Hubbard tree: an angle condition and an expanding condition.  Poirier proves his result for a slightly larger class of maps than the ones we have been considering, namely, the class of marked topological polynomials.  (We warn the reader that the words ``expansion'' and ``expanding'' in the terms ``forest expansion'' and ``expanding condition'' refer to distinct ideas.)

We begin in Section~\ref{sec:alex} by defining marked topological polynomials, invariant trees, and the dynamical map.  We also state and prove the Alexander method, a tool that we use throughout this paper to show that two topological polynomials are Thurston equivalent.  We state our version of Poirier's theorem as Proposition~\ref{prop:PoirierConditions} in Section~\ref{sec:p1}; readers unfamiliar with Hubbard trees might simply take this proposition as the definition.  In Sections \ref{sec:p2} and \ref{sec:p3} we show that any tree that is invariant under the tree lifting map admits a forest expansion satisfying the angle condition (Proposition~\ref{prop:ExpandAngleCondition}), and then that the latter admits a forest collapse satisfying the expanding condition (Proposition~\ref{prop:InvariantWithinTwo}).  With this in hand, we complete the proof of the first statement of Theorem~\ref{thm:main} in Section~\ref{sec:p4}.  

\subsection{Invariant trees and the Alexander method}
\label{sec:alex}

In this section we describe several combinatorial tools that we will use in the recognition of topological polynomials.  As discussed above, we begin by defining the class of marked topological polynomials.  We then define invariant trees for marked topological polynomials and their corresponding dynamical maps.  Finally we state and prove the Alexander method, as well as a version of the Alexander method that is specialized to the case of maps of degree 2.

\p{Marked topological polynomials}
 A \textit{marked topological polynomial} is a pair $(f,A)$, where $f$ is a post-critically finite topological polynomial and $A$ is a finite set in $\mathbb{R}^2$ that contains $P_f$ and satisfies $f(A)\subseteq A$.  We refer to $A$ as the set of marked points.  Every post-critically finite topological polynomial $f$ can be regarded as a marked topological polynomial with $A= P_f$.  
 
 Two marked topological polynomials $(f,A)$ and $(g,B)$ are \textit{Thurston equivalent} if there are orientation-preserving homeomorphisms $\phi_0,\phi_1\colon(\mathbb{R}^2,A)\to(\mathbb{R}^2,B)$ that are isotopic relative to~$A$ such that $\phi_0 f=g \phi_1$.

A marked topological polynomial $(f,A)$ is a \textit{marked polynomial} if $f$ is a polynomial map. The \textit{Hubbard tree} for a marked polynomial $(f,A)$ is the tree in $(\mathbb{R}^2,A)$ obtained as the union of all regulated arcs in the filled Julia set for $f$ between pairs of points in~$A$ (see \cite{DH1,DH2}).  More generally, if $(f,A)$ is a marked topological polynomial that is Thurston equivalent to a marked polynomial $(g,B)$ by homeomorphisms $\phi_0,\phi_1$ (as in the introduction), then a \textit{topological Hubbard tree} for $(f,A)$ is the preimage of the Hubbard tree for $(g,B)$ under $\phi_0$.

If $(f,A)$ is a marked topological polynomial and $T$ is a tree in $(\mathbb{R}^2,A)$, its \textit{lift} is the convex hull in $f^{-1}(T)$ of the points in~$A$.

\p{The Alexander method} The following proposition is one of the key technical tools of the paper.  In essence, it says that a marked topological polynomial is completely determined by its action on a single tree.  A closely related statement was proven by Bielefeld--Fisher--Hubbard \cite[Theorem 7.8]{BFH}.

\begin{proposition}[Alexander method]\label{prop:alexander} Let $(f,A)$ and $(g,B)$ be marked topological polynomials, let $T_f$ be a tree in $(\mathbb{R}^2,A)$ and $T_g$ be a tree in $(\mathbb{R}^2,B)$, and suppose there exists a homeomorphism $h\colon (\R^2,f^{-1}(A))\to (\R^2,g^{-1}(B))$ so that
\begin{enumerate}
    \item $h(T_f)$ is isotopic to $T_g$ in $(\mathbb{R}^2,B)$,\smallskip
    \item $h f$ and $g h$ agree on $f^{-1}(A)$, and\smallskip
    \item $h(f^{-1}(T_f))$ is isotopic to $g^{-1}(T_g)$ in $(\mathbb{R}^2,g^{-1}(B))$.
\end{enumerate}
Then $f$ and $g$ are Thurston equivalent (and $h$ gives the Thurston equivalence).  In the case where $A=B$, where $h$ is isotopic to the identity relative to $f^{-1}(A)$, and where $f$, $g$, and $h$ satisfy the above three conditions, we may further conclude that $f$ is isotopic to $g$ relative to~$f^{-1}(A)$.
\end{proposition}

\begin{proof}

It follows from the first condition in the statement that there exists a homeomorphism $h_0$ that is homotopic to $h$ relative to $A$ so that $h_0(T_f)$ is equal to $T_g$.  Similarly, it follows from the third condition that there exists a homeomorphism $h_1$ that is isotopic to $h$ relative to $f^{-1}(A)$ so that $h_1(f^{-1}(T_f))$ is equal to $g^{-1}(T_g)$.  It follows from the second condition that the restrictions of $g h_1$ and $h_0 f$ to the set of vertices of $f^{-1}(T_f)$ (i.e.~points of $f^{-1}(A)$ together with branching points) are equal.  Thus we may further modify $h_1$ by isotopy relative to $f^{-1}(A)$ so that the restrictions of $g h_1$ and $h_0 f$ to the entire tree $f^{-1}(T_f)$ are equal.  

We may identify $(\R^2,f^{-1}(A))$ and $(\R^2,B)$ with $(S^2,f^{-1}(A)\cup\infty)$ and $(S^2,B \cup \infty)$, by which we mean the sphere with marked points coming from $f^{-1}(A)$ (or~$B$) and $\infty$.  We may further regard $(S^2,f^{-1}(A)\cup\infty)$ as being obtained from $f^{-1}(T_f)$ by attaching a disk with a single marked point at $\infty$ and similarly for $(S^2,B \cup \infty)$ and $T_g$.  Since a branched cover of a disk with one marked point over another disk with one marked point is determined up to isotopy (relative to the boundary and the marked point) by its restriction to the boundary, it follows from the conclusion of the previous paragraph that we may further modify $h_1$ by an isotopy so that $g h_1$ is equal $h_0 f$.  In other words, $f$ and $g$ are Thurston equivalent, as desired.
\end{proof}

\begin{figure}[tb]
\centering
$\underset{\textstyle\text{(a)}}{\includegraphics{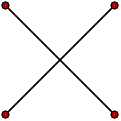}}$
\qquad\qquad
$\underset{\textstyle\text{(b)}}{\includegraphics{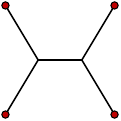}}$
\qquad\qquad
$\underset{\textstyle\text{(c)}}{\includegraphics{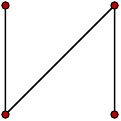}}$
\caption{(a)~The topological Hubbard tree $H_f$ for for the cubic polynomial $f(z)=z^3-3z/4+i\sqrt{7}/4$. (b)~A second invariant tree $T$ for $f$, obtained by expanding~$H_f$. (c)~A third invariant tree $T'$ for~$f$, obtained by contracting~$T$ along two edges}
\label{fig:TreesCubicRabbit}
\end{figure}

\p{Invariant trees} Given a post-critically finite topological polynomial~$f$, an \textit{invariant tree} for $f$ is a tree $T$ in $(\mathbb{R}^2,P_f)$ for which $\lambda_f(T)$ is isotopic to~$T$.   That is, $T$ is invariant if the convex hull of~$P_f$ in~$f^{-1}(T)$ is isotopic to~$T$ in $(\mathbb{R}^2,P_f)$.

We have seen that the Hubbard tree for a post-critically finite polynomial~$f$ is an invariant tree. One might guess that any invariant tree must be isotopic to the Hubbard tree, but this is not true.  For example, Figure~\ref{fig:TreesCubicRabbit} shows three distinct invariant trees for the cubic polynomial 
\[
f(z) = z^3 - \frac{3}{4} z + \frac{i\sqrt{7}}{4}.
\]
The polynomial $f$ has four post-critical points.

The tree $H_f$ on the left-hand side of Figure~\ref{fig:TreesCubicRabbit} is the topological Hubbard tree for $f$.  We can check with a computer that $f^{-1}(H_f)$ is the tree shown on the left-hand side of Figure~\ref{fig:example}.  The map $f^{-1}(H_f) \to H_f$ maps each edge labeled $\widetilde{e_i}$ in $f^{-1}(H_f)$ isomorphically to the edge $e_i$ in~$T$.  We then see from the picture that the convex hull of the marked points in $f^{-1}(H_f)$ is equal to $H_f$, and so $H_f$ is indeed invariant (even without knowing that it is the Hubbard tree).  

The tree $T$ in Figure~\ref{fig:TreesCubicRabbit}(b) is obtained from the Hubbard tree $H_f$ by a single edge expansion, and the tree $T'$ is obtained from $T$ by collapsing two edges.  If we perform the expansion on $H_f$ and at the same time perform the corresponding edge expansions on $f^{-1}(H_f)$ (again, refer to Figure~\ref{fig:example}), we construct in this way $f^{-1}(T)$ from $f^{-1}(H_f)$, and we again see that $T$ is the convex hull of the marked points of $f^{-1}(T)$.  We may similarly show that $T'$ is invariant by simultaneously performing edge collapses on $T$ and $f^{-1}(T)$.

\begin{figure}[tb]
\centering
\raisebox{-0.48\height}{\includegraphics{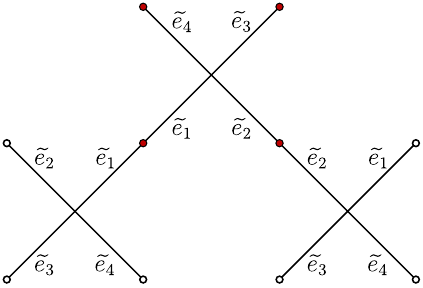}} $\qquad\longrightarrow\qquad$
\raisebox{-0.48\height}{\includegraphics{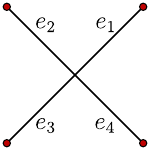}}
\caption{The topological Hubbard tree $H_f$ for $f(z) = z^3 - \frac{3}{4} z + \frac{i\sqrt{7}}{4}$ and $f^{-1}(H_f)$.  Here and throughout, preimages of marked points are shown in white, and each of the edges labeled $\widetilde{e_i}$ maps isomorphically to~$e_i$}
\label{fig:example}
\end{figure}

Pilgrim \cite{pilgrim_personal} has explained to us a topological construction of the map $f$ from Figure~\ref{fig:example}, as follows.  We start with a square in the plane.  We then connect all 4 corners to infinity by straight rays.  The result is a ``plane graph'' with 4 compact edges and 4 non-compact edges, or rays.  We may construct this graph so that it is invariant under rotation of the plane by $\pi$.  If we then blow up two rays at adjacent corners, as described in the paper by Pilgrim--Tan \cite[Section 5.2]{PT}, we obtain a topological polynomial that is Thurston equivalent to $f$.

We will use the polynomial $f$ as a running example throughout this section.

\p{The dynamical map}
Any invariant tree $T$ for a post-critically finite topological polynomial $f$ has an associated \textit{dynamical map} $f_*\colon T\to T$.  In the special case where $\lambda_f(T)=T$ setwise (or equivalently $f(T)\subseteq T$), this is simply the restriction of $f$ to~$T$, but in the general case where $\lambda_f(T)$ is isotopic to~$T$, we must compose $f$ with an isomorphism $T\to\lambda_f(T)$ obtained from an ambient isotopy fixing~$P_f$.  Note then that $f_*$ induces a well-defined map on vertices and edges, but is otherwise only defined up to homotopy.

The dynamical map $f_*\colon T\to T$ maps vertices of $T$ to vertices of~$T$, and maps each edge of $T$ to a path of one or more edges in~$T$.  Note that an edge of $T$ that has a critical point in its interior may map to a path of edges in $T$ that includes backtracking.  This mapping of vertices and edges depends only on the isotopy class of~$f$. 

For our example $f(z) = z^3 - \tfrac{3}{4} z + \tfrac{i\sqrt{7}}{4}$, we can see from Figure~\ref{fig:example} that the dynamical map $f_* \colon H_f \to H_f$ has the following action on the edges of $H_f$:
\[
f_*(e_i) = e_{i+2}
\]
where the indices are taken modulo 4.

\p{Alexander method in degree 2} Here we prove a useful consequence of the Alexander method for topological polynomials of degree 2, Proposition~\ref{prop:alex2} below. Shepelevtseva--Timorin~\cite[Theorem~A]{ST} prove a version of this for post-critically finite branched covers of the sphere of degree 2.  

A topological polynomial $f$ of degree 2 has a unique critical point $p_0$ in~$\R^2$; its image $f(p_0)$ is the critical value.  A topological polynomial of degree 2 is post-critically finite if and only if the critical point is either periodic or pre-periodic, and we will be focusing on the periodic case.

\begin{lemma}\label{lem:CriticalValueLeaf}Let $f$ be a topological polynomial of degree 2 whose critical point is periodic of period $n\geq 2$, and let $T$ be an invariant tree for~$f$.  Then the critical value for $f$ has valence 1 in~$T$.
\end{lemma}
\begin{proof}Let $p_0$ be the critical point for $f$ with orbit $(p_0,\ldots,p_{n-1})$.  Since $f_*$ is locally injective at $p_1,\ldots,p_{n-1}$, we know that
\[
\mathrm{val}(p_1) \leq \mathrm{val}(p_2) \leq \cdots \leq \mathrm{val}(p_{n-1})
\]
where $\mathrm{val}(p_i)$ denotes the valence of~$p_i$ in~$T$.  But the leaves of $T$ are marked, and at least one of $p_1,\dots,p_{n-1}$ is a leaf ($T$ has at least two leaves), and therefore $\mathrm{val}(p_1)=1$.
\end{proof}

The proof of the following proposition is a variant of an argument of Douady and Hubbard~\cite[Section~6.1]{DH1}.

\begin{proposition}\label{prop:alex2}Let $f$ and $g$ be topological polynomials of degree 2 whose critical points are periodic of the same period $n\geq 2$. Let $T_f$ and $T_g$ be invariant trees for~$f$ and $g$, and suppose there exists a homeomorphism $h\colon (\R^2,P_f)\to (\R^2,P_g)$ so that
\begin{enumerate}
\item $h(T_f)$ is isotopic to $T_g$ in $(\R^2,P_g)$,\smallskip
\item $h f$ agrees with $g h$ on $P_f$, and\smallskip
\item $h$ maps the critical point of $f$ to the critical point of $g$.
\end{enumerate}
Then $f$ and $g$ are Thurston equivalent (and $h$ gives the Thurston equivalence).  In the case where $P_f=P_g$, where $h$ is isotopic to the identity relative to~$P_f$, and where $f$, $g$, and $h$ satisfy the above three conditions, we may further conclude that $f$ is isotopic to $g$ relative to~$P_f$.
\end{proposition}
\begin{proof}Modifying $h$ by an isotopy, we may assume that $h(T_f)=T_g$.  Let $p_0$ and $q_0$ be the critical points for $f$ and $g$ respectively, with orbits $(p_0,\ldots,p_{n-1})$ and $(q_0,\ldots,q_{n-1})$.  By conditions~(2) and~(3), we know that $h(p_i) = q_i$ for each~$i$, and in particular $h(p_1)=q_1$.  We can therefore lift $h$ through $f$ and $g$ to obtain a homeomorphism $\tilde{h}\colon (\R^2,f^{-1}(P_f))\to (\R^2,g^{-1}(P_g))$ such that $g \tilde{h}=h f$.  Note then that $\tilde{h}$ maps $f^{-1}(T_f)$ to $g^{-1}(T_g)$.  Moreover, since $g(\tilde{h}(p_i)) = h(f(p_i)) = h(p_{i+1}) = q_{i+1}$ for each~$i$, we know that each $\tilde{h}(p_i)$ is either $q_i$ or $\delta_g(q_i)$, where $\delta_g$ is the nontrivial deck transformation for~$g$.  Replacing $\tilde{h}$ by $\delta_g \tilde{h}$ if necessary, we may assume that $\tilde{h}(p_1)=q_1$.  We claim that $\tilde{h}(p_i)=q_i$ for each~$i$.

We now prove the claim.  By Lemma~\ref{lem:CriticalValueLeaf}, the point $p_1$ is a leaf of $T_f$, so $p_0$ separates $f^{-1}(T_f)$ into two components that map homeomorphically to $T_f$ under $f$, and $q_0$ separates $g^{-1}(T_g)$ similarly.  Indeed, for each $i\geq 1$, one of the components of $g^{-1}(T_g)\setminus\{q_0\}$ contains $q_i$ and the other contains~$\delta_g(q_i)$.  But for such an $i$, the following statements are all equivalent:
\begin{enumerate}
\item $\tilde{h}(p_i)$ lies in the same component of $g^{-1}(T_g)\setminus\{q_0\}$ as $q_1$.
\item $p_i$ lies in the same component of $f^{-1}(T_f)\setminus\{p_0\}$ as $p_1$.
\item $p_i$ lies in the same component of $T_f\setminus\{p_0\}$ as $p_1$.
\item $q_i$ lies in the same component of $T_g\setminus\{q_0\}$ as $q_1$.
\item $q_i$ lies in the same component of $g^{-1}(T_g)\setminus\{q_0\}$ as $q_1$.
\end{enumerate}
Here the equivalence of (1) and (2) follows from the fact that $\tilde{h}$ maps $f^{-1}(T_f)$ homeomorphically to $g^{-1}(T_g)$, with $\tilde{h}(p_0)=q_0$ and $\tilde{h}(p_1)=q_1$; the equivalence of (2) and (3) follows from the invariance of~$T_f$; the equivalence of (3) and (4) follows from the homeomorphism between $T_f$ and $T_g$ induced by~$h$; and the equivalence of (4) and (5) follows from the invariance of~$T_g$.  From the equivalence of (1) and~(5), we conclude that $\tilde{h}(p_i)$ and $q_i$ lie in the same component of $g^{-1}(T_g)\setminus\{q_0\}$, and therefore $\tilde{h}(p_i)=q_i$ for all~$i$.  This completes the proof of the claim.  

Since $T_f$ is isotopic to the convex hull of $P_f$ in $f^{-1}(T_f)$ and $T_g$ is isotopic to the convex hull of $P_g$ in $f^{-1}(T_g)$, it follows that $\tilde{h}(T_f)$ is isotopic to $T_g$ in $(\R^2,P_g)$ (and hence $\tilde{h}$ is isotopic to $h$ relative to~$P_f$).  Also, $\tilde{h}$ agrees with $h$ on $P_f$, so $\tilde{h} f$ agrees with $g \tilde{h} = h f$ on $f^{-1}(P_f)$, and we know that $\tilde{h}$ maps $f^{-1}(T_f)$ to $g^{-1}(T_g)$. By the Alexander method (Proposition~\ref{prop:alexander}), we conclude that $h$ and $\tilde{h}$ provide a Thurston equivalence between $f$ and~$g$. 
\end{proof}

\subsection{Poirier's conditions}\label{sec:p1}

The goal of this subsection is to state Poirier's conditions and to prove that they are sufficient conditions for an invariant tree to be a topological Hubbard tree for an unobstructed post-critically finite topological polynomial; see Proposition~\ref{prop:PoirierConditions}.  In order to state the two conditions, we need to introduce the notion of angle assignments and the notion of a Julia edge.

\p{Angle assignments}
Given a tree $T$ in $(\R^2,P)$ and a vertex $v$ of $T$ with incident edges $e_1,\ldots,e_n$ (in counterclockwise order), the \textit{angles} at $v$ are the elements of the set
\[
\Theta(T,v) = \bigl\{ (e_1,e_2),\; (e_2,e_3),\; \ldots,\; (e_{n-1},e_n),\; (e_n,e_1) \bigr\}.
\]
An \textit{angle assignment} at $v$ is a function $\angle\colon \Theta(T,v)\to (0,1]$ such that
\[
\angle (e_1,e_2) + \cdots + \angle (e_{n-1},e_n) + \angle (e_n,e_1) = 1.
\]
More generally, the \textit{angles} $\Theta(T)$ of a tree $T$ are the disjoint union of $\Theta(T,v)$ as $v$ ranges over all the vertices of~$T$, and an \textit{angle assignment} for $T$ is a function $\angle\colon \Theta(T)\to (0,1]$ whose restriction to each $\Theta(T,p)$ is an angle assignment at~$v$.  We refer to $\angle(\theta)$ as the \emph{measure} of the angle $\theta$.  

If $f$ is a post-critically finite topological polynomial and $T$ is a tree in $(\mathbb{R}^2,P_f)$, then we can lift an angle assignment $\angle$ on $T$ to an angle assignment $\angle'$ on $T'=\lambda_f(T)$ using the following procedure.
\begin{enumerate}
\item First we lift $\angle$ to an angle assignment $\widetilde{\angle}$ on $f^{-1}(T)$ defined by
\[
\widetilde{\angle}(e_1,e_2) = \frac{\angle\bigl(f(e_1),f(e_2)\bigr)}{d_f(v)}
\]
for $(e_1,e_2)\in \Theta(f^{-1}(T),v)$,
where $d_f(v)$ is the local degree of $f$ at~$v$.\smallskip
\item Next we restrict $\widetilde{\angle}$ to an angle assignment $\angle'$ on $T'$ as follows.  Observe that each vertex $v$ of $T'$ is also a vertex of~$f^{-1}(T)$, with each angle at $v$ in $T'$ obtained by joining together one or more angles at $v$ in $f^{-1}(T)$.  As such, we define $\angle'$ so that the measure of each angle in $T'$ is the sum of the measures of the corresponding angles in $f^{-1}(T)$.
\end{enumerate}
An example of this procedure for the airplane polynomial is shown in Figure~\ref{fig:AnglesExample}.
\begin{figure}[t]
\centering
$\underset{\textstyle\text{(a)}}{\includegraphics{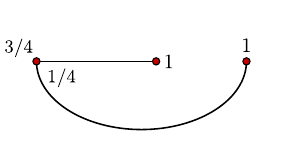}}$
\hfill
$\underset{\textstyle\text{(b)}}{\includegraphics{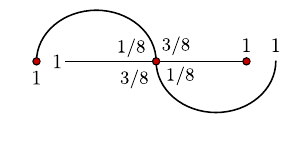}}$
\hfill
$\underset{\textstyle\text{(c)}}{\includegraphics{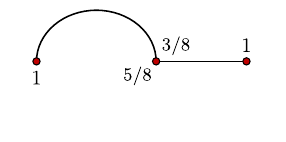}}$
\caption{(a)~An angle assignment $\angle$ for a tree~$T$. (b)~The lift $\widetilde{\angle}$ on $f^{-1}(T)$, where $f$ is the airplane polynomial. (c)~The restriction $\angle'$ of $\widetilde{\angle}$ to~$\lambda_f(T)$}
\label{fig:AnglesExample}
\end{figure}
Given an angle assignment $\angle$ on~$T$, we will refer to the resulting angle assignment $\angle'$ on $\lambda_f(T)$ as the \textit{lift} of $\angle$ to $\lambda_f(T)$, even though it is really a restriction of a lift.

If $T$ is an invariant tree for~$f$, then any angle assignment $\angle$ on $T$ lifts to another angle assignment $\angle'$ on the same~$T$.  Such an angle assignment is said to be \textit{invariant} if~$\angle=\angle'$.

\p{Fatou vertices/edges and Julia vertices/edges}  If $T$ is an invariant tree with dynamical map $f_*\colon T\to T$, a vertex $v$ of $T$ (possibly unmarked) is said to be a \textit{Fatou vertex} if the forward orbit of $v$ contains a periodic critical point of~$f$.  In the case where $T$ is a Hubbard tree for a polynomial map~$f$, these are precisely the vertices of $T$ that lie in the Fatou set of~$f$.  Vertices of $T$ that are not Fatou vertices are called \textit{Julia vertices}.  An edge is a \textit{Julia edge} if both of its vertices are Julia vertices.   An edge that is not a Julia edge is called a \emph{Fatou edge}.   An edge is \textit{periodic} if there exists a $k\geq 1$ such that $f_*^k$ maps $e$ homeomorphically to itself.

\p{The conditions} Let $(f,A)$ be a marked topological polynomial, and let $T$ be an invariant tree for $(f,A)$.  We will refer to the following conditions on $T$ as Poirier's conditions.  

\medskip

\hspace*{3ex}\emph{Angle condition.} There exists an invariant angle assignment for~$T$.\smallskip

\smallskip

\hspace*{3ex}\emph{Expanding condition.} $T$ has no periodic Julia edges.

\medskip

The main result of this subsection is Proposition~\ref{prop:PoirierConditions}, which states that if an invariant tree for $f$ satisfies the Poirier conditions, then it is the Hubbard tree $H_f$; this is a version of a result of Poirier.   

Consider, for example, the invariant trees shown in Figure~\ref{fig:TreesCubicRabbit}.  The tree $H_f$ is the Hubbard tree, and so it satisfies both of Poirier's conditions.  The tree $T$ satisfies the angle condition but not the expanding condition (specifically, the middle edge fails).  The tree $T'$ satisfies the expanding condition (there are no Julia edges) but not the angle condition; see the end of Section~\ref{sec:p2}.  

Our proof of Proposition~\ref{prop:PoirierConditions} requires the following lemma.

\begin{lemma}
\label{lem:poirier}
Let $(f,A)$ be a marked topological polynomial, where $A$ contains the critical points of~$f$.  An invariant tree $T$ for $f$ in $(\mathbb{R}^2,A)$ is a topological Hubbard 
tree for $f$ if and only if it satisfies the angle condition and the expanding condition.
\end{lemma}

\begin{proof}

Poirier proves \cite[Theorem~~1.1]{poirier} that the Hubbard tree for a marked polynomial (where critical points are marked) satisfies the angle condition and the expanding condition, and it follows immediately that the same holds for topological Hubbard trees.

For the converse, suppose that $T$ is an invariant tree for $(f,A)$ that satisfies the angle condition and the expanding condition.  Since the critical points for $f$ are marked, Poirier's theorem tells us that we can realize $T$ as the Hubbard tree for some marked polynomial.  That is, there exists a marked polynomial~$(g,B)$ and a homeomorphism $h\colon (\mathbb{R}^2,f^{-1}(A))\to(\mathbb{R}^2,g^{-1}(B))$ such that:
\begin{enumerate}
    \item $h(T)$ is the Hubbard tree $H_{(g,B)}$ for $(g,B)$, and\smallskip
    \item $h(f^{-1}(T))$ is isotopic to $g^{-1}(H_{(g,B)})$ in $(\R^2,g^{-1}(B))$.
\end{enumerate}
By the Alexander method (Proposition~\ref{prop:alexander}), it follows that $h$ gives a Thurston equivalence from $f$ and $g$, and therefore $T$ is a topological Hubbard tree for~$(f,A)$.
\end{proof}

We would like to apply Poirier's conditions to the case where $A$ is the post-critical set~$P_f$.  Since $P_f$ does not necessarily contain the critical points of~$f$, this requires the following proposition.

\begin{proposition}\label{prop:PoirierConditions} 
Let $f$ be a post-critically finite topological polynomial, and let $T$ be an invariant tree for~$f$. Then $T$ is a topological Hubbard tree for $f$ if and only if it satisfies the angle condition and the expanding condition.
\end{proposition}

\begin{proof}

Let $A = f^{-1}(P_f)$.  Then $(f,A)$ is a marked topological polynomial, and $A$ contains the critical points of~$f$. The tree $T$ is a topological Hubbard tree for $f$ if and only if $f^{-1}(T)$ is a topological Hubbard tree for $(f,A)$, and by Lemma~\ref{lem:poirier} this occurs if and only if $f^{-1}(T)$ satisfies the angle condition and the expanding condition.  Thus it suffices to prove that $T$ satisfies the two conditions if and only if $f^{-1}(T)$ does.

For the angle condition, observe first that any invariant assignment of angles for $f^{-1}(T)$ restricts to an invariant assignment of angles for~$T$ (where each angle of $T$ is associated to an angle of $f^{-1}(T)$ via an isotopy from $T$ to a subtree of $f^{-1}(T)$).  Conversely, given any invariant assignment of angles for $T$ we can take the preimage to obtain an angle assignment for $f^{-1}(T)$, and this is invariant since the lift of an angle assignment on $f^{-1}(T)$ is entirely determined by the restriction of that angle assignment to~$T$. (Since $f^{-1}(T)$ is invariant, every vertex of $f^{-1}(f^{-1}(T))$ that is a vertex of $f^{-1}(T)$ maps to a vertex of~$T$, and we can therefore restrict the angle assignment obtained by lifting from~$T$.)

For the expanding condition, observe that every periodic vertex of $f^{-1}(T)$ is a vertex of~$T$, and therefore any periodic Julia edge of $f^{-1}(T)$ must correspond to an edge of~$T$.  For the converse, if $e$ is a periodic Julia edge in~$T$, then $f_*(e)$ is a single edge in $T$, so $e$ must correspond to an edge of $f^{-1}(T)$, specifically a periodic Julia edge.
\end{proof}

Note that Proposition~\ref{prop:PoirierConditions} does not assume that $f$ is unobstructed.  In the case where $f$ is obstructed, the proposition implies that there cannot exist any invariant tree $T$ for $f$ that satisfies both the angle condition and the expanding condition.


\subsection{Achieving the angle condition}\label{sec:p2}

The goal of this section is to prove the following proposition.

\begin{proposition}\label{prop:ExpandAngleCondition}Let $f$ be a post-critically finite topological polynomial, and let $T$ be an invariant tree for~$f$.  Then there exists an invariant tree $T'$ for $f$ that is a forest expansion of $T$ and satisfies the angle condition.
\end{proposition}

The invariant trees $T$ and $T'$ in Figure~\ref{fig:TreesCubicRabbit} illustrate the proposition.  The tree $T'$ fails the angle condition, and the tree $T$---which is obtained from $T'$ by expanding two edges---does satisfy the angle condition.

Our proof of the proposition will require us to consider two different types of angles.  We begin by discussing these.

\p{Angles at Fatou and Julia vertices}
If $T$ is an invariant tree for a post-critically finite topological polynomial~$f$, then we can partition the angles of $T$ into the disjoint union of two sets
\[
\Theta(T) = \Theta_F(T) \uplus \Theta_J(T)
\]
where $\Theta_F(T)$ and $\Theta_J(T)$ are the sets of angles at the Fatou vertices and Julia vertices, respectively.  Since every vertex in the preimage of a Fatou vertex is a Fatou vertex  and every vertex in the preimage of a Julia vertex is a Julia vertex, we can lift angle assignments separately on these two sets.  Thus a tree $T$ satisfies the angle condition if and only if it has both an invariant angle assignment on its Fatou vertices and an invariant angle assignment on its Julia vertices.

The following lemma is essentially due to Poirier \cite[Section 1]{poirier}.

\begin{lemma}\label{lem:InvariantAnglesJuliaVertices}Let $f$ be a post-critically finite topological polynomial, and let $T$ be an invariant tree for~$f$.  Then $T$ has an invariant angle assignment on its Julia vertices.
\end{lemma}
\begin{proof}Let $f_*\colon T\to T$ be the dynamical map. If $v$ is any vertex of $T$ that is not a critical point, then $f_*$ must be locally one-to-one in a neighborhood of~$v$, and hence the degree of $f_*(v)$ is greater than or equal to the degree of~$v$.  It follows that all of the vertices in each periodic cycle of Julia vertices must have the same degree.  This means that we can construct an invariant angle assignment on the periodic Julia vertices by setting the angle measures at each such vertex to be all equal.  We can then extend the angle assignment to the remaining Julia vertices by lifting. \end{proof}

It follows from Lemma~\ref{lem:InvariantAnglesJuliaVertices} that an invariant tree $T$ satisfies the angle condition if and only if it has an invariant angle assignment on its Fatou vertices.

\p{Non-negative angle assignments}
Recall that an angle assignment $\angle\colon \Theta(T)\to (0,1]$  on a tree $T$ must be \textit{positive}, in the sense that all of the angles are required to have positive measure. A \textit{non-negative angle assignment} $\angle\colon \Theta(T)\to [0,1]$ is similar to an angle assignment, except that angles are allowed to have measure 0.  Note that we still require the angles at each vertex of a non-negative angle assignment to add up to 1.  Our definitions of lifting and invariance generalize to the case of non-negative angle assignments.

\begin{lemma}\label{lem:NonNegativeAnglesExist}Let $f$ be a post-critically finite topological polynomial, and let $T$ be an invariant tree for~$f$.  Then $T$ has an invariant non-negative angle assignment.
\end{lemma}
\begin{proof}Let $\bar{A}(T)\subseteq [0,1]^{\Theta(T)}$ be the set of all non-negative angle assignments for~$T$.  Observe that for each vertex $v$ of $T$ of valence~$k$, the set of possible non-negative angle assignments at $v$ is a closed $(k-1)$-simplex in $[0,1]^{\Theta(T,p)}$.  It follows that $\bar{A}(T)$ is a product of closed simplices, and is therefore homeomorphic to a closed, finite-dimensional ball.  Lifting defines a continuous function $L\colon \bar{A}(T)\to \bar{A}(T)$, so by Brouwer's fixed point theorem $L$ has a fixed point in~$\bar{A}(T)$. 
\end{proof}

The function $L$ in the proof of Lemma~\ref{lem:NonNegativeAnglesExist} is linear, and the fixed point of $L$ is an eigenvector with eigenvalue 1.  So, given Lemma~\ref{lem:NonNegativeAnglesExist}, the fixed point can be computed explicitly by linear algebra.  

Given an invariant tree $T$ with an invariant non-negative angle assignment, our strategy for proving Proposition~\ref{prop:ExpandAngleCondition} will be to modify $T$ to eliminate the angles that are assigned an angle measure of 0 by the map $L$.

\p{Foldings}
Let $T$ be a tree in $(\R^2,P)$ and let $v$ be a vertex of~$T$.  Given any proper subset $S \subset \Theta(T,v)$, the associated \textit{folding} of $T$ along $S$ is the tree obtained by identifying initial segments of pairs of edges adjacent to~$S$, as shown in Figure~\ref{fig:StarExpansion}.  Note that such a folding is actually a forest expansion of $T$ at $v$, with one new vertex-edge pair for each maximal set of consecutive angles in~$S$ (e.g.~$\{\theta_1\}$, $\{\theta_4,\theta_5\}$, and $\{\theta_7\}$ in Figure~\ref{fig:StarExpansion}).  The resulting angles of $\Theta(T',v)$ correspond to the angles of $\Theta(T,v)$ that do not lie in~$S$. 
\begin{figure}
\centering
\raisebox{-0.47\height}{\includegraphics{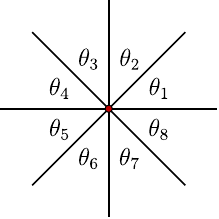}}
$\quad\qquad\longrightarrow\quad\qquad$
\raisebox{-0.47\height}{\includegraphics{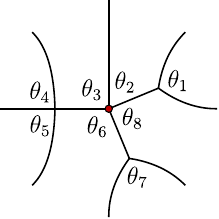}}
\caption{Folding at a vertex of valence 8 along the set~$\{\theta_1,\theta_4,\theta_5,\theta_7\}$}
\label{fig:StarExpansion}
\end{figure}

More generally, given a tree $T$ and a subset $S\subseteq \Theta(T)$ that does not include all the angles at any vertex, the \textit{folding} of $T$ along $S$ is the tree obtained by folding along $S\cap \Theta(T,v)$ at each vertex $v$ of~$T$. Again, note that such a folding is actually a forest expansion of~$T$.  

\begin{lemma}\label{lem:FoldingWorks}Let $f$ be a post-critically finite topological polynomial, and let $T$ be an invariant tree for~$f$. Let $\angle$ be an invariant non-negative angle assignment for~$T$, let
\[
S = \{\theta\in \Theta(T) \mid \angle(\theta)=0\}
\]
and let $T'$ be the folding of $T$ along $S$.  Then $T'$ is an invariant tree for $f$ that satisfies the angle condition.
\end{lemma}
\begin{proof}Let $\widetilde{\angle}$ be the lift of $\angle$ to $f^{-1}(T)$, and let 
\[
\overline{S} = \{\theta\in \Theta(f^{-1}(T)) \mid \widetilde{\angle}(\theta)=0\}.
\]
We have that $f^{-1}(T')$ is isotopic to the folding of $f^{-1}(T)$ along $\overline{S}$.  Since $T$ is an invariant tree for $f$, the convex hull of $P_f$ in $f^{-1}(T)$ is~$T$, so the convex hull of $P_f$ in $f^{-1}(T')$ must be a folding of~$T$.  But since $\angle$ is invariant, the angles of $T$ that are obtained by joining together angles of $\overline{S}$ are precisely the angles of~$S$.  It follows that the convex hull of $P_f$ in $f^{-1}(T')$ is precisely the folding of $T$ along~$S$, and therefore $T'$ is an invariant tree for~$f$.

Now, $T'$ has some ``old'' vertices that come from~$T$, while other vertices of $T'$ are ``new'' vertices that result from the folding.  The original non-negative angle assignment $\angle$ induces a positive angle assignment on each of the old vertices of~$T'$.  In particular, since all of the Fatou vertices of $T'$ are old, we have a positive, invariant angle assignment on all of the Fatou vertices of~$T$.  By Lemma~\ref{lem:InvariantAnglesJuliaVertices}, this extends to an invariant angle assignment on all of~$T$.
\end{proof}

\p{Completing the proof} We are now ready to prove Proposition~\ref{prop:ExpandAngleCondition}, which states that invariant trees have forest expansions with invariant angle assignments.  

\begin{proof}[Proof of Proposition~\ref{prop:ExpandAngleCondition}]By Lemma~\ref{lem:NonNegativeAnglesExist}, there exists a non-negative angle assignment $\angle$ for~$T$.  By Lemma~\ref{lem:FoldingWorks}, we can use $\angle$ to produce a folding $T'$ of $T$ which is invariant for~$f$ and satisfies the angle condition.
\end{proof}

We can illustrate Proposition~\ref{prop:ExpandAngleCondition} with our running example $f(z)=z^3-3z/4+i\sqrt{7}/4$ from Section~\ref{sec:alex}.  Let $T'$ be the invariant tree for $f$ shown in Figure~\ref{fig:TreesCubicRabbit}(c).  An angle assignment for $T'$ is a tuple $(\theta_1,\theta_2,\theta_3,\theta_4,\theta_5,\theta_6)$ of angles as shown in Figure~\ref{fig:InvariantTreeAngles}(a), where
\[
\theta_1=\theta_2+\theta_3=\theta_4+\theta_5=\theta_6=1.
\]
In particular, the polyhedron $\bar{A}(T')\subseteq [0,1]^6$ of non-negative angle assignments is a square. 

The lifting map $\bar{A}(T')\to \bar{A}(T')$ on angles is
\[
(\theta_1,\theta_2,\theta_3,\theta_4,\theta_5,\theta_6) \mapsto \bigl(\theta_6,\theta_4+\tfrac{1}{2}\theta_5,\tfrac{1}{2}\theta_5,\theta_2,\theta_3,\theta_1\bigr)
\]
as shown in parts (b) and (c) of Figure~\ref{fig:InvariantTreeAngles}. This map has a unique fixed point in $\bar{A}(T')$, namely  the point $(\theta_1,\theta_2,\theta_3,\theta_4,\theta_5,\theta_6)=(1,1,0,1,0,1)$.  Since $\theta_3$ and $\theta_5$ are $0$, we can fold along these angles to obtain a new invariant tree that satisfies the angle condition.  Specifically, we obtain the tree $T$ shown in Figure~\ref{fig:TreesCubicRabbit}(b), which satisfies the angle condition since all of its Fatou vertices are leaves.
\begin{figure}[t]
\centering
$\underset{\textstyle\text{(a)}}{\includegraphics{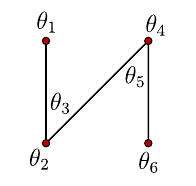}}$
\hfill
$\underset{\textstyle\text{(b)}}{\includegraphics{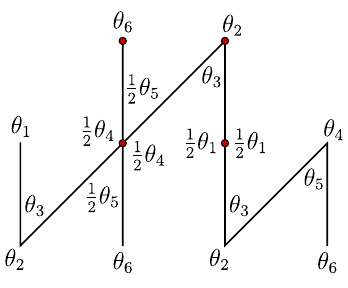}}$
\hfill
$\underset{\textstyle\text{(c)}}{\includegraphics{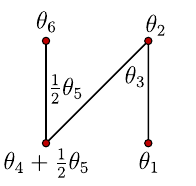}}$
\caption{(a)~An angle assignment for the tree~$T'$. (b)~The lift to $f^{-1}(T')$. (c)~The restriction of the lift to~$T'$}
\label{fig:InvariantTreeAngles}
\end{figure}


\subsection{Achieving the expanding condition}\label{sec:p3}

Our next goal is to prove the following proposition.

\begin{proposition}\label{prop:InvariantWithinTwo}Let $f$ be an unobstructed post-critically finite topological polynomial with topological Hubbard tree $H_f$, and let $T$ be an invariant tree for~$f$ that satisfies the angle condition.  Then $T$ is a forest expansion of~$H_f$.
\end{proposition}

The invariant tree $T$ in Figure~\ref{fig:TreesCubicRabbit} illustrates the proposition.  The tree $T$ fails the expanding condition, and the Hubbard tree $H_f$ is obtained from $T$ by an edge collapse.  

Before proceding to the proof of Proposition~\ref{prop:InvariantWithinTwo}, we begin with some preliminaries.

If $F$ is a subforest of~$T$, its \textit{lift} is the subforest $F'$ of $\lambda_f(T)$ consisting of all edges of $\lambda_f(T)$ that map entirely into $F$ under~$f$.  Assuming $F$ is collapsible, the lift  $F'$ is precisely the subforest of $\lambda_f(T)$ for which $\lambda_f(T)/F' = \lambda_f(T/F)$.  It follows that $F'$ must be collapsible whenever~$F$ is.

If $T$ is an invariant tree for $f$, then a subforest $F$ of $T$ is \textit{invariant} if it is equal to its own lift.  If $F$ is invariant and collapsible, then the quotient $T/F$ is again an invariant tree for~$f$.  Note that a subforest $F$ of $T$ satisfying $f_*(F)\subseteq F$ is not necessarily invariant, since its lift may properly contain~$F$.  In this case, the subforest $F'$ consisting of all edges that eventually map into~$F$ under $f_*$ is invariant, and is the smallest invariant subforest of $T$ that contains~$F$.

\begin{lemma}\label{lem:ContractExpandingCondition}Let $f$ be an unobstructed post-critically finite topological polynomial, and let $T$ be an invariant tree for~$f$.  Let $F$ be the union of all periodic Julia edges of~$T$, and let $F'$ be the smallest invariant subforest of $T$ that contains~$F$.  Then $F'$ is collapsible and the quotient $T/F'$ satisfies the expanding condition.
\end{lemma}
\begin{proof}After possibly modifying $f$ by an isotopy that fixes $P_f$,  we may assume that $f(T)\subseteq T$.   If $e$ is a periodic Julia edge of $T$, then clearly $e$ cannot have any critical points in its interior, and since the endpoints of $e$ are periodic Julia vertices they cannot be critical points either.  It follows that no critical points of $f$ lie in~$F$, so $f_*$ maps $F$ homeomorphically to~$F$, and is one-to-one in a neighborhood of~$F$ in~$\mathbb{R}^2$.

We claim that no connected component of $F$ can have more than one marked point.  For if $K_1$ were a component of $F$ with more than one marked point, then $K_1$ would be part of a periodic cycle $K_1,\ldots,K_m$ of connected components of~$F$, each of which would have more than one marked point.  Then the simple closed curves $c_1,\ldots,c_m$ that surround $K_1,\ldots,K_m$ are all essential and therefore form a Levy cycle for~$f$, a contradiction since $f$ is unobstructed.

We conclude that $F$ is a collapsible subforest of~$T$.  Since $F'$ can be obtained from $F$ by lifting finitely many times, it follows that $F'$ is collapsible as well.  All that remains is to prove that $T/F'$ has no periodic Julia edges.

Note first that, since $f$ is one-to-one in a neighborhood of~$F$, no edge in the complement of $F$ that is incident to a vertex of $F$ can map into~$F$, so such edges do not lie in~$F'$.  It follows that each connected component of $F$ is also a connected component of~$F'$, and indeed these are precisely the connected components of $F'$ that are periodic under~$f$.

Now suppose to the contrary that $e$ is an edge of $T$ not in $F'$ whose image $e'$ in $T/F'$ is a periodic Julia edge.  Let $f_*'$ be the dynamical map on $T/F'$, and let $k\geq 1$ so that $(f_*')^k$ maps $e'$ to itself in an orientation-preserving fashion.  Then $f_*^k$ maps $e$ to a path of the form $\alpha e \beta$ in~$T$, where each of $\alpha$ and $\beta$ is either trivial or a path in~$F'$. Note that $e$ must be a Julia edge in $T$ since Fatou vertices in $T$ map to Fatou vertices in $T/F'$.  But we know that $e$ is not a periodic Julia edge, so either $\alpha$ or $\beta$ must be nontrivial, say~$\alpha$.  Then the endpoint $v$ of $e$ at which $\alpha$ meets $e$ must lie in~$F'$.  Indeed, since $\alpha$ connects $v$ and $f^k(v)$, the component of $F'$ that contains~$v$ is periodic under~$f$, and is therefore a component of~$F$.  But $f$ is one-to-one in a neighborhood of~$F$, so the initial segment of $e$ near $v$ cannot map into $F$ under~$f^k$, a contradiction.
\end{proof}

\begin{proof}[Proof of Proposition~\ref{prop:InvariantWithinTwo}] By Lemma~\ref{lem:ContractExpandingCondition}, there exists a collapsible invariant subforest $F'$ of $T$ consisting entirely of Julia edges so that $T/F'$ satisfies the expanding condition.  Since $F'$ does not contain any of the Fatou vertices, the invariant angle assignment on the Fatou vertices of~$T$ descends to an invariant angle assignment on the Fatou vertices of $T/F'$. By Lemma~\ref{lem:InvariantAnglesJuliaVertices}, it follows that $T/F'$ satisfies the angle condition, so by Proposition~\ref{prop:PoirierConditions} the quotient $T/F'$ must be the topological Hubbard tree for~$f$.
\end{proof}

We can illustrate Proposition~\ref{prop:InvariantWithinTwo} with our running example $f(z)=z^3-3z/4+i\sqrt{7}/4$ from Section~\ref{sec:alex}. Let $T$ be the invariant tree $T$ for the polynomial $f$ shown in Figure~\ref{fig:TreesCubicRabbit}(b).  Since all of the Fatou vertices of $T$ are leaves, this tree satisfies the angle condition. However, it does not satisfy the expanding condition since the middle edge is a Julia edge that maps to itself.  The subforest $F$ consisting of this edge is an invariant subforest, and the quotient $T/F$ is the Hubbard tree for $f$ shown in Figure~\ref{fig:TreesCubicRabbit}(a).


\subsection{Proof of the first statement of Theorem~\ref{thm:main}}
\label{sec:p4}
We are now ready to prove the first statement of our main theorem.  

\begin{proof}[Proof of Theorem~\ref{thm:main}(1)]

As in the statement, let $f$ be an unobstructed post-critically finite topological polynomial and let $n = |P_f|$.

We claim that every vertex $T$ of $\T_n$ is either periodic or pre-periodic under the action of~$\lambda_f$.  In other words, for each vertex $T$ of $\T_n$, there is an $m \geq 0$ and a $r > 0$ so that
\[
\lambda_f^m(T) = \lambda_f^{m+r}(T).
\]
(Note that $m$ and $r$ depend on $T$.)  Since the action of $\lambda_f$ is simplicial (see Section~\ref{sec:lifting}) it never increases the distance between any two vertices.  In particular, since the Hubbard vertex $H_f$ is fixed under~$\lambda_f$, the ball of any finite radius around $H_f$ must map into itself.  Since $\T_n$ is locally finite, any such ball has only finitely many vertices.  Thus, using the fact that $\T_n$ is connected, we see that the orbit of any vertex $T$ of $\T_n$ under $\lambda_f$ must eventually repeat, whence the claim.

Now, given any vertex $T$ of $\T_n$, let $m$ and $r$ be as above.  Then $T'=\lambda_f^m(T)$ is fixed under~$\lambda_f^r$, so $T'$ is an invariant tree for $f^r$.  By Proposition~\ref{prop:ExpandAngleCondition}, there is a vertex $T''$ that has distance at most 1 from $T'$ and satisfies the angle condition for $f^r$.  Then by Proposition~\ref{prop:InvariantWithinTwo} the Hubbard vertex for $f^r$ has distance at most 1 from $T''$.  Since the Hubbard vertices for $f$ and $f^r$ are the same, the result follows. \end{proof}


\section{Finding Hubbard vertices for obstructed maps}
\label{sec:obs}

In this section we prove the second statement of Theorem~\ref{thm:main}, which says that if $f$ is an obstructed  post-critically finite topological polynomial with $|P_f|=n$ and $T$ is any vertex of~$\T_n$, then for all sufficiently large $k$ the vertex $\lambda_f^k(T)$ has distance at most 1 from the Levy set $L_f \subseteq \T_n$.

We begin in Section~\ref{sec:pilgrimselinger} by giving Pilgrim's original definition of the canonical obstruction (for an arbitrary post-critically finite branched self-cover of the sphere) and then describing Selinger's topological characterization of it.  In Section~\ref{sec:canonicalstructure} we prove some basic structural properties of the canonical obstruction, in the case of a topological polynomial; the main result of this section is Proposition~\ref{prop:CanonicalObstructionStructure}.  Using this proposition, we define in Section~\ref{sec:hubbardbubble} the Hubbard vertex for an obstructed topological polynomial, which lies in $\hat \T_n$.  We also give a version of the Poirier conditions for Hubbard vertices.  In Section~\ref{sec:canonical form} we describe and prove the existence of our normal form for topological polynomials, which is used in Section~\ref{sec:rabbit}.  Finally, in Section~\ref{sec:pfobs} we give the proof of the second statement of Theorem~\ref{thm:main}.  In the proof the Hubbard vertex plays the same role for obstructed maps that the analogous Hubbard vertex does for unobstructed maps.   At the end of Section~\ref{sec:pfobs} we explain two refinements of Theorem~\ref{thm:main} that are afforded by the proof.


\subsection{Canonical obstructions}
\label{sec:pilgrimselinger}

We briefly discuss here Thurston's theory of Thurston obstructions, Pilgrim's theory of canonical obstructions, and Selinger's characterization of the canonical obstruction.

\p{Thurston obstructions} Let $f : S^2 \to S^2$ be a post-critically finite branched covering map.  Given a multicurve $M$ in $(\R^2, P_f$), the \textit{lift} $\lambda_f(M)$ of $M$ is the multicurve obtained from~$f^{-1}(M)$ by deleting all inessential curves and deleting all but one curve from each isotopy class (this multicurve is well defined up to isotopy).  A multicurve $M$ is \textit{stable} under $f$ if every curve of $\lambda_f(M)$ is isotopic to a curve of~$M$, and is \textit{invariant} under $f$ if $\lambda_f(M)$ is isotopic to~$M$ (so any invariant multicurve is stable). Note that any invariant multicurve is stable, as is any multicurve whose lift is empty.

Thurston proved that $f$ fails to be Thurston equivalent to a rational map if and only if it has a stable multicurve that satisfies a certain contracting condition~\cite{DH,thurston1988}.  Such a multicurve is known as a \textit{Thurston obstruction} for~$f$.  In the case where $f$ is a post-critically finite topological polynomial, a Thurston obstruction for $f$ is precisely a stable multicurve in $\R^2\setminus P_f$ that contains a Levy cycle~\cite[Theorem 5.5]{BFH}.

\p{Canonical obstructions} The theory of canonical obstructions was initiated by Pilgrim \cite{pilgrim}.  In his work, the \textit{canonical obstruction} $\Gamma_f$ of post-critically finite branched cover $f\colon S^2\to S^2$ is defined to be the multicurve consisting of one representative from each isotopy class of essential simple closed curves in $S^2\setminus P_f$ whose geodesic length tends to $0$ in the sequence of hyperbolic metrics on $S^2 \setminus P_f$ obtained under iteration of Thurston's pullback map on Teichm\"uller space. This collection is a Thurston obstruction for~$f$, and is uniquely determined up to isotopy.

\p{Selinger's characterization} Selinger~\cite{selinger13} gave a purely topological description of the canonical obstruction.  For convenience, we state it only for the case of an obstructed topological polynomial~$f$.

If $M$ is a multicurve in $(\R^2,P_f)$, the \textit{complementary components} of $M$ are the connected components of $\R^2\setminus(M\cup P_f)$.  Each such component is homeomorphic to a sphere with punctures; these correspond precisely to the spheres of the noded surface obtained from $(\R^2,P_f)$ by contracting the curves of $M$ to points.  A complementary component $C$ is said to be \textit{periodic} if there exists a $k\geq 1$ and a finite set $Q\subseteq C$ so that $C\setminus Q$ is isotopic to a component $\widetilde{C}$ of $f^{-k}(C)$ in~$\R^2\setminus P_f$ (the elements of $Q$ correspond to points of $C$ that map to $P_f$ and to closed disks in $C$ that map to other complementary components).  The smallest such $k$ is the \textit{period} of~$C$.

Let $C$ be a periodic complementary component of a multicurve $M$ with period~$k$, with $\widetilde{C}$ the corresponding component of $f^{-k}(C)$, and with $Q$ the required finite set in $C$ as above.  We may restrict $f^k$ to obtain a covering map $\widetilde{C} \to C \setminus Q$.  After identifying $\widetilde{C}$ with $C \setminus Q$ via an ambient isotopy of $(\R^2,P_f)$, we may consider the map $\widetilde{C} \to C$ as a covering map $C \setminus Q \to C$.  By filling in all punctures in both $C \setminus Q$ and $C$ we obtain a map $f_C : S^2 \to S^2$ called the \textit{first return map} for~$C$.

Selinger's characterization of canonical obstructions for post-critically finite branched covers of $S^2$ reduces to the following statement for topological polynomials: a Thurston obstruction $M$ for a post-critically finite topological polynomial $f$ is the canonical obstruction $\Gamma_f$ if and only if it satisfies the following conditions:
\begin{enumerate}
    \item The first return map for each periodic complementary component of $M$ is either a homeomorphism or an unobstructed topological polynomial.\smallskip
    \item No Thurston obstruction $M'$ for $f$ that is properly contained in $M$ satisfies condition~(1).
\end{enumerate}
Note that the exterior complementary component $E$ of a stable multicurve $M$ is periodic under (and in fact preserved by) $f$.  If $M$ is un-nested, then $E$ can be viewed as a sphere with one puncture at $\infty$ and one puncture for each curve of~$M$, and the corresponding first return map $f_E$ is a topological polynomial.  Observe that $f_E$ maps punctures to punctures.  Specifically, $f_E$ maps the puncture corresponding to a curve $c$ in $M$ to the puncture corresponding to a curve $c'$ in $M$ if and only if $f^{-1}(c')$ has a component isotopic to~$c$.


\subsection{Structure of the canonical obstruction}
\label{sec:canonicalstructure}

In this section we use Selinger's characterization of the canonical obstruction in order to give a description of the canonical obstruction in terms of Levy cycles, Proposition~\ref{prop:CanonicalObstructionStructure} below.  

In order to state Proposition~\ref{prop:CanonicalObstructionStructure}, we require several definitions.  If $f$ is an obstructed topological polynomial, then we can view any Levy cycle $(c_1,\ldots,c_k)$ for $f$ as a multicurve in $(\R^2,P_f)$.  The curves of a Levy cycle all surround the same number of marked points, and therefore a Levy cycle is an un-nested multicurve.  More generally, a \textit{Levy multicycle} for $f$ is any multicurve in $(\R^2,P_f)$ that can be expressed as a union of Levy cycles. 

If $L$ is a Levy multicycle for~$f$, the \textit{invariant closure} of $L$ is the multicurve obtained from the set of iterated preimages of curves in~$L$, where an \textit{iterated preimage} of a curve $c$ is any essential component of $f^{-k}(c)$ for any $k\geq 1$; so that the result satisfies the definition of a multicurve, we take only one curve from each isotopy class.  Note that the invariant closure of $L$ contains~$L$, since each curve of $L$ is an iterated preimage of itself.

Central to our analysis going forward will be the Levy multicycles that are un-nested multicurves.  We place a partial order on the set of un-nested Levy multicycles for an obstructed topological polynomial $f$ as follows.  Given un-nested Levy multicycles $L$ and $L'$ for~$f$, we write $L\preceq L'$ if for each curve $c$ of $L$ there exists a curve $c'$ of $L'$ such that $c$ is isotopic to a curve that lies in the interior of~$c'$. We say that an un-nested Levy multicycle $L$ for $f$ is \textit{outermost} if it is maximal with respect to~$\preceq$.

\begin{proposition}\label{prop:CanonicalObstructionStructure}
Let $f$ be an obstructed post-critically finite topological polynomial. 
\begin{enumerate}
    \item The map $f$ has a unique outermost un-nested Levy multicycle~$L$.
    \item The invariant closure of $L$ is the canonical obstruction for~$f$. 
    \item The canonical obstruction for $f$ is an un-nested multicurve.
\end{enumerate}
\end{proposition}

In order to prove Proposition~\ref{prop:CanonicalObstructionStructure} we first prove three preliminary results.

\begin{proposition}\label{prop:LevyCycleStructure}
Let $L$ be an un-nested Levy multicycle for an obstructed topological polynomial~$f$.  Then the invariant closure of $L$ is an invariant, un-nested Thurston obstruction for~$f$.
\end{proposition}

\begin{proof}Let $M_0=L$, and for each $n\geq 1$ let $M_n=\lambda_f(M_{n-1})$ be the the lift of $M_{n-1}$ under~$f$.  Since $L$ is an un-nested Levy multicycle, each curve of $M_0$ is isotopic to a curve of~$M_1$, so without loss of generality we may assume that $M_0\subseteq M_1$, and indeed that $M_{n-1}\subseteq M_n$ for all $n\geq 1$. As such, the union $M=\bigcup_{n\geq 0} M_n$ is precisely the invariant closure of~$L$.  But a multicurve in $(\R^2,P_f)$ consists of at most $|P_f|-2$ curves, so the sequence \[M_0\subseteq M_1\subseteq M_2\subseteq \cdots\] eventually stabilizes and therefore $M = M_N$ for some sufficiently large~$N$.

Now, since $M_0$ is un-nested and the lift of an un-nested multicurve is un-nested, each of the multicurves $M_n$ is un-nested, and hence $M$ is as well. Note also that $M$ is invariant and hence stable.  Since $M$ contains each of the Levy cycles of~$L$, it is a Thurston obstruction for~$f$.
\end{proof}

We will need some basic information about the structure of un-nested Levy multicycles.

\begin{proposition}\label{prop:LevyCycleProperties}
Let $L$ be an un-nested Levy multicycle for an obstructed topological polynomial~$f$.  Then:
\begin{enumerate}
    \item $L$ can be expressed uniquely as a disjoint union of Levy cycles.\smallskip
    \item If $c$ is a curve of the invariant closure of~$L$, then $c$ is isotopic to a curve of $L$ if and only if some iterated preimage of $c$ is isotopic to~$c$.
\end{enumerate}
\end{proposition}

\begin{proof}Let $M$ be the invariant closure of $L$, and consider the directed graph with one vertex for each component of~$M$, with a directed edge from $c$ to $c'$ if $f^{-1}(c')$ has a component isotopic to~$c$.  By Proposition~\ref{prop:LevyCycleStructure}, $M$ is un-nested.  Therefore, no two edges in the graph have the same initial vertex.  It follows that the graph consists of finitely many directed cycles (the Levy cycles of~$L$) together with finitely many directed trees that feed into the cycles.  Statements (1) and (2) follow immediately.
\end{proof}

\begin{proposition}\label{prop:ExteriorComponentCheck}
Let $f$ be an obstructed topological polynomial, let $L$ be an un-nested Levy multicycle for~$f$, and let $M$ be the invariant closure of~$L$.  Then $M$ is the canonical obstruction for $f$ if and only if the first return map for the exterior complementary component of~$M$ is unobstructed.
\end{proposition}
\begin{proof}
The forward direction follows immediately from Selinger's characterization.  For the converse, suppose that the first return map for the exterior complementary component is unobstructed.  We know from Proposition~\ref{prop:LevyCycleProperties}(2) that the only bounded periodic complementary components of $M$ are the the interiors of the curves of~$L$, and since $L$ is a union of Levy cycles the first return map for each such component is a homeomorphism (if the degree of the first return map were greater than 1, then the restriction of $f$ to some curve of $L$ would not have degree 1, and so $L$ would not be a Levy multicycle). Thus $M$ satisfies condition (1) of Selinger's characterization.

For condition (2) of Selinger's characterization, let $M'$ be any $f$-stable multicurve that is properly contained in~$M$.  Then $M'$ cannot contain all of~$L$, and indeed there must be a Levy cycle $L_0\subseteq L$ whose curves are completely absent from~$M'$.  Then $L_0$ is a Levy cycle for the first return map corresponding to the exterior complementary component of~$M'$, so $M'$ fails to satisfy condition~(1) of Selinger's characterization.  We conclude that $M$ satisfies condition~(2), so $M$ is the canonical obstruction.
\end{proof}

\begin{lemma}\label{lem:DisjointLevyCycles}
Let $L$ and $L'$ be distinct Levy cycles for an obstructed topological polynomial~$f$. If $L\cup L'$ is a multicurve, then exactly one of the following holds: $L\preceq L'$, $L'\preceq L$, or $L\cup L'$ is un-nested.  Moreover, these three cases are mutually exclusive.
\end{lemma}

We emphasize that Lemma~\ref{lem:DisjointLevyCycles} is only stated for Levy cycles, not for Levy multicycles.

\begin{proof}[Proof of Lemma~\ref{lem:DisjointLevyCycles}] Suppose $L\cup L'$ is not un-nested, so without loss of generality some curve $c$ of $L$ lies in the interior of some curve $c'$ of~$L'$.  Say that $L'$ consists of the curves $c'=c_0',\dots,c_{k-1}'$, in that order.  Since $L'$ is a Levy cycle we have for every $i \geq 0$ a unique component $c_i$ of $f^{-i}(c)$ that lies in the interior of $c_j'$, where $j \equiv i \mod k$, and moreover each $c_i$ maps to $c_{i-1}$ with degree 1.  As the invariant closure of $c$ is finite, the set of $c_i$ forms a Levy cycle $L''$.  By construction, $L'' \preceq L'$.   By Proposition~\ref{prop:LevyCycleProperties}(1) the curve $c$ can lie in at most one Levy cycle, and so $L''$ is equal to $L$.  The lemma follows.
\end{proof}

We are now ready to prove our main characterization of the canonical Thurston obstruction.

\begin{proof}[Proof of Proposition~\ref{prop:CanonicalObstructionStructure}]
Note first that if $L_1\preceq L_2$ are distinct un-nested Levy multicycles for~$f$, then either
\begin{enumerate}
    \item the curves of $L_2$ enclose more marked points than the curves of~$L_1$, or\smallskip
    \item $L_2$ has fewer curves than~$L_1$.
\end{enumerate}
It follows that there are no infinite ascending chains $L_1\preceq L_2\preceq L_3\preceq \cdots$ of distinct un-nested Levy multicycles, so $f$ must have at least one un-nested Levy multicycle $L$ that is maximal under~$\preceq$.  

Let $M$ be the invariant closure of~$L$.  We claim that $M$ is the canonical obstruction for~$f$.  The uniqueness of $L$ follows from this, since by Proposition~\ref{prop:LevyCycleProperties}(2) we can recover $L$ from $M$ by taking the components of $M$ that have an iterated preimage isotopic to themselves.

By Proposition~\ref{prop:ExteriorComponentCheck}, it suffices to prove that the first return map $f_E$ for the exterior complementary component $E$ of $M$ is unobstructed.
Suppose to the contrary that the first return map for the exterior component has a Levy cycle~$L_E$ in~$E$.  Then $L_E$ is also a Levy cycle for $f$,  and it does not intersect any of the curves of~$L$.  Moreover, since the curves of $L_E$ are essential in~$E$, no curve of $L_E$ can be isotopic to a curve of~$L$ in $(\R^2, P_f)$.  By Proposition~\ref{prop:LevyCycleProperties}(1), we can express $L$ as a disjoint union of Levy cycles. Let $L'$ be the multicurve obtained by taking the union of $L_E$ together with all Levy cycles of $L$ that do not fit inside of~$L_E$.   By Lemma~\ref{lem:DisjointLevyCycles} the multicurve $L'$ is un-nested, and is therefore an un-nested Levy multicycle.  But $L'$ is distinct from $L$ and $L\preceq L'$, which contradicts the maximality of~$L$.  We conclude that $f_E$ is unobstructed, so $M$ must be the canonical obstruction.
\end{proof}


\subsection{The Hubbard vertex}
\label{sec:hubbardbubble}

The goal of this section is to define the Hubbard vertex for an obstructed topological polynomial and describe some basic properties of it that will be used in the proof of the second statement of Theorem~\ref{thm:main}.  Specifically, we prove Proposition~\ref{prop:InvariantIterates}, which states all iterates of a given obstructed topological polynomial have the same Hubbard vertex.  We then prove Proposition~\ref{prop:PoirierConditionsBubbleTree}, which is an analogue of Proposition~\ref{prop:PoirierConditions} in that it gives a variant of Poirier's conditions for obstructed maps.  Finally, we prove Lemma~\ref{lem:FindingHubbardVertex}, which is an analogue of Lemma~\ref{lem:ContractExpandingCondition}.

We begin with the definition of the Hubbard vertex.  If $f$ is an obstructed topological polynomial, the canonical obstruction~$\Gamma_f$ is un-nested by Proposition~\ref{prop:CanonicalObstructionStructure}. Moreover, the first return map $f_E$ for the exterior component is an unobstructed topological polynomial by Selinger's characterization of the canonical obstruction, so $f_E$ has a topological Hubbard tree~$T$.
As in Section~2, the pair $(\Gamma_f,T)$ specifies a bubble tree $B_f$ in $(\mathbb{R}^2,P_f)$, which we refer to as the \textit{Hubbard bubble tree} for~$f$. (Like the topological Hubbard tree, the Hubbard bubble tree is only defined up to isotopy.)  The associated vertex of the augmented complex $\hat{\T}_n$ (where $n=|P_f|$) is the \textit{Hubbard vertex} for $f$, which we denote~$H_f$.


\begin{proposition}\label{prop:InvariantIterates}If $f$ is an obstructed topological polynomial, the Hubbard bubble tree for any iterate of $f$ is the same as the Hubbard bubble tree for $f$.
\end{proposition}
\begin{proof}
Selinger proved that the canonical obstruction for any iterate $f^k$ of $f$ is the same as the canonical obstruction for $f$~\cite[Proposition 3.3]{selinger13}.  Furthermore, if $f_E$ is the first return map for $f$ on the exterior complementary component of~$\Gamma_f$, then $f_E^k$ is (up to isotopy) the first return map for~$f^k$ on that component.  Since the topological Hubbard tree $T$ for $f_E^k$ is the same as the topological Hubbard tree for~$f_E$, the Hubbard bubble trees for $f$ and $f^k$ are the same.
\end{proof}

Our next goal is to establish an analogue of Poirier's conditions for Hubbard bubble trees and prove that we can get from an arbitrary invariant tree to the Hubbard vertex through a forest expansion followed by a forest collapse.  The general outline is the same as for the unobstructed case in Section~\ref{sec:poirier}.

Let $f$ be a topological polynomial.  We say that a bubble tree $B$  is \textit{invariant} under $f$ if $\lambda_f(B)=B$.  If we specify $B$ by a pair $(M,T)$, then $(M,T)$ is invariant if the multicurve $M$ is invariant under $f$ and $T$ is an invariant tree for the first return map $f_E$ on the exterior complementary component of~$M$.  For example, the Hubbard bubble tree for an obstructed topological polynomial~$f$ is an invariant bubble tree.  

If $B$ is a bubble tree with multicurve $M$, then as above we refer to the curves of $M$ as bubbles.  In the case where $B$ is invariant under~$f$, we say that a bubble $c$ is \textit{periodic} if the corresponding marked point is periodic under the first return map for the exterior, i.e.~if some component of some iterated preimage of $c$ is isotopic to $c$.  We say that a bubble $c$ is \textit{critical} if the corresponding puncture is critical under the first return map, i.e.~if the curve of $f^{-1}(M)$ isotopic to $c$ maps to its image with degree two or greater.

As with invariant trees, every invariant bubble tree $B$ with exterior tree $T$ has a \textit{dynamical map} $f_*\colon T\to T$, namely the dynamical map on $T$ determined by~$f_E$.  We say that $B$ satisfies the \textit{angle condition} and \textit{expanding condition}, respectively, if $T$ satisfies the corresponding condition with respect to~$(f_E)_*$.

\begin{proposition}\label{prop:PoirierConditionsBubbleTree}
Let $f$ be a post-critically finite topological polynomial. If $f$ is obstructed and $B$ is the Hubbard bubble tree for $f$, then:
\begin{enumerate}
    \item $B$ satisfies the angle condition and the expanding condition, and\smallskip
    \item no periodic bubble of $B$ is critical.
\end{enumerate}
Conversely, if $f$ has an invariant bubble tree $B$ that satisfies conditions (1) and (2) above, then $f$ is obstructed and $B$ is the Hubbard bubble tree for~$f$.
\end{proposition}
\begin{proof}

Let $B$ be an invariant bubble tree for $f$.  Let $M$ be the multicurve of $B$, let $T_E$ be the exterior tree, and let $f_E$ be the first return map for the exterior complementary component of~$M$. 

Suppose first that $B$ is the Hubbard bubble tree for~$f$.  Then $T_E$ is a topological Hubbard tree for~$f_E$, so $T_E$ satisfies the angle and expanding conditions by Proposition~\ref{prop:PoirierConditions}. Furthermore, since $M$ is the canonical obstruction, by Proposition~\ref{prop:CanonicalObstructionStructure} it is the invariant closure of the outermost Levy cycle for~$f$. Then every periodic bubble of $M$ must be part of a Levy cycle by Proposition~\ref{prop:LevyCycleProperties}(2), and therefore no periodic bubble of $M$ is critical.

For the converse, suppose that $B$ satisfies the two conditions in the statement of the proposition, and let $L$ be the set of all periodic bubbles in~$M$.  Since no bubble of $L$ is critical and $L$ is un-nested, it is an un-nested Levy multicycle, and $M$ is the invariant closure of~$L$. Note then that $f$ is obstructed.  Moreover, since $B$ satisfies condition~(1), it follows from Proposition~\ref{prop:PoirierConditions} that $T_E$ is a topological Hubbard tree for~$f_E$, and therefore $f_E$ is unobstructed.  By Proposition~\ref{prop:ExteriorComponentCheck}, we conclude that $M$ is the canonical obstruction for~$f$, and therefore $B$ is the Hubbard bubble tree for~$f$.
\end{proof}

If $T$ is any tree and $F'$ is a subforest of $T$ that is not collapsible (meaning that some component of $F'$ contains more than one marked point), then the quotient $T/F'$ may be regarded as a bubble tree, with one bubble for each connected component of $F'$ that has more than one marked point.  Say that $F'$ is \emph{invariant} under $f$ if it is equal to its lift under $f$.  If $T$ is an invariant tree for a post-critically finite topological polynomial~$f$, then $T/F'$ corresponds to an invariant bubble tree for~$f$ if and only if $F'$ is an invariant subforest of~$T$.

\begin{lemma}\label{lem:FindingHubbardVertex}Let $f$ be an obstructed post-critically finite topological polynomial, and let $T$ be an invariant tree for $f$ that satisfies the angle condition. Let $F$ be the union of all periodic Julia edges of~$T$, and let $F'$ be the smallest invariant subforest of $T$ containing $F$.  Then $T/F'$ is the Hubbard vertex for~$f$.
\end{lemma}

\begin{proof}

Let $T/F'$ be the bubble tree obtained from $T$ by collapsing the subforest~$F'$, so that the bubbles of $T/F'$ correspond precisely to the simple closed curves that surround components of $F'$ that have more than one marked point. As in the proof of Lemma~\ref{lem:ContractExpandingCondition}, no connected component of $F$ has any critical points, and the connected components of $F$ are precisely the periodic connected components of~$F'$. It follows that $T/F'$ has no critical periodic bubbles.  As in the proof of Lemma~\ref{lem:ContractExpandingCondition}, the bubble tree $T/F'$ satisfies the expanding condition, and as in the proof of Proposition~\ref{prop:ExpandAngleCondition} it also satisfies the angle condition, so $T/F'$ is the Hubbard vertex by Proposition~\ref{prop:PoirierConditionsBubbleTree}.\end{proof}

We remark that unobstructed topological polynomials also may have invariant bubble trees that are not trees.  For instance, the tuning of the basilica polynomial with itself has an invariant bubble tree; condition (2) of Proposition~\ref{prop:PoirierConditionsBubbleTree} fails for this tree.


\subsection{Canonical form}\label{sec:canonical form} In this section we explain how to use the Hubbard bubble tree to give a complete topological description of an obstructed topological polynomial.  We will refer to this description as the canonical form for an obstructed map.  The canonical form will not be used in our proof of Theorem~\ref{thm:main}; rather, it will be used in our discussion of twisted $z^2+i$ problems in Sections~\ref{sec:z2+i} and~\ref{sec:genz2+i}.  

Let $f$ be an obstructed topological polynomial with post-critical set~$P = P_f$.  Then the Hubbard bubble tree $B$ for $f$ is a union $B_E\cup M$, where $B_E$ is the exterior forest and $M$ is the multicurve of bubbles (which must be isotopic to the canonical obstruction~$\Gamma_f$).  The preimage $\widetilde{B} = f^{-1}(B)$ is a bubble tree in $(\mathbb{R}^2,f^{-1}(P_f))$ with exterior forest $\widetilde{B}_E=f^{-1}(B_E)$ and multicurve $\widetilde{M}=f^{-1}(M)$.

Let $E$ be the exterior region for $M$, and let $\widetilde{E} = f^{-1}(E)$ be the exterior region for $\widetilde{M}$.  Both $E$ and $\widetilde E$ may be regarded as copies of $\R^2$, each with finitely many punctures corresponding to the bubbles.   We can view $E$ as having marked set $P_{f,E}=P_f\cap E$ and $\widetilde{E}$ as having marked set $\widetilde{P}_{f,E} = f^{-1}(P_{f,E})$.  We will refer to the restriction
\[
f_E\colon (\widetilde{E},\widetilde{P}_{f,E})\to (E,P_{f,E})
\]
of $f$ as the \textit{exterior map} for~$f$.  

Let $\Delta=\mathbb{R}^2\setminus E$, so $\Delta$ is a union of closed disks with $\partial\Delta = M$. Let $\widetilde{\Delta}=f^{-1}(\Delta)$, which is also a union of closed disks with $\partial\widetilde{\Delta}=\widetilde{M}$.  We can view $\Delta$ as having marked set $Q = Q_{\mathrm{int}}\cup Q_{\mathrm{bd}}$, where $Q_{\mathrm{int}} = P_f\cap \Delta$ is the set of post-critical points that lie inside the bubbles and $Q_{\mathrm{bd}} = B_E\cap M$ is the set points at which the bubbles intersect the exterior forest. Similarly, we can view $\widetilde{\Delta}$ as having marked set $\widetilde{Q}=f^{-1}(Q)$.  We will refer to the restriction
\[
f_I \colon (\widetilde{\Delta},\widetilde{Q})\to (\Delta,Q)
\]
as the \textit{interior map} for $f$.

We may regard $f_I$ as a collection of maps between closed disks.  Specifically, if $\widetilde{\Delta}_1,\ldots,\widetilde{\Delta}_m$ are the components of~$\widetilde{\Delta}$ and $\Delta_1,\ldots,\Delta_n$ are the components of~$\Delta$, then $f_I$ maps each $\widetilde{\Delta}_i$ to some~$\Delta_j$, so we can decompose $f_I$ into a collection of maps
\[
f_{ij}\colon (\widetilde{\Delta}_i,\widetilde{Q}\cap \widetilde{\Delta}_i)\to (\Delta_j,Q\cap \Delta_j).
\]
Each $f_{ij}$ is either a homeomorphism or a branched cover whose critical values are marked.  In particular, since no critical bubble of a Hubbard bubble tree is periodic, $f_{ij}$ must be a homeomorphism whenever $\widetilde{\Delta}_i$ is periodic under~$f$.

The pair of trees $\widetilde B_E$ and $B_E$ in $\widetilde E$ and $E$ can be extended in an arbitrary way to trees in $(\R^2,f^{-1}(P))$ and $(\R^2,P)$.  We therefore have the following consequence of the Alexander method.  

\begin{proposition}
\label{prop:canonical form}
Let $f$ be a post-critically finite topological polynomial, and let $f_E$ and $f_I$ be the exterior and interior maps for $f$.  Then $f$ is determined up to homotopy relative to~$f^{-1}(P_f)$ by the homotopy classes of the maps $f_E$ and $f_I$ relative to their respective sets of marked points.  
\end{proposition}

We refer to the pair $(f_E,f_I)$ as the \emph{canonical form} for $f$.  As in the statement of the proposition, we consider $f_E$ and $f_I$ to be defined up to homotopy relative to their respective sets of marked points.

We emphasize that $f_E$ maps the punctured surface $E$ to the punctured surface $\widetilde{E}$, and these punctures of course cannot move during homotopies.  Replacing these punctures with marked points would yield the first return map for the exterior component discussed earlier.  

The canonical form presented here is analogous to the Nielsen--Thurston normal form for mapping class groups; see \cite[Corollary 13.3]{primer}.  In general, the Nielsen--Thurston normal form is not canonical (for instance if $f=f_1f_2$, where $f_1$ and $f_2$ are supported on disjoint subsurfaces, each with a boundary component homotopic to the curve $c$, then we may also write $f$ as $f=(f_1T_c)(T_c^{-1}f_2)$).  On the other hand, the maps $f_E$ and $f_I$ in Proposition~\ref{prop:canonical form} are canonical.  A similar canonical form exists for braid groups; see the paper by the third author with Chen and Kordek \cite[Section 6]{CKM}.

\p{Infinitely many Levy cycles}
As a sample application of our canonical form, we give here an example of an obstructed topological polynomial $f$ that has infinitely many different Levy cycles. One way to accomplish this is to arrange for the Hubbard bubble tree to have a fixed bubble with at least three marked points whose interior maps to itself by the identity.  Then the interior of this fixed bubble will contain infinitely many essential curves, each of which is by itself a Levy cycle.

We will describe $f$ by giving its canonical form, including the Hubbard bubble tree $B$, its preimage $f^{-1}(B)$, the exterior map $f_E$, and the interior map $f_I$. Figure~\ref{fig:InfiniteLevyCycles} shows the Hubbard bubble tree $B$ and its preimage $f^{-1}(B)$.
\begin{figure}
\centering
\raisebox{-0.47\height}{\includegraphics{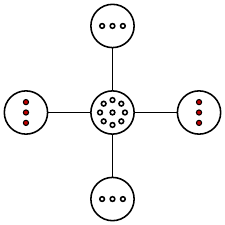}}\qquad$\overset{\textstyle f}{\longrightarrow}$\qquad
\raisebox{-0.47\height}{\includegraphics{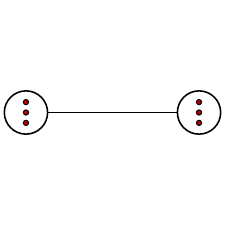}}
\caption{The Hubbard bubble tree $B$ and its preimage for a topological polynomial $f$ with infinitely many Levy cycles.  The rightmost bubble maps to itself by the identity.}
\label{fig:InfiniteLevyCycles}
\end{figure}
The exterior map $f_E$ is determined by $B$ and $f^{-1}(B)$, and the first return map for the exterior component is Thurston equivalent to the polynomial $z^4 - \sqrt[3]{2}$.  Under this Thurston equivalence, the left and right bubbles of $B$ correspond to the points $-\sqrt[3]{2}$ and $\sqrt[3]{2}$, respectively, and the top, center, and bottom bubbles of $f^{-1}(B)$ correspond to the points $i\sqrt[3]{2}$, $0$, and $-i\sqrt[3]{2}$, respectively.

We define the interior map $f_I$ to map the right bubble to itself by the identity, and map the top, left, and bottom bubbles of $f^{-1}(B)$ to the right bubble of $B$ by any homeomorphisms.  The center bubble of $f^{-1}(B)$ maps to the left bubble of $B$ by some degree four branched cover such that all three marked points of the left bubble are critical values. Figure~\ref{fig:InfiniteLevyCyclesInterior} shows an example of such a cover.  Since all three marked points of the left bubble are critical values, all three marked points of the right bubble are indeed post-critical, so the right bubble contains infinitely many different Levy cycles.
\begin{figure}
\centering
\raisebox{-0.47\height}{\includegraphics{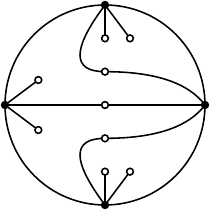}}\qquad$\overset{\textstyle f}{\longrightarrow}$\qquad
\raisebox{-0.47\height}{\includegraphics{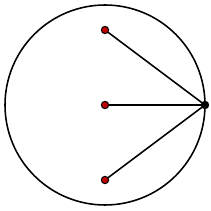}}
\caption{The portion of the interior map $f_I$ that maps the center bubble to the left bubble.}
\label{fig:InfiniteLevyCyclesInterior}
\end{figure}


\subsection{Proof of the theorem}
\label{sec:pfobs}

Like the proof of the first statement of Theorem~\ref{thm:main}, the proof of the second statement has two main steps.  In order to describe the steps, we require two definitions. In what follows, let $f$ be an obstructed post-critically finite topological polynomial.

First, we denote by $G_f$ the stabilizer of the Hubbard vertex $H_f$ in $\PMod(\R^2,P_f)$.  This is exactly the subgroup of $\PMod(\R^2,P_f)$ consisting of elements supported in the interiors of the components of the canonical obstruction~$\Gamma_f$.  
Next, we say that a vertex $T$ of $\T_n$ is invariant-modulo-$G_f$ under a map $\psi : \T_n \to \T_n$ if there is a $g \in G_f$ so that $\psi(T)=g \cdot T$.  

The two steps of the proof are:
\begin{enumerate}
    \item for each vertex $T$ of $\T_n$, some $\lambda_f^k(T)$ is invariant-modulo-$G_f$ under some power of $\lambda_f$, and \smallskip
    \item if a vertex $T$ of $\T_n$ is invariant-modulo-$G_f$ under a power of $\lambda_f$ then there is a vertex of the Levy set $L_f$ that is obtained from $T$ by a forest expansion. 
\end{enumerate}

As in Section~\ref{sec:p3}, we handle the second step in a separate proposition before proving the theorem.

\begin{proposition}
\label{prop:invtobs}
Let $f$ be an obstructed post-critically finite topological polynomial, let $n = |P_f|$, and let $T$ be a vertex of $\T_n$ that is invariant-modulo-$G_f$ under $\lambda_f$.  Then there is a vertex of the canonical Levy set $L_f \subseteq \T_n$ that is obtained from $T$ by a forest expansion. 
\end{proposition}

\begin{proof}

By hypothesis, there exists $g \in G_f$ so that $\lambda_f(T) = g \cdot T$.  It follows that $g^{-1} \cdot \lambda_f(T) = T$ and so $\lambda_{fg}(T)=T$, i.e.~$T$ is an invariant tree for the map~$fg$. Since $g$ is supported on the interiors of the curves of the canonical obstruction $\Gamma_f$, the maps $f$ and $fg$ induce the same first return map $f_E$ on the exterior complementary component of~$\Gamma_f$, so it follows from Proposition~\ref{prop:PoirierConditionsBubbleTree} that $fg$ has the same Hubbard vertex as~$f$.

By Proposition~\ref{prop:ExpandAngleCondition}, there is a forest expansion $T'$ of $T$ that is invariant and satisfies the angle condition with respect to $fg$.  Since $fg$ is obstructed, it follows from Lemma~\ref{lem:FindingHubbardVertex} that the Hubbard vertex $H_{fg}$, hence $H_f$, it obtained from $T'$ by a forest collapse.  Therefore $T'$ lies in the Levy set~$L_f$.
\end{proof}

\begin{proof}[Proof of Theorem~\ref{thm:main}(2)]

As in the statement, let $f$ be an obstructed post-critically finite topological polynomial and let $n = |P_f|$.  Let $T_0$ be a vertex of $\T_n$.

Since the lifting map $\lambda_f : \hat \T_n \dasharrow \hat \T_n$ fixes the Hubbard vertex $H_f$, it follows that $\lambda_f$ induces a simplicial map on the ``partially augmented tree complex'' $\T_n^{\bigcdot}$ which is the subcomplex of $\hat \T_n$ spanned by $\T_n \cup H_f$.

Since $G_f$ fixes $H_f$, it follows that $G_f$ acts on $\T_n^{\bigcdot}$.  The quotient $\T_n^{\bigcdot}/G_f$ is a locally finite cell complex.  Indeed, the only vertex of $\T_n^{\bigcdot}$ that is not locally finite is $H_f$, and the action of $G_f$ on the set of edges incident to $H_f$ is cofinite.  This is because the quotient of $\hat \T_n$ by $\PMod(\R^2,P_f)$ is finite, and two cells incident to $H_f$ are in the same orbit under the stabilizer $G_f$ if and only if they are in the same orbit under $\PMod(\R^2,P_f)$.  

Let $\pi : \T_n^{\bigcdot} \to \T_n^{\bigcdot}/G_f$ denote the quotient map.  We consider the sequence of vertices $T_i = \lambda_f^i(T_0)$ and the corresponding sequence $\pi(T_i)$ in $\T_n^{\bigcdot}/G_f$.  As in the proof of Theorem~\ref{thm:main}(1), we may use the local finiteness of $\T_n^{\bigcdot}/G_f$, the fact that $\T_n^{\bigcdot}$ is connected, and the fact that $\lambda_f$ is simplicial to conclude that the sequence $\pi(T_i)$ is either periodic or pre-periodic.  

Thus, there is an $m \geq 0$ and an $r > 0$ so that
\[
\pi(\lambda_f^{m}(T_0)) = \pi(\lambda_f^{m+r}(T_0)).
\]
Let $T = \lambda^{m}(T_0)$.  The previous equality can be restated as follows: there exists a $g \in G_f$ so that
\[
\lambda_{f^r}(T) = g \cdot T.
\]
We thus have $g^{-1} \cdot \lambda_{f^r}(T) = T$, which means that $\lambda_{f^r g}(T) = T$.  By Proposition~\ref{prop:invtobs}, there is a vertex of $L_{f^r g}$ that is obtained from $T$ by a forest expansion.  Since $L_f = L_{f^r}$ by Proposition~\ref{prop:InvariantIterates} and since $g$ preserves $L_f$ we have that $L_{f^r g} = L_f$.  The theorem now follows.
\end{proof}

\p{Refinements of Theorem~\ref{thm:main}} Our proof of Theorem~\ref{thm:main} gives more information than what is given by the statement.  We give here two successive refinements that are immediate from the proof.

Let $X$ be a set of vertices of $\hat \T_n$.  We define $E(X)$ to be the subset of $\T_n$ consisting of all vertices obtained from an element of $X$ by performing a (possibly trivial) forest expansion.  Similarly, we define $C(X)$ to be the subset of $\T_n$ consisting of vertices obtained by performing a (possibly trivial) forest collapse along a collapsible forest.  We write $CE(X)$ for $C(E(X))$.  

Our proof of Theorem~\ref{thm:main} shows that the nucleus for an unobstructed $f$ is contained in $CE(H_f)$ and that the nucleus for an obstructed $f$ is contained in $C(L_f)$, hence in $CE(H_f)$.  This is the first of the two refinements of Theorem~\ref{thm:main}.

Consider, for example, the rabbit, co-rabbit, and airplane polynomials.  As shown in Figure~\ref{fig:nuclei}, the minimal nuclei for the rabbit and co-rabbit polynomials are equal to the corresponding sets $CE(H_R)$ and $CE(H_C)$.  On the other hand, the minimal nucleus for the airplane polynomial is strictly smaller than $CE(H_A)$: it is equal to $\{H_A\}$.  And so for all three polynomials, the refinement gives an improvement over what is stated in Theorem~\ref{thm:main}(1), but for the airplane polynomial the first refinement does not give the minimal nucleus.

As for our obstructed example $D_a^{-1} I$, where $I(z)=z^2+i$, we have that $E(H_{D_a^{-1} I})$ is equal to the given nucleus (the Levy set), and that $CE(H_{D_a^{-1} I})$ is strictly larger.  In fact, $CE(H_{D_a^{-1} I})$ is the 1-neighborhood of the Levy set, which is the nucleus guaranteed by Theorem~\ref{thm:main}(2).  In other words, for this example, the first refinement does not improve upon Theorem~\ref{thm:main}(2).

In order to state our second refinement, let $f$ be a post-critically finite topological polynomial and let $H_f$ be the Hubbard vertex for $f$ in $\hat \T_n$.  We define $E'(H_f)$ to be the subset of $\T_n$ consisting of all vertices that are obtained from $H_f$ by performing expansions at the Julia vertices.  For each vertex of $E'(H_f)$ we may then collapse any collection of edges that correspond to non-Julia edges in $H_f$ and that form a collapsible forest; we denote by $C'E'(H_f)$ the set of all vertices obtained in this way.   For any $f$ we have $C'E'(H_f) \subseteq CE(H_f)$.  

The second refinement of Theorem~\ref{thm:main} is that $C'E'(H_f)$ is a nucleus for $f$.  For the rabbit polynomial we have that $C'E'(H_R)$ is again equal to the minimal nucleus since $C'E'(H_R) \subseteq CE(H_R)$ and since $CE(H_R)$ was already equal to the minimal nucleus (and similarly for the co-rabbit polynomial).  As for the airplane polynomial, there are no Julia vertices in $H_A$, and there are no edges that can be contracted, either, and so the minimal nucleus $\{H_A\}$ is equal to $C'E'(H_A)$.  So in all three of these cases $C'E'(H_f)$ is exactly the minimal nucleus.  There are, on the other hand, examples of unobstructed topological polynomials, such as the rabbit polynomial tuned with the basilica polynomial, where the minimal nucleus is strictly smaller than $C'E'(H_f)$.

For the obstructed example $D_a^{-1} I$, the set $C'E'(H_{D_a^{-1} I})$ is equal to the nucleus given in the introduction (the Levy set).  And so for this example, the second refinement is an improvement over what is given by Theorem~\ref{thm:main}(2).


\section{Twisted polynomial problems}
\label{sec:rabbit}

The goal of this section is to give concrete applications of our tree lifting algorithm.  A feature of our methods is that they allow us to readily discover and give unified arguments for infinitely many recognition problems, with arbitrary numbers of post-critical points.

We emphasize that the recognition problems we solve in this section are intended to be a sampling of the types of problems that can be solved with our methods, and to demonstrate how our methods can be applied without added difficulty in the presence of arbitrarily many post-critical points. There are many other generalizations and variants of the twisted rabbit problem that one can solve using our methods; for instance one can solve problems that involve twisting along different curves, or that twist other families of polynomials. We expect that recognizing structure in the solutions of many such problems will allow for the discovery of new phenomena and lead to new questions.  

We begin in Section~\ref{sec:onear} by applying our methods to solve Hubbard's original twisted rabbit problem, first solved by Bartholdi--Nekrashevych.  Then in Section~\ref{sec:manyear} we explain a generalization of this problem to a family of rabbit polynomials with arbitrarily many post-critical points, which we call the twisted many-eared rabbit problem.  

In Section~\ref{sec:z2+i} we use our methods to solve another twisted polynomial problem where the rabbit polynomial is replaced by the polynomial $z^2+i$; again this was originally solved by Bartholdi--Nekrashevych.  Finally in Section~\ref{sec:genz2+i} we give a generalization of this problem to a family of polynomials with arbitrarily many post-critical points.  As in the introduction, an important distinction between the twisted rabbit problems and the twisted $z^2+i$ problems is that in the latter case the twisted polynomials are sometimes obstructed.  

\begin{figure}
$\underset{\textstyle\text{(a)}}{\includegraphics{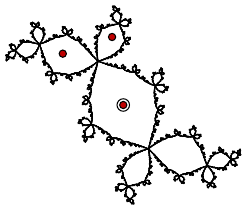}}
\qquad\qquad
\underset{\textstyle\text{(b)}}{\includegraphics{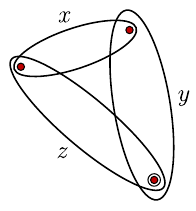}}$
\caption{(a) The Julia set for~$R$. (b) The curves $x$, $y$, and $z$ in $(\R^2,P_R)$}
\label{fig:generators}
\end{figure}

\subsection{The original twisted rabbit problem}
\label{sec:onear}

It follows from the Berstein--Levy theorem that if we post-compose the rabbit polynomial $R$ with an element of $\PMod(\R^2,P_R)$ the result is Thurston equivalent to either $R$, the co-rabbit polynomial $C$, or the airplane polynomial $A$.  Let $x$ be the curve in $(\R^2,P_R)$ shown in Figure \ref{fig:generators}, and let $D_x$ denote the Dehn twist about $x$ (as in the introduction, we take Dehn twists to be left-handed in this paper).  In this figure, and in each of the figures that follow, circled marked points are the critical points. (Bartholdi--Nekrashevych denote $D_x$ by $T$ and $D_z$ by $S$.)

The original twisted rabbit problem of Hubbard is:
\begin{quote}
    \emph{Let $m \in \Z$.  To which polynomial is $D_x^m R$ Thurston equivalent?}
\end{quote}

Bartholdi--Nekrashevych give an algorithm \cite[Theorem~4.8]{BaNe} that computes the Thurston equivalence class of $gR$ for any $g \in \PMod(\R^2,P_R)$. They also give a closed-form answer for the equivalence class of $D^m_xR$, showing that this can be read off from the 4-adic expansion of $m$ \cite[Theorem 4.7]{BaNe}.  Specifically, they show  that 
\[
D_x^m R \simeq
\begin{cases}
A&\text{ if the 4-adic expansion of }m\text{ contains a 1 or a 2}\\
R&\text{ if the 4-adic expansion of  }m\text{ contains only 0's and 3's and } m\geq 0\\
C&\text{ if the 4-adic expansion of }m\text{ contains only 0's and 3's and } m<0.
\end{cases}
\]
Here and throughout this section, we will use the symbol $\simeq$ to denote Thurston equivalence. 

\p{The basic strategy} Following Bartholdi--Nekrashevych, the basic strategy for our solution to the twisted rabbit problem has two parts.
\begin{enumerate}
    \item Give a set of reduction formulas that allow us to simplify $D_x^{m}R$ to one of the ``base cases,'' where $m \in \{-1,0,1\}$.
    \item Determine the base cases by showing that $D_x R \simeq A$ and $D^{-1}_x R \simeq C$.
\end{enumerate}
With these two steps in hand, it is straightforward to deduce the formula for $D_x^m R$ in terms of 4-adic expansions.  Before explaining the first step, we describe our main tool for producing the reduction formulas.

\begin{figure}
\centering
\raisebox{-0.47\height}{\includegraphics{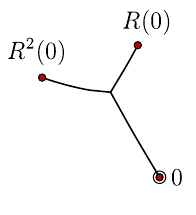}}
\qquad\quad
\raisebox{-0.47\height}{\includegraphics{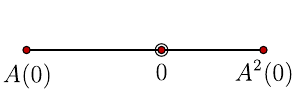}}
\quad\qquad
\raisebox{-0.47\height}{\includegraphics{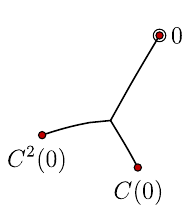}}
\caption{The Hubbard trees for \ $R$, \ $A$, \ and \ $C$}
\label{fig:HubbardTrees}
\end{figure}

\p{Lifting and lifting by borrowing} Let $f$ be a topological polynomial with post-critical set $P$.  There is an equivalence relation on $\PMod(\R^2,P)$ defined as follows: $g_1 \sim g_2$ if $g_1f \simeq g_2f$.  Our goal here is to give a procedure for replacing a given $g \in \PMod(\R^2,P)$ with another equivalent (and hopefully simpler) one.  

We require some preliminaries.  Let $h : (\R^2,P) \to (\R^2,P)$ be a homeomorphism that fixes $P$ pointwise.  We say that $h$ is \emph{liftable} through $f$ if there is a homeomorphism $\tilde h : (\R^2,P) \to (\R^2,P)$ that fixes $P$ pointwise and satisfies $f \tilde h = hf$.  The liftable mapping class group $\LMod(\R^2,P)$ is the subgroup of $\PMod(\R^2,P)$ consisting of elements with representatives that are liftable through $f$.  We emphasize here that $\LMod(\R^2,P)$ depends on $f$; it will be clear from context which topological polynomial is being used to define $\LMod(\R^2,P)$.  Also associated to $f$ there is a homomorphism 
\[
\psi : \LMod(\R^2,P) \to \PMod(\R^2,P)
\]
given by lifting through $f$.  The following lemma is due to Bartholdi--Nekrashevych.  They did not state it in exactly this form (for instance they only discuss the case of the rabbit polynomial), but the proof is the same \cite[Proposition 4.1]{BaNe}.

\begin{lemma}
\label{lem:liftbyborrow}
Let $f$ be a post-critically finite topological polynomial.  Let $g \in \PMod(\R^2,P)$ and let $h \in \PMod(\R^2,P)$ be any element with $h^{-1}g \in \LMod(\R^2,P)$.  Then
\[
g \sim \psi(h^{-1}g)h.
\]
\end{lemma}

When $g$ already lies in $\LMod(\R^2,P)$, we may take $h=\textrm{id}$ in Lemma~\ref{lem:liftbyborrow}, and we obtain the following special case: 
\[
g \sim \psi(g).
\]
In words, $g$ is equivalent to the lift of $g$.  The more general case of Lemma~\ref{lem:liftbyborrow} can be phrased as: $g$ is equivalent to the mapping class obtained by the process of ``lifting by borrowing.'' 

\begin{proof}[Proof of Lemma~\ref{lem:liftbyborrow}]

Since $h^{-1}g$ lies in $\LMod(\R^2,P)$, we may apply the map $\psi$ to obtain the lift $\psi(h^{-1}g)$ of $h^{-1}g$ under $f$.  By the definition of a lift we have $h^{-1}gf=f\psi(h^{-1}g).$  Composing both sides of this equality with $h$ we obtain
\[
f\psi(h^{-1}g)h=h^{-1}gfh.
\]
Conjugation by $\PMod(\R^2,P)$ preserves the Thurston equivalence class.  So after conjugating the left side of the above equality by $\psi(h^{-1}g)h$ and the right side by $h$ we obtain that 
\[\psi(h^{-1}g)hf\simeq gf,\]
as desired.
\end{proof}

An additional difference between our Lemma~\ref{lem:liftbyborrow} and the corresponding statement in the work of Bartholdi--Nekrashevych is that they fix once and for all a set of coset representatives for $\LMod(\R^2,P)$, and they always take $h$ to lie in this set of coset representatives.  As a result, they obtain a well-defined set map $\bar \psi : \PMod(\R^2,P) \to \PMod(\R^2,P)$.

\p{A topological description of the rabbit polynomial} We will now use the Alexander method (Proposition~\ref{prop:alex2}) to give a combinatorial description of a topological polynomial that is homotopic to the rabbit polynomial.  We will use this substitute when computing the lifts of curves, so that the lifting operation can be carried out by means of combinatorial topology, rather than through an actual analytic map.  The topological description is a composition of a topological polynomial and a homeomorphism; we think of this description as being analogous to a decomposition of a mapping class into a product of Dehn twists.  In what follows we denote by $R$ the rabbit polynomial, by $P = \{0,R(0),R^2(0)\}$ the post-critical set of $R$, and by $\Delta$ the triangle in $\R^2$ with vertex set $P$.

Let $\Sq$ be any orientation-preserving double branched cover $(\R^2,P)\to(\R^2,P)$ that is branched over 0, and that fixes $\Delta$ pointwise.  Any such map fixes the isotopy class of any tree contained in $\Delta$.  Thus, it follows from Proposition~\ref{prop:alex2} that all such double covers are homotopic relative to $P$, so there is no ambiguity.

Next, let $\Rot$ be a homeomorphism of $(\R^2,P)$ that rotates the points $P$ counterclockwise and preserves $\Delta$ as a set.  Again, any two such maps are homotopic relative to $P$ (here we can even apply the Alexander method for mapping class groups).  

We claim that the map $\Rot\Sq$ is homotopic to the rabbit polynomial relative to $P$.  This can be seen by applying Proposition~\ref{prop:alex2} to the two maps $\Rot\Sq$ and $R$.  One convenient tree to use is the tripod contained in $\Delta$.  

\p{Basic facts about mapping class groups} Below, we will frequently use without mention the following basic facts from the theory of mapping class groups.  First, we have necessary and sufficient conditions for a power of a Dehn twist to lie in the group $\LMod(\R^2,P)$ for a topological polynomial $f$ of degree 2, and we also have descriptions of the lifts:

\medskip

\begin{enumerate}[itemsep=2ex]
    \item $D_c$ lifts if and only if $f^{-1}(c)$ has two components $\tilde c_1$ and $\tilde c_2$; in this case $\psi(D_c) = D_{\tilde c_1}D_{\tilde c_2}$, and
    \item $D_c^2$ lifts if $f^{-1}(c)$ has one component $\tilde c$; in this case $\psi(D_c^2) = D_{\tilde c}$.
\end{enumerate}

\medskip

We will also use the formula
\[
hD_ch^{-1}=D_{h(c)}
\]
for any $h, D_c \in \PMod(\R^2,P)$.  Finally, we will also use several applications of the lantern relation \[
D_xD_yD_z=\textrm{id}
\]
where $x$, $y$, and $z$ are the curves in Figure~\ref{fig:generators}.  All of these facts (or versions of them) are also used by Bartholdi--Nekrashevych \cite[Section 4]{BaNe}.

\p{A triviality lemma} It will be useful in our resolution of the twisted rabbit problem and the twisted many-eared rabbit problem to have a condition under which a power of a Dehn twist lifts through a polynomial to the trivial mapping class.  We require some preliminaries.  Let $f$ be a topological polynomial of degree 2.  First, a \emph{branch cut} will mean any arc $b$ in $(\R^2,P_f)$ that connects the critical value to $\infty$.  The preimage $f^{-1}(b)$ is a pair of arcs in $(\R^2,P_f)$ that connect the critical point to $\infty$ and that intersect only at the critical point.  We say that $b$ is \emph{special} if all of the points of $P_f$ lie on one side of $f^{-1}(b)$.  Each point of $P_f$ that is not the critical value has two preimages, one in $P_f$ and one not in $P_f$.   Necessarily the marked and unmarked preimages lie on opposite sides of $f^{-1}(b)$.

Next, suppose $c$ is a curve in $(\R^2,P_f)$ that surrounds exactly two points $p_1$ and $p_2$ of $P_f$.  Then $c$ is the boundary of a neighborhood of an arc $a$ in $(\R^2,P_f)$ connecting $p_1$ to $p_2$; we refer to $a$ as a \emph{defining arc} for $c$.  The defining arc $a$ is well defined up to isotopy.  

\begin{lemma}\label{lem:triviality}
Let $f$ be a topological polynomial of degree 2 and let $b$ be a special branch cut for $f$.  Suppose $c$ is a curve in $(\R^2,P_f)$ that surrounds exactly two points of $P_f$, neither of which is the critical value, and that $a$ is a defining arc for $c$.  If $a$ is transverse to $b$ and crosses $b$ in an odd number of points, then the lift of $D_c$ is trivial.
\end{lemma}

\begin{proof}

Since $c$ does not surround the critical value of $f$, the arc $a$ does not have an endpoint at the critical value.  Therefore $f^{-1}(a)$ is a pair of arcs that are disjoint, including at their endpoints.  Since $a$ intersects $b$ in an odd number of points, the endpoints of each component of $f^{-1}(a)$ lie on opposite sides of $f^{-1}(b)$.  Since $b$ is a special branch cut, it follows that each component of $f^{-1}(a)$ connects a point of $P_f$ to a point of $f^{-1}(P_f) \setminus P_f$.  The lift $\psi(D_c)$ is equal to the product of the Dehn twists about the curves of the boundary of a neighborhood of $f^{-1}(a)$.  Since each such curve surrounds only one point of $P_f$, this product is trivial.
\end{proof}

\p{A convenient notation}  Let $f$ be a topological polynomial and let $\LMod(\R^2,P)$ and $\psi$ be the associated liftable mapping class group and homomorphism.  We will use the notation
\[
g_1 \stackrel{h}{\leadsto} g_2 
\]
to mean that $g_2 = \psi(h^{-1}g_1)h$ (the element $h$ above the arrow lies in the coset $g_1\LMod(\R^2,P)$ as in Lemma~\ref{lem:liftbyborrow}).  We emphasize that, by Lemma~\ref{lem:liftbyborrow}, we have that $g_1 \stackrel{h}{\leadsto} g_2$ implies $g_1 \sim g_2$.  When $h=\textrm{id}$ we write $g_1 \leadsto g_2$.  In such a case, since $\psi$ is a homomorphism on $\LMod(\R^2,P)$ we have $g_1^k \leadsto g_2^k$ for any $k$.  While the arrow notation $g_1 \leadsto g_2$ belies the fact that $g_1$ and $g_2$ are equivalent, it is on the other hand meant to suggest a simplification process.

\p{Reduction formulas} We are now ready to explain the first of the two steps in our solution to the twisted rabbit problem.  The reduction formulas of Bartholdi--Nekrashevych are:
\[
D_x^{m} R\simeq 
\begin{cases} D_x^{k}R&m=4k\\
D_x R&m=4k+1\\
D_xR&m=4k+2\\
D_x^k R&m=4k+3.
\end{cases}
\]
We give here our version of the Bartholdi--Nekrashevych calculations that justify the reduction formulas \cite[Section 4]{BaNe}.  Our calculations in Section~\ref{sec:manyear} will be modeled on these.  Our arguments are much shorter than the corresponding ones by Bartholdi--Nekrashevych; where they use iterated monodromy groups to compute the lifts of Dehn twists, we simply lift the corresponding curves.  

\begin{figure}
\centering
\raisebox{-0.47\height}{\includegraphics{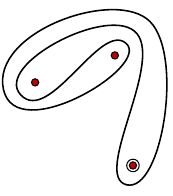}}
$\qquad\quad\longrightarrow\quad\qquad$
\raisebox{-0.47\height}{\includegraphics{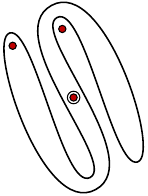}}
\caption{The curve $D_x^{-1}(y)$ and its lift $\lambda_{\Sq}\Rot^{-1}(D_x^{-1}(y))$}
\label{fig:rabbit_lift}
\end{figure}

To begin, we observe the following facts.  In what follows, let $y$ and $z$ be the curves in $(\R^2,P)$ shown in Figure~\ref{fig:generators}.  The curve $z$ has two curves in its preimage under $R$, and the only essential one is homotopic to $x$.   Therefore, as above, $D_z$ lies in $\LMod(\R^2,P)$ and $\psi(D_z)=D_x$; we may write this as $D_z \leadsto D_x$.  The preimage of $x$ has a single component, isotopic to $y$.  Therefore $D_x$ does not lie in $\LMod(\R^2,P)$, but $D_x^2$ does and $\psi(D_x^2) = D_y$, or $D_x^2 \leadsto D_y$.  Similarly $D_y^2 \leadsto D_z$.  Finally, it follows from Lemma~\ref{lem:triviality} that $D_{D_x^{-1}(z)} \leadsto \textrm{id}$; the special branch cut here is the straight ray from the critical value to $\infty$ that avoids the interior of the triangle determined by $P$.  

\bigskip

\noindent \emph{Case 1: $m=4k$.} In this case we have
\[
D_x^{4k}=(D_x^2)^{2k}\leadsto D_y^{2k} = (D_y^{2})^k\leadsto D_z^k \leadsto D_x^k.
\]
Thus $D_x^{4k} \sim D_x^k$, as desired.

\bigskip

 \noindent\emph{Case 2: $m=4k+1$.}  
 In this case we require one additional fact, namely that  $D^2_{D_x^{-1}(y)} \leadsto D_z$.  This follows from the fact that  $R^{-1}(D_x^{-1}(y))$ has one component, namely $z$; see Figure \ref{fig:rabbit_lift}.  We have
\[
D_x^{4k+1}\stackrel{D_x}{\leadsto}D_y^{2k}D_x\stackrel{D_x}{\leadsto}D_z^k D_x
\stackrel{D_x}{\leadsto}
D_x.
\]
Thus, $D_x^{4k+1} \sim D_x$, as desired.

\bigskip

\noindent\emph{Case 3: $m=4k+2$.} In this case we have
\[ D_x^{4k+2}
\leadsto D_y^{2k+1}
\stackrel{D_y^{-1}}{\leadsto}
D_z^{k+1}D_y^{-1}=D_z^{k+2}D_x
\stackrel{D_x}{\leadsto}
D_x,\]
where the equality uses the lantern relation.  Thus $D_x^{4k+2}\sim D_x$ as desired.
\bigskip

\noindent\emph{Case 4: $m=4k+3$.}  In this case we have 
\begin{align*}
D_x^{4k+3}
&\stackrel{D_x}{\leadsto}
D_y^{2k+1}D_x=D_y^{2k}D_x^{-1}D_z^{-1}D_x = D_y^{2k}D_{D_x^{-1}(z)}^{-1}
\leadsto
D_z^k
\leadsto
D_x^k,
\end{align*}
where in the first equality we used the lantern relation.  Thus, $D_x^{4k+3} \sim D_x^k$, as desired.

\begin{figure}
\centering
$\begin{array}{c}
\raisebox{-0.47\height}{\includegraphics{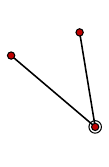}}
\xrightarrow{\textstyle D_x^{-1}}\;
\raisebox{-0.47\height}{\includegraphics{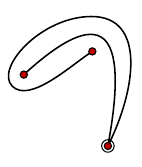}}
\;\xrightarrow{\textstyle \mathrm{Rot}^{-1}}\;
\raisebox{-0.47\height}{\includegraphics{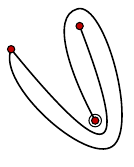}}
\xrightarrow{\textstyle \lambda_{\mathrm{Sq}}}
\raisebox{-0.47\height}{\includegraphics{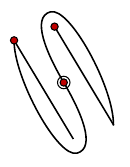}}
=
\raisebox{-0.47\height}{\includegraphics{RabbitBigLift5}}
\end{array}$
\caption{An invariant tree for $D_xR$}
\label{fig:invariant_under_DxR}
\end{figure}

\begin{figure}
\centering
$\begin{array}{c}
\raisebox{-0.47\height}{\includegraphics{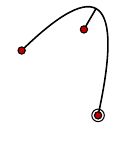}}
\xrightarrow{\textstyle D_x}\;
\raisebox{-0.47\height}{\includegraphics{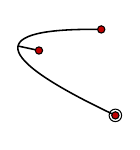}}
\;\xrightarrow{\textstyle \mathrm{Rot}^{-1}}\;
\raisebox{-0.47\height}{\includegraphics{Rabbit2ndBigLift1}}
\xrightarrow{\textstyle \lambda_{\mathrm{Sq}}}
\raisebox{-0.47\height}{\includegraphics{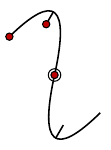}}
=
\raisebox{-0.47\height}{\includegraphics{Rabbit2ndBigLift1}}
\end{array}$
\caption{An invariant tree for $D_x^{-1}R$}
\label{fig:invariant_under_Dx-1R}
\end{figure}

\p{The base cases} We now execute the second step of the solution to the twisted rabbit problem, which is to show that $D_xR \simeq A$ and that $D_x^{-1}R \simeq C$.  Here we use the Alexander method in place of the theory of iterated monodromy groups.  

In the two base cases we replace $D_x^{m} R$ with $D_x^{m}\Rot\Sq$.  We further observe that 
\[
\lambda_{D_x^{m}\Rot\Sq}=\lambda_{\Sq}\Rot^{-1} D_x^{-m}.
\]
We begin with the case $m=1$.  The leftmost tree in Figure~\ref{fig:invariant_under_DxR} is invariant under $\lambda_{\Sq}\Rot^{-1} D_x^{-1}$.  Since there is no tree with 2 edges invariant under $\lambda_R$ or $\lambda_C$ (cf. Figure~\ref{fig:nuclei}), we conclude that $D_x\Rot\Sq$, hence $D_xR$, is Thurston equivalent to $A$.

We now treat the case $m=-1$.  The leftmost tree in Figure~\ref{fig:invariant_under_Dx-1R} is invariant under $\lambda_{\Sq}\Rot^{-1} D_x$, and the action of $D_x^{-1}\Rot\Sq$ on the tree is a clockwise rotation (meaning that the half-edges emanating from the vertex of degree 3 are permuted in the clockwise direction).  Since the nucleus for $A$ contains no tripods, and since the invariant tree for $R$ is rotated in the counterclockwise direction by $R$, it follows that $D_x^{-1}R$ is Thurston equivalent to $C$.

\begin{figure}
\centering
$\begin{array}{c}
\raisebox{-0.47\height}{\includegraphics{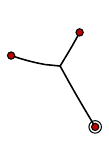}}
\xrightarrow{\textstyle D_x^{-1}}
\raisebox{-0.47\height}{\includegraphics{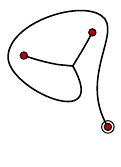}}
\xrightarrow{\textstyle \mathrm{Rot}^{-1}}
\raisebox{-0.47\height}{\includegraphics{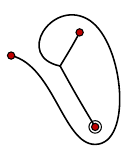}}
\xrightarrow{\textstyle\lambda_{\mathrm{Sq}}} \raisebox{-0.47\height}{\includegraphics{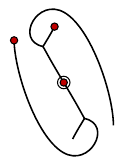}}
=
\raisebox{-0.47\height}{\includegraphics{RabbitBigLift5}}
\end{array}$
\caption{Applying the tree lifting algorithm to find the Hubbard tree for $D_xR$}
\label{fig:tree_lift}
\end{figure}

\p{The hidden role of the tree lifting algorithm} Our approach to the base cases of the twisted rabbit problem is very direct: to show the given Thurston equivalences, we simply produce the correct topological Hubbard trees.  But how would one guess the topological Hubbard trees?  The answer is not to guess, but to apply our tree lifting algorithm.  For example, Figure~\ref{fig:tree_lift} shows the image of $H_R$ under $\lambda_{D_xR}$.  So after a single iterate of the tree lifting algorithm, we arrive at the Hubbard tree for $D_xR$, and hence recognize that $D_xR$ is equivalent to $A$.  Using the tree lifting algorithm is in fact how we determined that these maps are equivalent in the first place.  After the fact, however, this step is not required for the proof.

\bigskip\noindent\emph{Why 4-adic?} We can see from our solution to the twisted rabbit problem why the 4-adic expansion of the power $m$ appears in the answer.  In the derivation of the reduction formulas, the lifting operation permutes the curves $x$, $y$, and $z$ cyclically.  Each time we lift an even power of $D_x$ or $D_y$, the power divides by 2, and when we lift a power of $D_z$, the power stays the same.  The reason is because $x$ and $y$ surround the critical value, while $z$ does not.  So when repeatedly lifting a power of $D_x$ that is divisible by 4, we ``lose'' a factor of 4 in the power for every three iterations of the lifting map (as in Case 1 of the reduction formulas).  We will see the same phenomenon in our solution to the twisted many-eared rabbit problem below.

\subsection{The twisted many-eared rabbit problem}
\label{sec:manyear}

In this section we pose and give a closed-form answer to a generalization of the Hubbard's twisted rabbit problem. To pose the problem, we turn to the setting of post-critically finite topological polynomials $f$ with $|P_f| > 3$. 

\p{Quadratic polynomials with periodic critical point} Consider the quadratic polynomials of the form $z^2+c$ where the unique critical point (namely, 0) is $n$-periodic. Denote this set $PC_{n}$ (the PC stands for ``periodic critical point'').  For $n=3$ we have $PC_3 = \{R,C,A\}$.  The cardinality of $PC_n$ grows exponentially with $n$ and its elements correspond to the solutions of $0=f^{n}(0)$ that do not satisfy the same equation for any smaller value of $n$. 

As in the case $n=3$, it follows from the Berstein--Levy theorem that if we compose any element $f$ of $PC_n$ with an element of $\PMod(\R^2,P_f)$ then the result is Thurston equivalent to a polynomial and is in particular Thurston equivalent to an element of $PC_n$. 

\p{The set of $1/n$-rabbit polynomials} For each $n \geq 3$ and $1\leq q < n$ relatively prime to $n$, there is a polynomial in $PC_n$ called the $q/n$-rabbit polynomial.  In this paper we will focus on the $1/n$-rabbit polynomials, which we denote $R_n$.  One description of $R_n$ is that its address in the Mandelbrot set is $1/n$.  Another description of $R_n$ is that it is a quadratic polynomial $R_n(z) = z^2+c$ with the following properties: (1) the critical point 0 is $n$-periodic, (2) the post-critical set lies on the boundary of a convex polygon $\Delta$, (3) the Hubbard tree $H_{R_n}$ is an $n$-pod contained in $\Delta$  (here, an $n$-pod is a tree with leaves at all $n$ post-critical points and with one unmarked vertex of degree $n$), and (4) the action of $R_n$ on $H_{R_n}$ is counterclockwise rotation by $1/n$; cf. Figure~\ref{fig:ngen}.  We will give a combinatorial description of $R_n$ below.  

The $1/n$-rabbit polynomial is sometimes called the $(n-1)$-eared rabbit polynomial, so that the $1/3$-rabbit polynomial $R_3$ is the usual rabbit polynomial and the $1/4$-rabbit polynomial $R_4$ is the 3-eared rabbit polynomial. (The so-called basilica polynomial $z^2-1$ can be thought of as the $1/2$-rabbit polynomial $R_2$, and the map $z^2$ can be thought of as the $1/1$-rabbit polynomial.)    The $q/n$-rabbit polynomial is described in the same way as the $1/n$-rabbit polynomial, with the $1/n$-rotation replaced by a $q/n$-rotation.  For example, the $2/3$-rabbit polynomial is the co-rabbit polynomial.

\p{Statement of the problem} For each $1/n$-rabbit polynomial $R_n$, we define $x_n$ to be the curve in $\R^2$ obtained as the boundary of a regular neighborhood of the straight line segment between $R_n(0)$ and $R^{2}_n(0)$.  For instance $x_6$ is the curve shown in Figure~\ref{fig:ngen}.

\begin{figure}
\centering
$\underset{\textstyle\text{(a)}}{\includegraphics{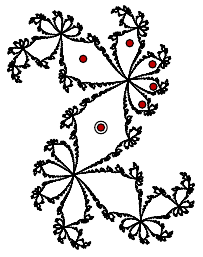}}$
\qquad\qquad\quad
$\underset{\textstyle\text{(b)}}{\includegraphics{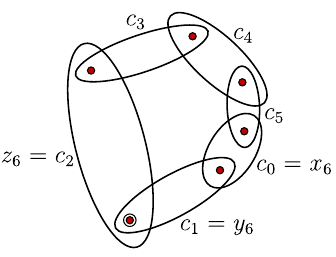}}$
\caption{(a) The Julia set and post-critical set of $R_6$.  (b) The curves $c_i$ for $R_6$}
\label{fig:ngen}
\end{figure}

We are now ready to state our twisted many-eared rabbit problem: 
\begin{quote}
\emph{Let $m \in \Z$.  To which polynomial is $D_{x_n}^m R_{n}$ Thurston equivalent?} 
\end{quote}
The cases $n=1$ and $n=2$ are trivial, and the case $n=3$ is Hubbard's original problem. 

In what follows we give a closed-form answer to the twisted many-eared rabbit problem for $n \geq 4$.  Our analysis treats all $n \geq 4$ with a single argument. As in Section~\ref{sec:onear}, we proceed in two steps, first producing reduction formulas and then computing base cases.

\p{The answer} Before stating the answer to our twisted many-eared rabbit problem, Theorem~\ref{thm:many} below, we need to describe the polynomials that appear.  There are three sequences of polynomials, $A_n$, $B_n$, and $K_n$, all defined for $n \geq 4$ and lying in $PC_n$.  The polynomial $A_n$ is the real polynomial of period $n$ that is furthest from the main cardioid of the Mandelbrot set.  The polynomial $A_4$ is sometimes called the airbus polynomial, as it is a period 4 version of the airplane polynomial $A \in PC_3$.  The polynomial $B_n$ is the second furthest real polynomial of period $n$ from the main cardioid of the Mandelbrot set; the polynomial $B_4$ is the tuning of the basilica polynomial with the basilica polynomial (the other polynomials $B_n$ are not tunings).  The polynomial $K_4$ is sometimes called the Kokopelli polynomial.  To each $p/n$-rabbit polynomial, there is an associated Kokopelli polynomial that lies nearby in the Mandelbrot set; the polynomial $K_n$ is the Kokopelli polynomial associated to $R_n$.  The polynomial $K_n$ is given by the complex kneading sequence $(1\,\; 1 \;\cdots\; 1\; {-1}\; \ast)$,  where the symbol 1 appears $n-2$ times (see Schleicher \cite{schleicher} for the definition of kneading sequence). 

Because of the Alexander method for quadratic polynomials (Proposition~\ref{prop:alex2}), the only things we need to know about $A_n$, $B_n$, and $K_n$ are the dynamical maps on their Hubbard trees.  The Hubbard trees are indicated in Figure~\ref{fig:fig}.  The dynamical maps on the Hubbard trees are described below.  

\[
\begin{array}{r@{\;=\;}l}
(A_n)_*(e_i) &
\begin{cases}
e_1e_2\cdots e_{n-1} & i=1 \\
e_{i-1}& 2\leq i\leq n-1
\end{cases} \\[20pt]
(K_n)_*(e_i) &
\begin{cases}
e_2e_3\hspace*{7.9ex}  & i=1 \\
e_{i+1}& 2\leq i\leq n-1\\
e_1e_2 & i=n
\end{cases}
\end{array}
\qquad
(B_n)_*(e_i) =
\begin{cases}
e_3\cdots e_{n-1}\hspace*{2ex} & i=1 \\
e_1e_2& i=2\\
e_1 & i=3 \\
e_2e_3 & i=4 \\
e_{i-1} & 5\leq i\leq n-1
\end{cases}
\]

\begin{figure}
\centering
\includegraphics{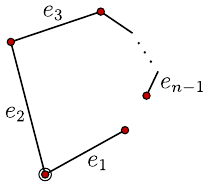}
\quad\qquad
\includegraphics{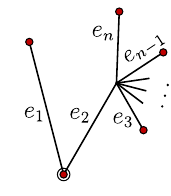}
\quad\qquad\qquad
\includegraphics{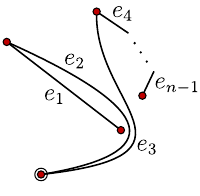}
\caption{The topological Hubbard trees for $D_x  R_n\simeq A_n$, $ D_y R_n\simeq K_n$, and $D_x^{-1} R_n\simeq B_n$}
\label{fig:fig}
\end{figure}

Finally, in the statement of the theorem, the \emph{4-free part} of a nonzero integer $m$ is $m/4^\ell$, where $4^\ell$ is the largest power of 4 that divides $m$.  We also define the 4-free part of 0 to be 0.  

\begin{theorem}
\label{thm:many}
Let $n \geq 4$, let $x=x_n$, and let $m \in \Z$.  Let $m'$ be the 4-free part of $m$.  Then  
\[D_x^m R_n\simeq \begin{cases}
R_n& m'=m=0\\
A_n& m' \equiv 1 \mod 4\\
K_n& m' \equiv 2 \mod 4\\
B_n& m' \equiv 3 \mod 4.
\end{cases}\]
\end{theorem}

Our answer to the twisted many-eared rabbit problem can be put into a similar form to the way Bartholdi--Nekrashevych phrased the answer to the original twisted rabbit problem:  if the first non-zero digit in the 4-adic expansion of $m$ is 1, 2, or 3, then $D_x^mR_n$ is Thurston equivalent to $A_n$, $K_n$, or $B_n$, respectively; otherwise, if $m=0$ then $D_x^mR_n$ is Thurston equivalent to $R_n$.

\p{Combinatorial description of the $1/n$-rabbit polynomial}  We will now give a combinatorial topological description of a topological polynomial that is Thurston equivalent to $R_n$.   In what follows we denote by $P_n$ the post-critical set of $R_n$.  

Let $\Sq_n$ be any double branched cover $(\R^2,P_n)\to(\R^2,P_n)$ that is branched over 0 and fixes pointwise the convex polygon $\Delta$ determined by $P_n$.  Also, let $\Rot_n$ be a homeomorphism of $(\R^2,P_n)$ that rotates the points $P_n$ counterclockwise and preserves $\Delta$.  As in Section~\ref{sec:onear}, it follows from the Alexander method (Proposition~\ref{prop:alex2}) that $\Rot_n \Sq_n$ is homotopic to $R_n$ relative to $P_n$, hence is Thurston equivalent to $R_n$.  (Similarly, the $q/n$-rabbit polynomial is homotopic to $\Rot_n^q \Sq_n$.)

\p{Reduction formulas} The reduction formulas for the twisted many-eared rabbit problem are similar to the reduction formulas for the original twisted rabbit problem. The main difference is in the $4k+3$ case: in the original twisted rabbit problem, the power of $D_x$ subtracts 3 and divides by 4, while here it immediately drops to $-1$. The underlying reason for the difference in this case is that for $n>3$ there are certain Dehn twists that arise in the calculation that commute, while for $n=3$ no two distinct Dehn twists commute.

Another difference to highlight is that the $4k+1$ and $4k+2$ cases reduce to different base cases in the twisted many-eared rabbit problem, whereas for the original twisted rabbit problem, they reduce to the same base case. 

Throughout this section, let $y=y_n$ and $z=z_n$ be the curves shown in Figure~\ref{fig:ngen}.

\begin{lemma}
\label{n_reduction}
Let $n \geq 4$, let $x=x_n$, and let $m \in \Z$.  Then 
\[D_x^m R_n\simeq \begin{cases}
D_x^k R_n&m=4k\\
D_x  R_n&m=4k+1\\
D_y R_n&m=4k+2\\
D_x^{-1} R_n&m=4k+3.
\end{cases}
\]
\end{lemma}

\begin{proof}

As in Section~\ref{sec:onear} we write $g \sim h$ if $gR_n \simeq hR_n$.  As in the proof of the reduction formulas for the original twisted rabbit problem, we begin by listing some basic formulas that will be used in the four cases.  Let $P_n = P_{R_n}$ and let $\{c_0,\dots,c_{n-1}\}$ be the cyclically ordered curves in $(\R^2,P_n)$ obtained by taking the boundary of a regular neighborhood of the straight line segments between pairs of consecutive post-critical points (for $R_6$, the curves $c_0,c_1,c_2,c_3,c_4,c_5$ are shown in Figure~\ref{fig:ngen}).  The curves $c_0$, $c_1$, and $c_2$ are $x_n$, $y_n$, and $z_n$, respectively; in what follows we refer to these curves as $x$, $y$, and $z$.

The basic formulas we will use are 
\[
D_x^2 \leadsto D_y
\]
and
\[D_y^2 \leadsto D_z \leadsto D_{c_3} \leadsto \cdots \leadsto D_{c_{n-1}} \leadsto D_x.
\]

We will treat the cases where $m$ is equal to $4k$, $4k+1$, $4k+2$, and $4k+3$ in turn.  In the second, third, and fourth cases, we will apply Lemma~\ref{lem:triviality}.  The special branch cut is again the straight ray from the critical value to $\infty$ that avoids the convex hull of $P_n$.

\bigskip

\noindent \emph{Case 1: $m=4k$.} In this case we have
\[D_x^{4k}\leadsto D_y^{2k} 
\leadsto D_z^k
\leadsto D_{c_3}^k \leadsto \cdots \leadsto D_{c_{n-1}}^k  \leadsto
D_x^k.\]
We conclude that $D_x^{4k}  \sim D_x^k$, as desired.\\

\begin{figure}
\centering
\raisebox{-0.47\height}{\includegraphics{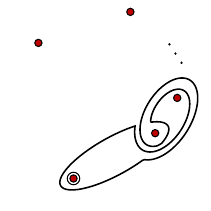}}
$\quad\qquad\longrightarrow\quad\qquad$
\raisebox{-0.47\height}{\includegraphics{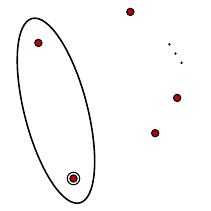}}
\caption{\emph{Left:} the curve $D_x^{-1}(y)$; \emph{Right:} its lift $\lambda_{R_n}(D_x^{-1}(y))=z$}
\label{fig:twistz}
\end{figure}

\bigskip

\noindent \emph{Case 2: $m=4k+1$.} In this case we require three new facts.  The first fact is that
\[
D^2_{D_x^{-1}(y)} \leadsto D_z.
\]
Indeed, the curve $D_x^{-1}(y)$ and its preimage under $R_n$ are shown in Figure \ref{fig:twistz}.  The second fact is that $D_x^{-1}(c_{n-1}) \leadsto \textrm{id}$; this follows from Lemma~\ref{lem:triviality}.  The third is that the Dehn twists $D_z, D_{c_3}, \cdots, D_{c_{n-2}}$ all commute with $D_x$.  With these facts in hand we have
\[
D_x^{4k+1}\stackrel{D_x}{\leadsto}D_y^{2k}D_x\stackrel{D_x}{\leadsto}D_z^k D_x
\stackrel{D_x}{\leadsto}
D_{c_3}^k D_x \stackrel{D_x}{\leadsto}
\cdots
\stackrel{D_x}{\leadsto}D_{c_{n-1}}^kD_x\stackrel{D_x}{\leadsto} D_x.
\]
Therefore we conclude that $D_x^{4k+1} \sim D_x$, as desired.\\

\bigskip

\noindent \emph{Case 3: $m=4k+2$.} In this case we use the fact that $D_{D_y^{-1}(z)} \leadsto \textrm{id}$; again this follows from Lemma~\ref{lem:triviality}.  We thus have
\[D_x^{4k+2}\leadsto D_y^{2k+1}\stackrel{D_y}{\leadsto}D_z^kD_y\stackrel{D_y}{\leadsto}D_y.\] 
So $D_x^{4k+2} \sim D_y$, as desired.

\bigskip

\noindent \emph{Case 4: $m=4k+3$.} In this case we use two additional facts. The first is that $D_{D_x(y)} \leadsto D_{D_y(z)}$ (similar to Figure~\ref{fig:twistz}). The second is that $D_xD_yD_zD_y^{-1}D_x^{-1}=D_{D_xD_y(z)} \leadsto \textrm{id}$; again this follows from Lemma~\ref{lem:triviality}.  We have
\begin{align*}
D_x^{4k+3}&\stackrel{D_x^{-1}}{\leadsto}D_y^{2k+2}D_x^{-1}\stackrel{D_x^{-1}}{\leadsto} D_yD_z^{2k+2}D_y^{-1}D_x^{-1}\stackrel{D_x^{-1}}{\leadsto}D_x^{-1}.
\end{align*}
We conclude that $D_x^{4k+3}  \sim D_x^{-1}$, as desired.
\end{proof}

\p{Base cases} To complete our solution to the twisted many-eared rabbit problem, it remains to determine the base cases.

\begin{proof}[Proof of Theorem~\ref{thm:many}]

It follows from Lemma~\ref{n_reduction} that when $m \neq 0$ the map $D_x^m R_n$ is Thurston equivalent to either $D_xR_n$, $D_yR_n$, or $D_x^{-1} R_n$. 
It remains to check that $D_xR_n$, $D_yR_n$, and $D_x^{-1} R_n$ are Thurston equivalent to $A_n$, $K_n$, and $B_n$, respectively.  We may do this by applying the Alexander method (Proposition \ref{prop:alex2}) using the Hubbard trees and dynamical maps for $A_n$, $K_n$, and $B_n$ given above.
\end{proof}

\p{More generalizations} Notice that because each of $c_2,c_3,\dots,c_{n-1}$ lifts to $x$ under iteration of the lifting map for $R_n$, we have that $D_{c_i}^m \simeq D_x^m$ for all $i\neq 1$ and $m \in \Z$.  This means that our solutions to the twisted rabbit problems of the form $D_x^mR_n$ also immediately give solutions to all twisted rabbit problems of the form $D_{c_i}^mR_n$ with $i \in \{2,\dots,n-1\}$.


\subsection{Twisting \boldmath$z^2+i$}
\label{sec:z2+i}

In this section, we use our methods to recover results of Bartholdi--Nekrashevych on twisting the polynomial $I(z)=z^2+i$ by the elements of a particular cyclic subgroup of $\PMod(\R^2,P_I)$.  In the case that the twisted map is obstructed, we give a complete topological description, using the canonical form for topological polynomials from Section~\ref{sec:canonical form}.  At the end of the section, we explain how to classify twistings by other elements of $\PMod(\R^2,P_I)$.

\begin{figure}
\centering
\raisebox{-0.47\height}{\includegraphics{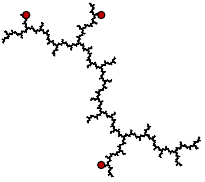}}
\qquad\qquad
\raisebox{-0.47\height}{\includegraphics{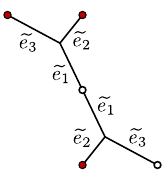}}
$\;\;\rightarrow\;\;$
\raisebox{-0.47\height}{\includegraphics{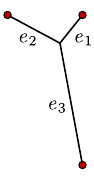}}
\caption{The Julia set and post-critical set for $I=I_3$, as well as the topological Hubbard tree for~$I$ and its preimage.  Preimages of marked points are shown in white, and each edge labeled $\widetilde{e_i}$ maps isomorphically to~$e_i$}
\label{fig:I3}
\end{figure}

Let $a$, $b$, and $c$ be the curves in $(\R^2,P_{I})$ shown in Figure~\ref{fig:I3curves}(a); these curves are situated similarly as in the twisted rabbit problems, but we match notation for $a$ and $b$ with that of Bartholdi--Nekrashevych (see \cite[Section~6.3]{BaNe}).   

The twisted $z^2+i$ problem that we solve here is:
\begin{quote}
    \emph{Let $m \in \Z$.  To which topological polynomial is $D_b^mI$ Thurston equivalent?}
\end{quote}
We pose the problem using the twist about $b$ rather than $a$ or $c$, as this choice yields the most interesting phenomena.

The Hubbard tree for $I$ as well as its full preimage under $I$ are depicted in Figure~\ref{fig:I3}.  This pair of trees will play an analogous role in our calculations that the $\Rot \Sq$ map played in our calculations for the rabbit polynomial.  Indeed, the data of the pair of trees suffices to determine the lifts of trees and curves.

Recall from the introduction that $\bar I(z) = z^2-i$.  It follows from the work of Bartholdi--Nekrashevych that the answer to the above twisted $z^2+i$ problem is
\[
D_b^m I \simeq
\begin{cases}
I &m \equiv 0 \mod 4\\
D_b^{-1} D_a^{-1} I&m \equiv 1 \mod 4\\
D_a^{-1} I &m \equiv 2 \mod 4\\
\bar{I}&m \equiv 3 \mod 4.\\
\end{cases}
\]
The maps $D_b^{-1} D_a^{-1} I$ and $D_a^{-1} I$ are obstructed; we describe them below.

Note that in their calculations, Bartholdi--Nekrashevych use right-handed Dehn twists for twisting $z^2+i$, while they use left-handed Dehn twists for twisting the rabbit polynomial. We will stay with our convention of using left-handed Dehn twists for twisting $z^2+i$; the reader should be cognizant of this difference when comparing our answers and calculations to those of Bartholdi--Nekrashevych.

\begin{figure}
\centering
$\underset{\textstyle\text{(a)}}{\includegraphics{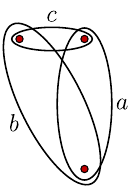}}$
\qquad\qquad
$\underset{\textstyle\text{(b)}}{\includegraphics{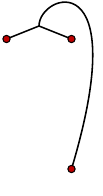}}$
\qquad\qquad
$\underset{\textstyle\text{(c)}}{\includegraphics{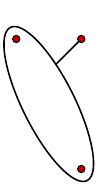}}$
\caption{(a)~The curves $a$, $b$, and $c$ in $(\R^2,P_I)$. (b)~The topological Hubbard tree for $D_c^{-1}I$. (c)~The Hubbard bubble tree for $D_a^{-1}I$ and $D_b^{-1}D_a^{-1}I$}
\label{fig:I3curves}
\end{figure}

\p{Reduction formulas} The reduction formulas for the twisted $z^2+i$ problem are:
\[
D_b^{m} I\simeq 
\begin{cases}
I &m=4k\\
D_b^{-1} D_a^{-1} I &m=4k+1\\
D_a^{-1} I&m=4k+2\\
D_c^{-1} I&m=4k+3\\
\end{cases}
\]
We now justify the reduction formulas. We will use the following facts: $D_a^2 \leadsto \textrm{id}$, $D_b \leadsto D_c$, and $D_c^2 \leadsto D_a$.  We will also use the fact that the preimage of $D_c(a)$ under $I$ is a trivial curve.  We treat the four cases in turn.

\bigskip

\noindent \emph{Case 1: $m=4k$.}  $D_b^{4k} \leadsto D_c^{4k} \leadsto D_a^{2k} \leadsto \textrm{id}$.

\medskip

\noindent\emph{Case 2: $m=4k+1$.}\,$D_b^{4k+1}\/{\leadsto} D_c^{4k+1} \stackrel{D_c^{-1}}\leadsto D_a^{2k+1}D_c^{-1} \stackrel{D_a^{-1}D_c^{-1}}\leadsto D_a^{-1}D_c^{-1} = D_b \leadsto D_c=D_b^{-1}D_a^{-1}$.

\medskip

\noindent\emph{Case 3: $m=4k+2$.}  $D_b^{4k+2} \leadsto D_c^{4k+2} \leadsto D_a^{2k+1} \stackrel{D_a^{-1}}\leadsto D_a^{-1}$.

\medskip

\noindent\emph{Case 4: $m=4k+3$.} $D_b^{4k+3} \leadsto D_c^{4k+3} \stackrel{D_c^{-1}} \leadsto D_a^{2k+2}D_c^{-1} \stackrel{D_c^{-1}}\leadsto D_c^{-1}$.

\p{Base cases} From the above topological description of $I$, we see that the dynamical map $I_*$ on the Hubbard tree $H_I$ is 
\begin{align*}
I_*(e_i) &= 
\begin{cases}
e_2 e_3 e_3 & i=1 \\
e_{i-1}& 2\leq i\leq 3.
\end{cases}
\end{align*}
The Hubbard tree for $\bar I$ is the reflection of $H_I$ across the $x$-axis, and the dynamical map $\bar I_*$ is the conjugate of $I_*$ by this reflection.  

The Hubbard trees for $I$ and $\bar I$ are both tripods.  We can distinguish their Thurston equivalence classes as follows: the polynomial $I$ rotates two edges of its Hubbard tree in the counterclockwise direction and $\bar I$ rotates two of the edges of its Hubbard tree in the clockwise direction (in both cases there is a third edge that gets stretched over three edges).  

As per the reduction formulas, there are three base cases to consider: $D_c^{-1} I$, $D_a^{-1} I$, and $D_b^{-1} D_a^{-1} I$. Figure~\ref{fig:I3curves}(b) shows the topological Hubbard tree for $D_c^{-1} I$; this tree lies in the same $\Mod(\R^2,P_I)$-orbit as the trees for $I$ and $\bar I$. The map $D_c^{-1} I$ is Thurston equivalent to $\bar{I}$ (and not $I$), since $(D_c^{-1} I)_*$ rotates two of the edges in the clockwise direction.  

The maps $D_a^{-1}I$ and $D_b^{-1} D_a^{-1} I$ are both obstructed; we will describe them in terms of the canonical form from Section~\ref{sec:canonical form}. They each have the Hubbard bubble tree shown in Figure~\ref{fig:I3curves}(c). For each, the single curve $b$ is the canonical Levy cycle; there is a single marked point  on the boundary of the closed Levy disk bounded by $b$ where the exterior forest meets~$b$. And for each, the exterior map is Thurston equivalent to the polynomial $z^2-2$, the unique polynomial with the required portrait (up to affine equivalence). The Hubbard tree for $z^2-2$ is an edge between two vertices.

For $D_a^{-1} I$,\label{dai} the interior map on the closed Levy disk bounded by $b$ is a left-handed half-twist.  As above, this can be seen from the computation of the lift under $D_a^{-1} I$ of a single tree.  The required calculation is shown in Figure~\ref{fig:I3calculation}.

For $D_b^{-1} D_a^{-1} I$, the interior map on the closed Levy disk bounded by $b$ is a right-handed half-twist.  This is simply the induced interior map of $D_a^{-1} I$ post-composed by the mapping class $D_b^{-1}$.

\begin{figure}
\centering
$\raisebox{-0.47\height}{\includegraphics{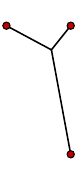}}
\quad\xrightarrow{\textstyle D_a}\quad
\raisebox{-0.47\height}{\includegraphics{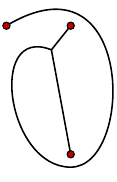}}
\quad\xrightarrow{\textstyle I^{-1}}\quad
\raisebox{-0.47\height}{\includegraphics{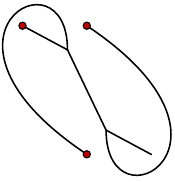}}
\quad=\quad
\raisebox{-0.47\height}{\includegraphics{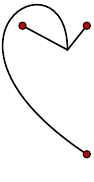}}$
\caption{A tree and a calculation of its lift under $D_a^{-1}I$}
\label{fig:I3calculation}
\end{figure}

\subsection{Twisting the generalized \boldmath$z^2+i$ polynomials}
\label{sec:genz2+i}

In this section we define an infinite family of polynomials $I_n(z)$ with $I_3=I$ and solve a generalization of the twisted $z^2+i$ problem from the previous section.  We give a unified argument for all numbers of marked points $n \geq 4$, just as we did for the rabbit polynomial. As in that problem, the answer here for $n \geq 4$ marked points differs in character from the answer for $n=3$.

\p{A set of generalized $z^2+i$ maps} For $n \geq 3$, we define the polynomial $I_n$ to be the polynomial $z^2+c$, where $c$ is the landing point of the external ray of the Mandelbrot set of angle $\tfrac{1}{3}\left(\tfrac{1}{2^{n-2}}\right)$ (see, e.g.,  Douady--Hubbard \cite{DH2}). Each polynomial $I_n$ is a quadratic polynomial such that 0 is a preperiodic critical point, there are $n$ post-critical points, and the last two post-critical points form a 2-cycle.  The Hubbard tree for $I_n$, as well as its preimage under $I_n$, is indicated in Figure~\ref{fig:In}(b).  As in Section~\ref{sec:z2+i}, we may regard the pair of trees in Figure~\ref{fig:In}(b) as a combinatorial description of $I_n$.  

\p{Statement of the problem} For each $n$ we define $n$ curves $d_i$ in $(\R^2,P_{I_n})$ as shown in Figure~\ref{fig:In}(c) for the case $n=5$. For each $n$, set $a=a_n=d_0$, $b=b_n=d_1$, and $c=c_n=d_{n-1}$. 

The twisted generalized $z^2+i$ problem is:
\begin{quote}
\emph{Let $m \in \Z$. To which topological polynomial is $D_{b}^mI_n$ Thurston equivalent?}
\end{quote}

\p{The answer} In the answer to the generalized twisted $z^2+i$ problem, there is only one polynomial, namely $I_n$, and one obstructed topological polynomial, $D_c^{-1} I_n$.  We give the canonical form for $D_c^{-1} I_n$ below.

\begin{theorem}
\label{thm:I}
Let $n \geq 4$ and let $m \in \Z$. Then
\[
D_b^m I_n \simeq
\begin{cases}
I_n  & m \text{ even}\\
D_c^{-1} I_n    & m \text{ odd}.\\
\end{cases}
\]
\end{theorem}

\begin{figure}
$\underset{\textstyle\text{(a)}}{\raisebox{-0.47\height}{\includegraphics{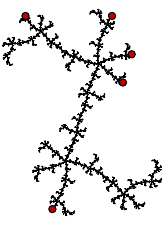}}}$
\qquad\hfill
$\underset{\textstyle\text{(b)}}{
\raisebox{-0.47\height}{\includegraphics{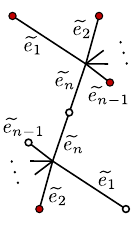}}
\;\;\rightarrow\;\;
\raisebox{-0.47\height}{\includegraphics{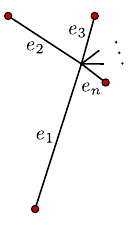}}
}$
\hfill
$\underset{\textstyle\text{(c)}}{\raisebox{-0.47\height}{\includegraphics{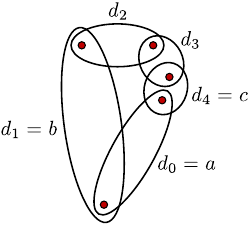}}}$
\caption{(a) The Julia set for~$I_5$. (b)~The topological Hubbard tree for $I_n$ and its preimage. (c)~The curves $d_i$ in $(\mathbb{R}^2,P_{I_n})$, including $a$, $b$, and~$c$}
\label{fig:In}
\end{figure}

\p{Reduction formulas} The reduction formulas for the generalized $z^2+i$ problem are:
\[
D_b^{m} I_n\simeq 
\begin{cases}
I_n    & m=2k\\
D_c^{-1} I_n    & m=2k+1.\\
\end{cases}
\]

To prove the reduction formulas, we again need only a few facts: $D^2_{a}~\leadsto~\textrm{id}$, $D^2_{c}~\leadsto~\textrm{id}$, and $D_{d_i}~\leadsto~D_{d_{i+1}}$ for all other $d_i$. One difference from the case of $I_3=I$ is that for $n=3$ we have $D^2_c \leadsto D_a$. 

\bigskip

\noindent \emph{Case 1: $m=2k$.}  $D_b^{2k} \leadsto \cdots \leadsto D_c^{2k}  \leadsto
\textrm{id}.$

\medskip

\noindent\emph{Case 2: $m=2k+1$.} $D_b^{2k+1} \leadsto \cdots \leadsto D_c^{2k+1}  \stackrel{D_c^{-1}}\leadsto
D_c^{-1}.$

\begin{figure}
\centering
\includegraphics{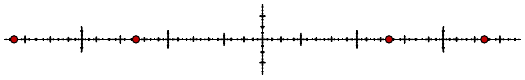} \\[16pt]
\raisebox{-0.4\height}{\includegraphics{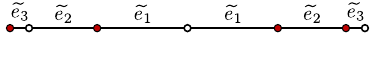}}$\qquad\rightarrow\qquad$
\raisebox{-0.4\height}{\includegraphics{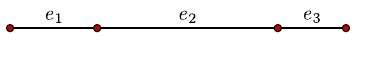}}
\caption{The Julia set for $J_4$, as well as the Hubbard tree for $J_4$ and its preimage}
\label{fig:Jn-1}
\end{figure}

\p{A sequence of polynomials} We now describe a sequence of polynomials $J_{n-1}$ that will appear as the exterior maps for the obstructed maps $D_c^{-1} I_n$ that appear in our answer to the generalized twisted $z^2+i$ problem.

The polynomial $J_{n-1}$ is the real quadratic polynomial with kneading sequence $10\cdots0\bar{1}$ where the symbol 0 is listed $n-3$ times when $n\geq 3$.  Its critical point is pre-periodic with pre-period $n-2$ and period 1.  The Hubbard tree for $J_{n-1}$ is a path consisting of $n-2$ edges.  If, as in Figure~\ref{fig:Jn-1}, the edges of $H_{J_{n-1}}$ are labeled $e_1,\dots,e_{n-1}$, then
\begin{align*}
(J_{n-1})_*(e_i) &= 
\begin{cases}
e_{n-2}e_{n-3} \cdots e_2 & i=1 \\
e_{i-1}& 2\leq i\leq n-2
\end{cases}
\end{align*}
The Julia set, Hubbard tree, and the full preimage of the Hubbard tree for $J_4$ are depicted in Figure~\ref{fig:Jn-1}.  For $n=3$, the kneading sequence for $J_{n-1}=J_2$ degenerates to $\bar{1}$, the kneading sequence for the polynomial $z^2-2$ that we saw in the solution to the twisted $z^2+i$ problem. 

\begin{figure}
\includegraphics{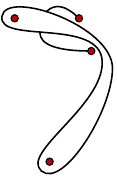}\qquad\qquad\qquad
\includegraphics{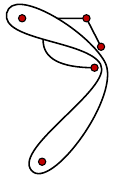}\qquad\qquad\qquad
\includegraphics{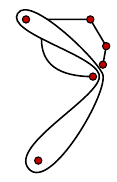}
\caption{The Hubbard bubble trees $D_c^{-1} I_n$, for $n=4,5,6$}
\label{fig:bubbleGenI}
\end{figure}

\p{Base cases} Unlike in the previous sections, the reduction formulas constitute a complete proof of Theorem~\ref{thm:I}.  To give a more satisfying resolution to the generalized twisted \mbox{$z^2+i$} problem, we further describe the canonical form of the obstructed map $D_c^{-1} I_n$, as per Section~\ref{sec:canonical form}.

Figure~\ref{fig:bubbleGenI} depicts the Hubbard bubble tree for $D_c^{-1} I_n$.  For each $D_c^{-1} I_n$, the single curve $D_a(b)$ is the canonical Levy cycle. 

For each $D_c^{-1}I_n$, we may use the Alexander method to show that the exterior map is Thurston equivalent to the polynomial $J_{n-1}$ described above.  

The interior map for each $D_c^{-1}I_n$ is a mapping class on the closed disk bounded by the Levy cycle $D_a(b)$ (this disk has two marked points on the boundary, in addition to the two in the interior).  Using the Alexander method we find that the interior map is the (unique) rotation of order 2. 

\p{More generalizations} Each curve $d_i$ except for $a=d_0$ lifts to $c$ under iteration of the lifting map for $I_n$. This means that (similar to Section~\ref{sec:manyear}), we obtain $D_{d_i}^m I_n \simeq D_c^m I_n$ for all $i\neq 0$ and $m \in \Z$.  Thus our solution to the twisted generalized $z^2+i$ problem of the form $D_b^mI_n$ also immediately gives solutions to all twisted generalized $z^2+i$ problems of the form $D_{c_i}^mI_n$ with $i \in \{2,\dots,n-1\}$.  

\bibliographystyle{plain}
\bibliography{LiftingAlgorithm}

\end{document}